\documentclass[oneside,english,a4paper]{amsart}
\usepackage[T1]{fontenc}
\usepackage[latin9]{inputenc}
\usepackage{verbatim}
\usepackage{mathrsfs}
\usepackage{amsthm}
\usepackage{amstext}
\usepackage{amssymb}

\makeatletter

\providecommand{\tabularnewline}{\\}

\numberwithin{equation}{section}
\numberwithin{figure}{section}
\numberwithin{table}{section}
  \theoremstyle{plain}
  \newtheorem*{thm*}{\protect\theoremname}
\theoremstyle{plain}
\newtheorem{thm}{\protect\theoremname}[section]
  \theoremstyle{definition}
  \newtheorem{defn}[thm]{\protect\definitionname}
  \theoremstyle{plain}
  \newtheorem{prop}[thm]{\protect\propositionname}
  \theoremstyle{remark}
  \newtheorem{notation}[thm]{\protect\notationname}
  \theoremstyle{remark}
  \newtheorem{rem}[thm]{\protect\remarkname}
  \theoremstyle{definition}
  \newtheorem{example}[thm]{\protect\examplename}
  \theoremstyle{plain}
  \newtheorem{lem}[thm]{\protect\lemmaname}
  \theoremstyle{plain}
  \newtheorem{cor}[thm]{\protect\corollaryname}
  \theoremstyle{definition}
  \newtheorem{problem}[thm]{\protect\problemname}

\usepackage{extarrow}
\usepackage{enumerate}
\usepackage{fontenc}
\usepackage{graphicx}
\usepackage[colorlinks=true, linkcolor=blue]{hyperref}
\usepackage{latexsym}
\usepackage{mathrsfs}
\usepackage{mathtext}
\usepackage{tikz}
\usepackage{textcomp}
\usepackage[colorinlistoftodos]{todonotes}
\usepackage{upgreek}
\usepackage{xy}

\usetikzlibrary{arrows}
\usetikzlibrary{matrix}
\usetikzlibrary{shapes}
\usetikzlibrary{snakes}
\usetikzlibrary{matrix}

\DeclareMathOperator{\Aff}{\textup{Aff}}
\DeclareMathOperator{\Alg}{\textup{Alg}}
\DeclareMathOperator{\Ann}{\textup{Ann}}
\DeclareMathOperator{\Arr}{\textup{Arr}}
\DeclareMathOperator{\Art}{\textup{Art}}
\DeclareMathOperator{\Ass}{\textup{Ass}}
\DeclareMathOperator{\Aut}{\textup{Aut}}
\DeclareMathOperator{\Autsh}{\underline{\textup{Aut}}}
\DeclareMathOperator{\Bi}{\textup{B}}
\DeclareMathOperator{\CaCl}{\textup{CaCl}}
\DeclareMathOperator{\Cart}{\textup{Cart}}
\DeclareMathOperator{\Cl}{\textup{Cl}}
\DeclareMathOperator{\Coh}{\textup{Coh}}
\DeclareMathOperator{\Coker}{\textup{Coker}}
\DeclareMathOperator{\Cov}{\textup{Cov}}
\DeclareMathOperator{\Der}{\textup{Der}}
\DeclareMathOperator{\Div}{\textup{Div}}
\DeclareMathOperator{\End}{\textup{End}}
\DeclareMathOperator{\Endsh}{\underline{\textup{End}}}
\DeclareMathOperator{\Ext}{\textup{Ext}}
\DeclareMathOperator{\Extsh}{\underline{\textup{Ext}}}
\DeclareMathOperator{\FGrad}{\textup{FGrad}}
\DeclareMathOperator{\FMod}{\textup{FMod}}
\DeclareMathOperator{\FVect}{\textup{FVect}}
\DeclareMathOperator{\Fibr}{\textup{Fibr}}
\DeclareMathOperator{\Fix}{\textup{Fix}}
\DeclareMathOperator{\Funct}{\textup{Funct}}
\DeclareMathOperator{\GL}{\textup{GL}}
\DeclareMathOperator{\GRis}{\textup{GRis}}
\DeclareMathOperator{\GRiv}{\textup{GRiv}}
\DeclareMathOperator{\Gal}{\textup{Gal}}
\DeclareMathOperator{\Gl}{\textup{Gl}}
\DeclareMathOperator{\Grad}{\textup{Grad}}
\DeclareMathOperator{\Hilb}{\textup{Hilb}}
\DeclareMathOperator{\Hl}{\textup{H}}
\DeclareMathOperator{\Hom}{\textup{Hom}}
\DeclareMathOperator{\Homsh}{\underline{\textup{Hom}}}
\DeclareMathOperator{\ISym}{\textup{Sym}^*}
\DeclareMathOperator{\Imm}{\textup{Im}}

\DeclareMathOperator{\Iso}{\textup{Iso}}
\DeclareMathOperator{\Isosh}{\underline{\textup{Iso}}}
\DeclareMathOperator{\Ker}{\textup{Ker}}
\DeclareMathOperator{\Left}{\textup{L}}

\DeclareMathOperator{\MHilb}{\textup{M-Hilb}}
\DeclareMathOperator{\Map}{\textup{Map}}
\DeclareMathOperator{\Mod}{\textup{Mod}}
\DeclareMathOperator{\Ob}{\textup{Ob}}
\DeclareMathOperator{\Obj}{\textup{Obj}}
\DeclareMathOperator{\PDiv}{\textup{PDiv}}
\DeclareMathOperator{\PGL}{\textup{PGL}}
\DeclareMathOperator{\Pic}{\textup{Pic}}
\DeclareMathOperator{\Picsh}{\underline{\textup{Pic}}}
\DeclareMathOperator{\Pro}{\textup{Pro}}
\DeclareMathOperator{\Proj}{\textup{Proj}}
\DeclareMathOperator{\QCoh}{\textup{QCoh}}
\DeclareMathOperator{\R}{\textup{R}}
\DeclareMathOperator{\Riv}{\textup{Riv}}
\DeclareMathOperator{\SFibr}{\textup{SFibr}}
\DeclareMathOperator{\SchI}{\textup{SchI}}
\DeclareMathOperator{\Sh}{\textup{Sh}}
\DeclareMathOperator{\Soc}{\textup{Soc}}
\DeclareMathOperator{\Spec}{\textup{Spec}}
\DeclareMathOperator{\Specsh}{\underline{\textup{Spec}}}
\DeclareMathOperator{\Stab}{\textup{Stab}}
\DeclareMathOperator{\Supp}{\textup{Supp}}
\DeclareMathOperator{\Sym}{\textup{Sym}}
\DeclareMathOperator{\TMod}{\textup{TMod}}
\DeclareMathOperator{\Top}{\textup{Top}}
\DeclareMathOperator{\Tor}{\textup{Tor}}
\DeclareMathOperator{\Vect}{\textup{Vect}}
\DeclareMathOperator{\alt}{\textup{ht}}
\DeclareMathOperator{\car}{\textup{char}}
\DeclareMathOperator{\codim}{\textup{codim}}
\DeclareMathOperator{\degtr}{\textup{degtr}}
\DeclareMathOperator{\depth}{\textup{depth}}
\DeclareMathOperator{\divis}{\textup{div}}
\DeclareMathOperator{\et}{\textup{et}}
\DeclareMathOperator{\h}{\textup{h}}
\DeclareMathOperator{\ilim}{\displaystyle{\lim_{\longrightarrow}}}
\DeclareMathOperator{\indim}{\textup{inj dim}}
\DeclareMathOperator{\lf}{\textup{LF}}
\DeclareMathOperator{\ord}{\textup{ord}}
\DeclareMathOperator{\pd}{\textup{pd}}
\DeclareMathOperator{\plim}{\displaystyle{\lim_{\longleftarrow}}}
\DeclareMathOperator{\pr}{\textup{pr}}
\DeclareMathOperator{\pt}{\textup{pt}}
\DeclareMathOperator{\rk}{\textup{rk}}
\DeclareMathOperator{\set}{\textup{set}}
\DeclareMathOperator{\tr}{\textup{tr}}
\DeclareMathOperator{\type}{\textup{r}}
\DeclareMathOperator*{\colim}{\textup{colim}}

\makeatother

\usepackage{babel}
  \providecommand{\corollaryname}{Corollary}
  \providecommand{\definitionname}{Definition}
  \providecommand{\examplename}{Example}
  \providecommand{\lemmaname}{Lemma}
  \providecommand{\notationname}{Notation}
  \providecommand{\problemname}{Problem}
  \providecommand{\propositionname}{Proposition}
  \providecommand{\remarkname}{Remark}
  \providecommand{\theoremname}{Theorem}
\providecommand{\theoremname}{Theorem}

\begin{document}

\title{Stacks of Ramified Abelian Covers}

\author{Fabio Tonini}

\address{Scuola Normale Superiore, Piazza dei Cavalieri 7, 56126 Pisa, Italy}

\email{fabio.tonini@sns.it}

\maketitle

\global\long\def\A{\mathbb{A}}

\global\long\def\Ab{(\textup{Ab})}

\global\long\def\Cat{(\textup{cat})}

\global\long\def\Di#1{\textup{D}(#1)}

\global\long\def\E{\mathcal{E}}

\global\long\def\F{\mathbb{F}}

\global\long\def\GCov{G\textup{-Cov}}

\global\long\def\Gcat{(\textup{Galois cat})}

\global\long\def\Gfsets#1{#1\textup{-fsets}}

\global\long\def\Gm{\mathbb{G}_{m}}

\global\long\def\GrCov#1{\textup{D}(#1)\textup{-Cov}}

\global\long\def\Grp{(\textup{Grp})}

\global\long\def\Gsets#1{(#1\textup{-sets})}

\global\long\def\HCov{H\textup{-Cov}}

\global\long\def\MCov{\textup{D}(M)\textup{-Cov}}

\global\long\def\N{\mathbb{N}}

\global\long\def\PGor{\textup{PGor}}

\global\long\def\PGrp{(\textup{Profinite Grp})}

\global\long\def\PP{\mathbb{P}}

\global\long\def\Pj{\mathbb{P}}

\global\long\def\Q{\mathbb{Q}}

\global\long\def\RR{\mathbb{R}}

\global\long\def\Sch{\textup{Sch}}

\global\long\def\WW{\textup{W}}

\global\long\def\Z{\mathbb{Z}}

\global\long\def\alA{\mathscr{A}}

\global\long\def\alB{\mathscr{B}}

\global\long\def\arr{\longrightarrow}

\global\long\def\arrdi#1{\xlongrightarrow{#1}}

\global\long\def\catC{\mathscr{C}}

\global\long\def\catD{\mathscr{D}}

\global\long\def\catF{\mathscr{F}}

\global\long\def\catG{\mathscr{G}}

\global\long\def\comma{,\ }

\global\long\def\covU{\mathcal{U}}

\global\long\def\covV{\mathcal{V}}

\global\long\def\covW{\mathcal{W}}

\global\long\def\duale#1{{#1}^{\vee}}

\global\long\def\fasc#1{\widetilde{#1}}

\global\long\def\fsets{(\textup{f-sets})}

\global\long\def\iL{r\mathscr{L}}

\global\long\def\id{\textup{id}}

\global\long\def\odi#1{\mathcal{O}_{#1}}

\global\long\def\sets{(\textup{sets})}

\global\long\def\shA{\mathcal{A}}

\global\long\def\shB{\mathcal{B}}

\global\long\def\shC{\mathcal{C}}

\global\long\def\shD{\mathcal{D}}

\global\long\def\shE{\mathcal{E}}

\global\long\def\shF{\mathcal{F}}

\global\long\def\shG{\mathcal{G}}

\global\long\def\shH{\mathcal{H}}

\global\long\def\shI{\mathcal{I}}

\global\long\def\shJ{\mathcal{J}}

\global\long\def\shK{\mathcal{K}}

\global\long\def\shL{\mathcal{L}}

\global\long\def\shM{\mathcal{M}}

\global\long\def\shN{\mathcal{N}}

\global\long\def\shO{\mathcal{O}}

\global\long\def\shP{\mathcal{P}}

\global\long\def\shQ{\mathcal{Q}}

\global\long\def\shR{\mathcal{R}}

\global\long\def\shS{\mathcal{S}}

\global\long\def\shT{\mathcal{T}}

\global\long\def\shU{\mathcal{U}}

\global\long\def\shV{\mathcal{V}}

\global\long\def\shW{\mathcal{W}}

\global\long\def\shX{\mathcal{X}}

\global\long\def\shY{\mathcal{Y}}

\global\long\def\shZ{\mathcal{Z}}

\global\long\def\st{\ | \ }

\global\long\def\stB{\mathcal{B}}

\global\long\def\stF{\mathcal{F}}

\global\long\def\stG{\mathcal{G}}

\global\long\def\stX{\mathcal{X}}

\global\long\def\stY{\mathcal{Y}}

\global\long\def\stZ{\mathcal{Z}}

\global\long\def\then{\ \Longrightarrow\ }

\begin{abstract}
Given a flat, finite group scheme $G$ finitely presented over a base
scheme we introduce the notion of ramified Galois cover of group $G$
(or simply $G$-cover), which generalizes the notion of $G$-torsor.
We study the stack of $G$-covers, denoted with $\GCov$, mainly in
the abelian case, precisely when $G$ is a finite diagonalizable group
scheme over $\Z$. In this case we prove that $\GCov$ is connected,
but it is irreducible or smooth only in few finitely many cases. On
the other hand, it contains a 'special' irreducible component $\stZ_{G}$,
which is the closure of $\Bi G$ and this reflects the deep connection
we establish between $\GCov$ and the equivariant Hilbert schemes.
We introduce 'parametrization' maps from smooth stacks, whose objects
are collections of invertible sheaves with additional data, to $\stZ_{G}$
and we establish sufficient conditions for a $G$-cover in order to
be obtained (uniquely) through those constructions. Moreover a toric
description of the smooth locus of $\stZ_{G}$ is provided.
\end{abstract}
\maketitle

\section*{Introduction}

In this paper we study $G$-Galois covers of very general schemes.
Let $G$ be a flat, finite group scheme finitely presented over a
base scheme (say over a field, or, as in this paper, over $\Z)$.
We define a (ramified) $G$\emph{-cover} as a finite morphism $f\colon X\arr Y$
with an action of $G$ on $X$ such that $f$ is $G$-equivariant
and $f_{*}\odi X$ is fppf-locally isomorphic to the regular representation
$\odi Y[G]$ as $\odi Y[G]$-comodule. This definition is somehow
the most natural: it generalizes the notion of $G$-torsors and, under
suitable hypothesis, coincides with the usual definition of Galois
cover when the group $G$ is constant (see for example \cite{Pardini1991,Alexeev2011,Easton2008}).
Moreover, as explained below, in the abelian case $G$-covers are
strictly related to the theory of equivariant Hilbert schemes (see
for example \cite{Nakamura2001,Stillman2002,Haiman2002,Alexeev2003}).

We call $\GCov$ the stack of $G$-covers and the aim of this article
will be to describe its structure, especially in the abelian case.
Our first result is:
\begin{thm*}
{[}\ref{pro:BG open in GCov},\ref{cor:GCov is algebraic}{]} The
stack $\GCov$ is algebraic and finitely presented over $S$. Moreover
$\Bi G$ is an open substack of $\GCov$.
\end{thm*}
In many concrete problems, one is interested in a more direct and
concrete description of a $G$-cover $f\colon X\arr Y$. This is very
simple and well known when $G=\mu_{2}$: such a cover $f$ is given
by an invertible sheaf $\shL$ on $Y$ with a section of $\shL^{\otimes2}$.
Similarly, when $G=\mu_{3}$, a $\mu_{3}$-cover $f$ is given by
a pair $(\shL_{1},\shL_{2})$ of invertible sheaves on $Y$ with maps
$\shL_{1}^{\otimes2}\arr\shL_{2}$ and $\shL_{2}^{\otimes2}\arr\shL_{1}$
(see \cite[§ 6]{Arsie2004}).

In general, however, there is no comparable description of $G$-covers.
Very little is known when $G$ is not abelian, beyond the cases $G=S_{d}$
with $d=3,4,5$: see \cite{Easton2008} for the case $G=S_{3}$ and
\cite{Miranda1985,Hahn1999,Casnati1996,Casnati1996a} for the non-Galois
case (of course, ramified covers of degree $d$ are strictly linked
with ramified $S_{d}$-covers).

Even in the abelian case, the situation become complicated very quickly
when the order of $G$ grows. The paper that inspires our work is
\cite{Pardini1991}; here the author describes $G$-covers $X\arr Y$
when $G$ is an abelian group, $Y$ is a smooth variety over an algebraically
closed field of characteristic prime to $|G|$ and $X$ is normal,
in terms of certain invertible sheaves on $Y$, generalizing the description
given above for $G=\mu_{2}$ and $G=\mu_{3}$.

Here we concentrate on the case when $G$ is a finite diagonalizable
group scheme over $\Spec\Z$; thus, $G$ is isomorphic to a finite
direct product of group scheme of the form $\mu_{d,\Z}$ for $d\geq1$.
We consider the dual finite abelian group $M=\Hom(G,\Gm)$ so that,
by standard duality results (see \cite{SGA1}), $G$ is the fppf sheaf
of homomorphisms $M\arr\Gm$ and a decomposition of $M$ into a product
of cyclic groups yields the decomposition of $G$ into a product of
$\mu_{d}$'s.

In this case we have an explicit description of a $G$-cover in terms
of sequences of invertible sheaves. Indeed a $G$-cover over $Y$
is of the form $X=\Spec\alA$ where $\alA$ is a coherent sheaf of
algebras over $Y$ with a decomposition
\[
\alA=\bigoplus_{m\in M}\shL_{m}\text{ s.t. }\shL_{0}=\odi Y\comma\shL_{m}\text{ invertible}\text{ and }\shL_{m}\shL_{n}\subseteq\shL_{m+n}\text{ for any }m,n\in M
\]
So a $G$-cover corresponds to a sequence of invertible sheaves $(\shL_{m})_{m\in M}$
with maps $\psi_{m,n}\colon\shL_{m}\otimes\shL_{n}\arr\shL_{m+n}$
satisfying certain rules and our principal aim will be to simplify
the data necessary to describe such covers. For instance $G$-torsors
correspond to sequences where all the maps $\psi_{m,n}$ are isomorphism.
Therefore, if $G=\mu_{l}$, a $G$-torsor is simply given by an invertible
sheaf $\shL=\shL_{1}$ and an isomorphism $\shL^{\otimes l}\simeq\odi{}$. 

When $G=\mu_{2}$ or $G=\mu_{3}$ the description given above shows
that the stack $\GCov$ is smooth, irreducible and very easy to describe.
In the general case its structure turns out to be extremely intricate.
For instance, as we will see, $\GCov$ is almost never irreducible,
but has a 'special' irreducible component, called $\stZ_{G}$, which
is the scheme-theoretically closure of $\Bi G$. This parallels what
happens in the theory of $M$-equivariant Hilbert schemes (see \cite[Remark 5.1]{Haiman2002}).
It turns out that this theory and the one of $G$-covers is deeply
connected: given an action of $G$ on $\A^{r}$, induced by elements
$\underline{m}=m_{1},\dots,m_{r}\in M$, the equivariant Hilbert scheme
$M\text{-Hilb }\A^{r}$, that we will denote by $\MHilb^{\underline{m}}$
to underline the dependency on the sequence $\underline{m}$, can
be viewed as the functor whose objects are $G$-covers whit an equivariant
closed immersion in $\A^{r}$. The forgetful map $\vartheta\colon\MHilb^{\underline{m}}\arr\GCov$
is smooth and an atlas provided that $\underline{m}$ contains all
the elements in $M-\{0\}$ (\ref{pro:MHlb --> MCov has irreducible fibers}).
Moreover $\vartheta^{-1}(\stZ_{G})$ coincides with the main component
of $\MHilb^{\underline{m}}$, first studied by Nakamura in \cite{Nakamura2001}.

We will prove the following results on the structure of $\GCov$.
\begin{thm*}
{[}\ref{thm:Mcov geom connected},\ref{cor:Mcov reducible},\ref{pro:smooth DMCov},\ref{rem:MCov for M=00003DZfour}{]}
$\GCov$ is
\begin{itemize}
\item flat and of finite type over $\Z$ with geometrically connected fibers,
\item smooth if and only if $G\simeq\mu_{2},\mu_{3},\mu_{2}\times\mu_{2}$,
\item normal if $G\simeq\mu_{4}$,
\item reducible if $|G|\geq8$ and $G\not\simeq(\mu_{2})^{3}$.
\end{itemize}
The above properties continue to hold if we replace $\GCov$ with
$\MHilb^{\underline{m}}$ if $M-\{0\}\subseteq\underline{m}$.
\end{thm*}
We don't know whether $\GCov$ is integral for $G\simeq\mu_{5},\mu_{6},\mu_{7},(\mu_{2})^{3}$.
So $\GCov$ is usually reducible, its structure is extremely complicated
and we have little hope of getting to a real understanding of the
components that don't contain $\Bi G$. Therefore we will focus on
the main irreducible component $\stZ_{G}$ of $\GCov$. The main idea
behind this paper, inspired by the results in \cite{Pardini1991},
is to try to decompose the multiplications $\psi_{m,n}\in\shL_{m+n}\otimes\shL_{m}^{-1}\otimes\shL_{n}^{-1}$
as a tensor product of sections of other invertible sheaves. Following
this idea we will construct parametrization maps $\pi_{\underline{\E}}\colon\stF_{\underline{\E}}\arr\stZ_{G}\subseteq\GCov$,
where $\stF_{\underline{\E}}$ are 'nice' stacks, for example smooth
and irreducible, whose objects are those decompositions. This construction
can be better understood locally, where a $G$-cover over $Y=\Spec R$
is just $X=\Spec A$, where $A$ is an $R$-algebra with an $R$-basis
$\{v_{m}\}_{m\in M}$, $v_{0}=1$ $(\shL_{m}=\odi Yv_{m})$, so that
the multiplications are elements $\psi_{m,n}\in R$ such that $v_{m}v_{n}=\psi_{m,n}v_{m+n}$.

Consider $a\in R$, a collection of natural numbers $\E=(\E_{m,n})_{m,n\in N}$
and set $\psi_{m,n}=a^{\E_{m,n}}$. The condition that the product
structure on $A=\oplus_{m}Rv_{m}$ defined by the $\psi_{m,n}$ yields
an associative, commutative $R$-algebra, i.e. makes $\Spec A$ into
a $G$-cover over $\Spec R$, translates in some additive relations
on the numbers $\E_{m,n}$. Call $\duale K_{+}$ the set of such collections
$\E$. More generally given $\underline{\E}=\E^{1},\dots,\E^{r}\in\duale K_{+}$
we can define a parametrization
\[
R^{r}\ni(a_{1},\dots,a_{r})\arr\psi_{m,n}=a_{1}^{\E_{m,n}^{1}}\cdots a_{r}^{\E_{m,n}^{r}}
\]
This is essentially the local behavior of the map $\pi_{\underline{\E}}\colon\stF_{\underline{\E}}\arr\GCov$.
In the global case the elements $a_{i}$ will be sections of invertible
sheaves.

From this point of view the natural questions are: given a $G$-cover
over a scheme $Y$ when does there exist a lift to an object of $\stF_{\underline{\E}}(Y)$?
Is this lift unique? How can we choose the sequence $\underline{\E}$?

The key point is to give an interpretation to $\duale K_{+}$ (that
also explains this notation). Consider $\Z^{M}$ with canonical basis
$(e_{m})_{m\in M}$ and define $v_{m,n}=e_{m}+e_{n}-e_{m+n}\in\Z^{M}/<e_{0}>$.
If $p\colon\Z^{M}/<e_{0}>\arr M$ is the map $p(e_{m})=m$, the $v_{m,n}$
generate $\Ker p$. Now call $K_{+}$ the submonoid of $\Z^{M}/<e_{0}>$
generated by the $v_{m,n}$, $K=\Ker p$ its associated group and
also consider the torus $\shT=\Homsh(\Z^{M}/<e_{0}>,\Gm)$, which
acts on $\Spec\Z[K_{+}]$. By construction we have that a collection
of natural numbers $(\E_{m,n})_{m,n\in M}$ belongs to $\duale K_{+}$
if and only if the association $v_{m,n}\arr\E_{m,n}$ defines an additive
map $K_{+}\arr\N$. Therefore, as the symbol suggests, we can identify
$\duale K_{+}$ with $\Hom(K_{+},\N)$, the dual monoid of $K_{+}$.
Its elements will be called \emph{rays}. More generally a monoid map
$\psi\colon K_{+}\arr(R,\cdot)$, where $R$ is a ring, yields a multiplication
$\psi_{m,n}=\psi(v_{m,n})$ and therefore we obtain a map $\Spec\Z[K_{+}]\arr\stZ_{G}$.
We will prove that (see \ref{cor:MCov as global quotient}):
\begin{thm*}
We have $\stZ_{G}\simeq[\Spec\Z[K_{+}]/\shT]$ and $\Bi G\simeq[\Spec\Z[K]/\shT]$.
\end{thm*}
Given $\underline{\E}=\E^{1},\dots,\E^{r}\in\duale K_{+}$ we have
defined a map $\pi_{\underline{\E}}\colon\stF_{\underline{\E}}\arr\stZ_{G}$.
Notice that if $\underline{\gamma}$ is a subsequence of $\underline{\E}$
then $\stF_{\underline{\gamma}}$ is an open substack of $\stF_{\underline{\E}}$
and $(\pi_{\underline{\E}})_{|\stF_{\underline{\gamma}}}=\pi_{\underline{\gamma}}$.
The lifting problem for the maps $\pi_{\underline{\E}}$ clearly depends
on the choice of the sequence $\underline{\E}$. Considering larger
$\underline{\E}$ allows to parametrize more covers, but also makes
uniqueness of the lifting unlikely. In this direction we have proved
that:
\begin{thm*}
{[}\ref{pro:characterization of points of Zphi}{]} Let $k$ be an
algebraically closed field and suppose to have a collection $\underline{\E}$
whose rays generate the rational cone $\duale K_{+}\otimes\Q$. Then
$\stF_{\underline{\E}}(k)\arr\stZ_{G}(k)$ is essentially surjective.
In other words a $G$-cover of $\Spec k$ in the main component $\stZ_{G}$
has a multiplication of the form $\psi_{m,n}=0^{\E_{m,n}}$ for some
$\E\in\duale K_{+}$. 
\end{thm*}
On the other hand small sequences $\underline{\E}$ can guarantee
uniqueness but not existence. The solution we have found is to consider
a particular class of rays, called extremal, that have minimal non
empty support. Set $\underline{\eta}$ for the sequence of all extremal
rays (that is finite). Notice that extremal rays generate $\duale K_{+}\otimes\Q$.
We prove that:
\begin{thm*}
{[}\ref{pro:smooth locus of Z[T +]}, \ref{thm:fundamental theorem for the smooth locus of ZM}{]}
The smooth locus $\stZ_{G}^{\textup{sm}}$ of $\stZ_{G}$ is of the
form $[X_{G}/\shT]$ where $X_{G}$ is a smooth toric variety over
$\Z$ (whose maximal torus is $\Spec\Z[K]$). Moreover $\pi_{\underline{\eta}}\colon\stF_{\underline{\eta}}\arr\stZ_{G}$
induces an isomorphism of stacks
\[
\pi_{\underline{\eta}}^{-1}(\stZ_{G}^{\textup{sm}})\arrdi{\simeq}\stZ_{G}^{\textup{sm}}
\]

\end{thm*}
Among the extremal rays there are special rays, called smooth, that
can be defined as extremal rays $\E$ whose associated multiplication
$\psi_{m,n}=0^{\E_{m,n}}$ yields a cover in $\stZ_{G}^{\textup{sm}}$.
Set $\underline{\xi}$ for the sequence of smooth extremal rays. It
turns out that theorem above holds if we replace $\underline{\eta}$
with $\underline{\xi}$.

If we set $\Picsh X$ for the category of invertible sheaves on $X$
and any map we also have: 
\begin{thm*}
{[}\ref{cor:lift when Picsh is the same}{]} Consider a $2$-commutative
diagram    \[   \begin{tikzpicture}[xscale=2.0,yscale=-1.0]     \node (A0_0) at (0, 0) {$Y$};     \node (A0_1) at (1, 0) {$\stF_{\underline \E}$};     \node (A1_0) at (0, 1) {$X$};     \node (A1_1) at (1, 1) {$\GCov$};     \path (A0_0) edge [->]node [auto] {$\scriptstyle{}$} (A0_1);     \path (A1_0) edge [->,dashed]node [auto] {$\scriptstyle{}$} (A0_1);     \path (A0_0) edge [->]node [auto,swap] {$\scriptstyle{f}$} (A1_0);     \path (A0_1) edge [->]node [auto] {$\scriptstyle{\pi_{\underline \E}}$} (A1_1);     \path (A1_0) edge [->]node [auto] {$\scriptstyle{}$} (A1_1);   \end{tikzpicture}   \] where
$X,Y$ are schemes and $\underline{\E}$ is a sequence of elements
of $\duale K_{+}$. If $\Picsh X\arrdi{f^{*}}\Picsh Y$ is fully faithful
(an equivalence) the dashed lifting is unique (exists).
\end{thm*}
In particular theorems above allow to conclude that:
\begin{thm*}
{[}\ref{thm:fundamental theorem for the smooth locus of ZM}, \ref{thm:fundamental theorem for locally factorial schemes}{]}
Let $X$ be a locally noetherian and locally factorial scheme. A cover
$\chi\in\GCov(X)$ such that $\chi_{|k(p)}\in\stZ_{G}^{\textup{sm}}(k(p))$
for any $p\in X$ with $\codim_{p}X\leq1$ lifts uniquely to $\stF_{\underline{\xi}}(X)$.
\end{thm*}
An interesting problem is to describe all (smooth) extremal rays.
This seems very difficult and it is related to the problem of finding
$\Q$-linearly independent sequences among the $v_{m,n}\in K_{+}$.
A natural way of obtaining extremal rays is trying to describe $G$-covers
with special properties. The first examples of them arise looking
at covers with normal total space. Indeed in \cite{Pardini1991} the
author is able to describe the multiplications yielding regular $G$-covers
of a DVR. This description, using the language introduced above, yields
a sequence $\underline{\delta}=(\E^{\phi})_{\phi\in\Phi_{M}}$ of
smooth extremal rays, where $\Phi_{M}$ is the set of surjective maps
$M\arr\Z/d\Z$ with $d>1$. In this paper we will define a stratification
of $\GCov$ by open substacks $\Bi G=U_{0}\subseteq U_{1}\subseteq\cdots\subseteq U_{|G|-1}=\GCov$
and we will prove that there exists an explicitly given sequence $\underline{\E}$
of smooth integral extremal rays (defined in \ref{pro:classification sm int ray htwo})
containing $\underline{\delta}$ such that:
\begin{thm*}
{[}\ref{thm:fundamental theorem for hleqone}, \ref{thm:fundamental thm for hleqtwo}{]}
We have that $U_{2}\subseteq\stZ_{G}^{\textup{sm}}$ and that $\pi_{\underline{\E}}\colon\stF_{\underline{\E}}\arr\stZ_{G}$
induces isomorphisms of stacks
\[
\pi_{\underline{\E}}^{-1}(U_{2})\arrdi{\simeq}U_{2}\comma\pi_{\underline{\delta}}^{-1}(U_{1})=\pi_{\underline{\E}}^{-1}(U_{1})\arrdi{\simeq}U_{1}
\]

\end{thm*}
Theorem above implies that $\MHilb\A^{2}$ is smooth and irreducible
(\ref{cor:MHilbmn smooth and irreducible}). In this way we get an
alternative proof of the result in \cite{Maclagan2002} (later generalized
in \cite{Maclagan2010}) in the particular case of equivariant Hilbert
schemes.
\begin{thm*}
{[}\ref{thm:for hleqone}, \ref{thm:fundamental thm locally factoria hleqtwo}{]}
Let $X$ be a locally noetherian and locally factorial scheme and
$\chi\in\GCov(X)$. If $\chi_{|k(p)}\in U_{1}$ for any $p\in X$
with $\codim_{p}X\leq1$, then $\chi$ lifts uniquely to $\stF_{\underline{\delta}}(X)$.
If $\chi_{|k(p)}\in U_{2}$ for any $p\in X$ with $\codim_{p}X\leq1$,
then $\chi$ lifts uniquely to $\stF_{\underline{\E}}(X)$.
\end{thm*}
Notice that $\underline{\E}=\underline{\delta}$ if and only if $G\simeq(\mu_{2})^{l}$
or $G\simeq(\mu_{3})^{l}$ (\ref{pro:when sigmaM is empty}). Finally
we prove:
\begin{thm*}
{[}\ref{thm:regular in codimension 1 covers}, \ref{thm:NC in codimension one}{]}
Let $X$ be a locally noetherian and locally factorial integral scheme
with $\dim X\geq1$ and $(\car X,|M|)=1$ and $Y/X$ be a $G$-cover.
If $Y$ is regular in codimension $1$ it is normal and $Y/X$ comes
from a unique object of $\stF_{\underline{\delta}}(X)$. If $Y$ is
normal crossing in codimension $1$ (see \ref{def:normal crossing codimesion one})
then $Y/X$ comes from a unique object of $\stF_{\underline{\gamma}}(X)$,
where $\underline{\delta}\subseteq\underline{\gamma}\subseteq\underline{\E}$
is an explicitly given sequence.
\end{thm*}
The part concerning regular in codimension $1$ covers is essentially
a rewriting of Theorem $2.1$ and Corollary $3.1$ of \cite{Pardini1991}
extended to locally noetherian and locally factorial schemes, while
the last part generalizes Theorem $1.9$ of \cite{Alexeev2011}.

\subsection*{Table of contents.}

We now briefly summarize how this paper is divided. 

\emph{Section 1. }We will introduce the notion of $G$-covers and
prove some general facts about them, e.g. the algebraicity of $\GCov$. 

All the other sections will be dedicated to the study of $\GCov$
when $G$ is a finite diagonalizable group with dual group $M=\Hom(G,\Gm)$.

\emph{Section 2.} $\GCov$ and some of its substacks, like $\stZ_{G}$
and $\Bi G$, share a common structure, i.e. they are all of the form
$\stX_{\phi}=[\Spec\Z[T_{+}]/\shT]$, where $T_{+}$ is a finitely
generated commutative monoid whose associated group is free of finite
rank, $\shT$ is a torus over $\Z$ and $\phi\colon T_{+}\arr\Z^{r}$
is an additive map, that induces the action of $\shT$ on $\Spec\Z[T_{+}]$.
Section $2$ will be dedicated to the study of such stacks. As we
will see many facts about $\GCov$ are just applications of general
results about such stacks. For instance the existence of a special
irreducible component $\stZ_{\phi}$ of $\stX_{\phi}$ as well as
the use of $\duale T_{+}=\Hom(T_{+},\N)$ for the study of the smooth
locus of $\stZ_{\phi}$ are properties that can be stated in this
setting.

\emph{Section 3.} We will explain how $\GCov$ can be viewed as a
stack of the form $\stX_{\phi}$ and how it is related to the equivariant
Hilbert schemes. Then we will study the properties of connectedness,
irreducibility and smoothness for $\GCov$. Finally we will introduce
the stratification $U_{0}\subseteq U_{1}\subseteq\cdots\subseteq U_{|G|-1}=\GCov$
and we will characterize the locus $U_{1}$. 

\emph{Section 4. }We will study the locus $U_{2}$ and $G$-covers
whose total space is normal crossing in codimension $1$.

\subsection*{Acknowledgments.}

The first person I would like to thank is surely my advisor Angelo
Vistoli, first of all for having proposed me the problem I will discuss
and, above all, for his continuous support and encouragement. I also
acknowledge Prof. Rita Pardini, Prof. Diane Maclagan and Prof. Bernd
Sturmfels for the useful conversations we had and all the suggestions
I have received from them. A special thank goes to Tony Iarrobino,
who first suggested me the relation between $G$-covers and equivariant
Hilbert schemes. Finally I want to thank Mattia Talpo and Dajano Tossici
for the countless times we found ourselves staring at the blackboard
trying to answer some mathematical question.

\section*{Notations}

A map of schemes $f\colon X\arr Y$ will be called a \emph{cover}
if it is finite, flat and of finite presentation or, equivalently,
if it is affine and $f_{*}\odi X$ is locally free of finite rank.
If $X$ is a scheme and $p\in X$ we set $\codim_{p}X=\dim\odi{X,p}$
and we will denote with $X^{(1)}=\{p\in X\st\codim_{p}X=1\}$ the
set of codimension $1$ points of $X$. 

If $N$ is an abelian group we set $\Di N=\Homsh_{\textup{groups}}(N,\Gm)$
for the diagonalizable group associated to it, while if $f\colon G\arr S$
is an affine group scheme we set $\odi S[G]=f_{*}\odi G$.

Given an element $f=\sum_{i}a_{i}e_{i}\in\Z^{r}$ and invertible sheaves
$\shL_{1},\dots,\shL_{r}$ on a scheme we will use the notation
\[
\underline{\shL}^{f}=\bigotimes_{i}\shL_{i}^{\otimes a_{i}}\comma\ISym\underline{\shL}=\ISym(\shL_{1},\dots,\shL_{r})=\bigoplus_{g\in\Z^{r}}\underline{\shL}^{g}
\]
Notice also that, if for any $i$ we have $\shL_{i}=\odi{}$, then
there is a canonical isomorphism $\underline{\shL}^{f}\simeq\odi{}$.

Finally if $\stX$ is an algebraic stack we denote with $|\stX|$
its associated topological space.

\section{$G$-covers}

In this section we will fix a base scheme $S$ and a group scheme
$G$ over it, which is also a cover of $S$. We will denote by $\alA$
the regular representation of $G$, i.e. $\alA=\odi S[G]$ with the
$G$-comodule structure $\mu\colon\alA\arr\alA\otimes\odi S[G]$ induced
by the multiplication of $G$.

The aim of this section is to introduce the notion of a ramified Galois
cover and prove that the associated stack is algebraic.
\begin{defn}
\label{def:Gcovers}Given a scheme $T$ over $S$, a \emph{ramified
Galois cover of group $G$}, or simply a $G$\emph{-cover,} over it
is a cover $X\arrdi fT$ together with an action of $G_{T}$ on it
such that there exists an fppf coverings $\{U_{i}\arr T\}$ and isomorphisms
of $G$-comodules
\[
(f_{*}\odi X)_{|U_{i}}\simeq\alA_{|U_{i}}
\]
We will call $\GCov(T)$ the groupoid of $G$-covers over $T$.
\end{defn}
The $G$-covers form a stack $\GCov$ over $S$. Moreover any $G$-torsor
is a $G$-cover and more precisely we have:
\begin{prop}
\label{pro:BG open in GCov}$\Bi G$ is an open substack of $\GCov$.\end{prop}
\begin{proof}
Given a scheme $U$ over S and a $G$-cover $X=\Spec\alB$ over $U$,
$X$ is a $G$-torsor if and only if the map $G\times X\arr X\times X$
is an isomorphism. This map is induced by a map $\alB\otimes\alB\arrdi h\alB\otimes\odi{}[G_{U}]$
and so the locus over which $X$ is a $G$-torsor is given by the
vanishing of $\Coker h$, which is an open subset.\end{proof}
\begin{defn}
\label{def:the main component of GCov}The main component $\stZ_{G}$
of $\GCov$ is the reduced closed substack induced by the closure
of $\Bi G$ in $\GCov$.
\end{defn}
In order to prove that $\GCov$ is an algebraic stack we will present
it as a quotient stack by a smooth group scheme.
\begin{notation}
Let $S$ be a scheme and $\shF\in\QCoh S$. We define $\WW(\shF)\colon(\Sch/S)^{\textup{op}}\arr\set$
as
\[
\WW(\shF)(U\arrdi fS)=\Hl^{0}(U,f^{*}\shF)
\]
Remember that if $\shF$ is a locally free sheaf of finite rank, then
$\WW(\shF)$ is smooth, affine and finitely presented over $S$.\end{notation}
\begin{prop}
The functor   \[   \begin{tikzpicture}[xscale=4.5,yscale=-0.7]     \node (A0_0) at (0, 0) {$(\Sch/S)^{\textup{op}}$};     \node (A0_1) at (1, 0) {$\set$};     \node (A1_0) at (0, 1) {$T$};     \node (A1_1) at (1, 1) {$\{\odi{}[G_T]\text{-coalgebra structures on } \alA_T \}$};     \path (A0_0) edge [->] node [auto] {$\scriptstyle{X_G}$} (A0_1);     \path (A1_0) edge [|->,gray] node [auto] {$\scriptstyle{}$} (A1_1);   \end{tikzpicture}   \] 
is an affine scheme finitely presented over $S$.\end{prop}
\begin{proof}
Denote by $m_{G}\colon\odi{}[G]\otimes\odi{}[G]\arr\odi{}[G]$ the
multiplication map of $\odi{}[G]$. Let also $T$ be a scheme over
$S$. An element of $X_{G}(T)$ is given by maps
\[
\alA_{T}\otimes\alA_{T}\arrdi m\alA_{T}\comma\odi T\arrdi e\alA_{T}
\]
for which $\alA$ becomes a sheaf of algebras with multiplication
$m$ and identity $e(1)$ and such that $\mu$ is a homomorphism of
algebras over $\odi T$. In particular $e$ has to be an isomorphism
onto $\alA^{G}=\odi T$. Therefore we have an inclusion $X_{G}\subseteq\Homsh(\WW(\alA\otimes\alA),\WW(\alA))\times\mathbb{G}_{m}$
which is locally scheme-theoretically defined by the vanishing of
certain polynomials.\end{proof}
\begin{prop}
$\Autsh^{G}\WW(\alA)$ is a smooth group scheme finitely presented
over $S$.\end{prop}
\begin{proof}
The morphisms   \[   \begin{tikzpicture}[xscale=2.5,yscale=-0.6]     \node (A0_0) at (0, 0) {$\varepsilon \circ \phi$};     \node (A0_1) at (1, 0) {$\phi$};     \node (A1_0) at (0, 1) {$\duale{\odi{S}[G]}$};     \node (A1_1) at (1, 1) {$\Endsh^G \alA$};     \node (A2_0) at (0, 2) {$f$};     \node (A2_1) at (1, 2) {$(f \otimes \id) \circ \Delta$};     \path (A1_0) edge [->]node [auto] {$\scriptstyle{}$} (A1_1);     \path (A2_0) edge [|->,gray]node [auto] {$\scriptstyle{}$} (A2_1);     \path (A0_1) edge [|->,gray]node [auto] {$\scriptstyle{}$} (A0_0);   \end{tikzpicture}   \] 
where $\Delta$ and $\varepsilon$ are respectively the co-multiplication
and the co-unity of $\odi S[G]$, are one the inverse of the other.
In particular we obtain an isomorphism $\Endsh^{G}\WW(\alA)\simeq\WW(\duale{\odi S[G]})$,
so that $\Endsh^{G}\WW(\alA)$ and its open subscheme $\Autsh^{G}\WW(\alA)$
are smooth and finitely presented over $S$.\end{proof}
\begin{rem}
$\Autsh^{G}\WW(\alA)$ acts on $X_{G}$ in the following way. Given
a scheme $T$ over $S$, a $G$-equivariant automorphism $f\colon\alA_{T}\arr\alA_{T}$
and $(m,e)\in X_{G}(T)$ we can set $f(m,e)$ for the unique structure
of sheaf of algebras on $\alA_{T}$ such that $f\colon(\alA_{T},m,e)\arr(\alA_{T},f(m,e))$
is an isomorphism of $\odi T$-algebras.\end{rem}
\begin{prop}
\label{pro:GCov is a quotient stack}The natural map $X_{G}\arrdi{\pi}\GCov$
is an $\Autsh^{G}\WW(\alA)$-torsor, i.e. 
\[
\GCov\simeq[X_{G}/\Autsh^{G}\WW(\alA)]
\]
\end{prop}
\begin{proof}
Consider a cartesian diagram   \[   \begin{tikzpicture}[xscale=1.8,yscale=-1.2]     \node (A0_0) at (0, 0) {$P$};     \node (A0_1) at (1, 0) {$X_G$};     \node (A1_0) at (0, 1) {$U$};     \node (A1_1) at (1, 1) {$\GCov$};     \path (A0_0) edge [->] node [auto] {$\scriptstyle{}$} (A0_1);     \path (A1_0) edge [->] node [auto,swap] {$\scriptstyle{f}$} (A1_1);     \path (A0_1) edge [->] node [auto] {$\scriptstyle{\pi}$} (A1_1);     \path (A0_0) edge [->] node [auto] {$\scriptstyle{}$} (A1_0);   \end{tikzpicture}   \] 
where $U$ is a scheme and $f\colon Y\arr U$ is a $G$-cover. We
want to prove that $P$ is an $\Autsh^{G}\WW(\alA)$ torsor over $U$
and that the map $P\arr X_{G}$ is equivariant. Since $\pi$ is an
fppf epimorphism, we can assume that $f$ comes from $X_{G}$, i.e.
$f_{*}\odi Y=\alA_{U}$ with multiplication $m$ and neutral element
$e$. It is now easy to prove that   \[   \begin{tikzpicture}[xscale=2.1,yscale=-0.5]     \node (A0_0) at (0, 0) {$\Autsh^G \WW(\alA_U)$};     \node (A0_1) at (1, 0) {$P$};     \node (A1_0) at (0, 1) {$h$};     \node (A1_1) at (1, 1) {$h(m,e)$};     \path (A0_0) edge [->]node [auto] {$\scriptstyle{\simeq}$} (A0_1);     \path (A1_0) edge [|->,gray]node [auto] {$\scriptstyle{}$} (A1_1);   \end{tikzpicture}   \] 
is a bijection and that all the other claims hold.
\end{proof}
Using above propositions we can conclude that:
\begin{thm}
\label{cor:GCov is algebraic}The stack $\GCov$ is algebraic and
finitely presented over $S$.
\end{thm}

\section{\label{sec:stack Xphi}The stack $\stX_{\phi}$.}

In the following sections we will study the stack $\GCov$ when $G=\Di M$,
the diagonalizable group of a finite abelian group $M$. The structure
of this stack and of some of its substacks is in somehow special and
in this section we will provide general constructions and properties
that will be used later. Given a monoid map $T_{+}\arrdi{\phi}\Z^{r}$,
we will associate a stack $\stX_{\phi}$ whose objects are sequences
of invertible sheaves with additional data and we will study particular
'parametrization' of these objects, defined by a map of stack $\stF_{\underline{\E}}\arrdi{\pi_{\underline{\E}}}\stX_{\phi}$,
where $\stF_{\underline{\E}}$ will be a 'nice' stack.

In this section we will consider given a commutative monoid $T_{+}$
together to a monoid map $\phi\colon T_{+}\arr\Z^{r}$. 
\begin{defn}
We define the stack $\stX_{\phi}$ over $\Z$ as follows.
\begin{itemize}
\item \emph{Objects. }An object over a scheme $S$ is a pair $(\underline{\shL},a)$
where:

\begin{itemize}
\item $\underline{\shL}=\shL_{1},\dots,\shL_{r}$ are invertible sheaves
on S;
\item $T_{+}\arrdi a\ISym\underline{\shL}$ is an additive map such that
$a(t)\in\underline{\shL}^{\phi(t)}$ for any $t\in T_{+}$.
\end{itemize}
\item \emph{Arrows. }An isomorphism $(\underline{\shL},a)\arrdi{\underline{\sigma}}(\underline{\shL}',a')$
of objects over $S$ is given by a sequence $\underline{\sigma}=\sigma_{1},\dots,\sigma_{r}$
of isomorphisms $\sigma_{i}\colon\shL_{i}\arrdi{\simeq}\shL_{i}'$
such that 
\[
\underline{\sigma}^{\phi(t)}(a(t))=a'(t)\text{ for any }t\in T_{+}
\]

\end{itemize}
\end{defn}
\begin{example}
Let $f_{1},\dots,f_{s},g_{1},\dots,g_{t}\in\Z^{r}$ and consider the
stack $\stX_{\underline{f},\underline{g}}$ of invertible sheaves
$\shL_{1},\dots,\shL_{r}$ with maps $\odi{}\arr\underline{\shL}^{f_{i}}\quad\text{and}\quad\odi{}\arrdi{\simeq}\underline{\shL}^{g_{j}}$.
If $T_{+}=\N^{s}\times\Z^{t}$ and $\phi\colon T_{+}\arr\Z^{r}$ is
the map given by the matrix $(f_{1}|\cdots|f_{s}|g_{1}|\cdots|g_{t})$
then $\stX_{\underline{f},\underline{g}}=\stX_{\phi}$.\end{example}
\begin{notation}
We set
\[
\Z[T_{+}]=\Z[x_{t}]_{t\in T_{+}}/(x_{t}x_{t'}-x_{t+t'},x_{0}-1)
\]
 and $\odi S[T_{+}]=\Z[T_{+}]\otimes_{\Z}\odi S$. The scheme $\Spec\odi S[T_{+}]$
over $S$ represents the functor that associates to any scheme $U/S$
the set of additive map $T_{+}\arr(\odi U,*)$. $\Di{\Z^{r}}$ acts
on $\Spec\Z[T_{+}]$ by the graduation $\deg t=\phi(t)$.\end{notation}
\begin{prop}
\label{pro:atlas for the stack associated to a monoid map}Set $X=\Spec\Z[T_{+}]$.
The choice $\shL_{i}=\odi X$ and   \[   \begin{tikzpicture}[xscale=1.8,yscale=-0.6]     \node (A0_0) at (0, 0) {$\underline{\shL}^{\phi(t)}$};     \node (A0_1) at (1, 0) {$\odi{X}$};     \node (A1_0) at (0, 1) {$a(t)$};     \node (A1_1) at (1, 1) {$x_t$};     \path (A0_0) edge [->] node [auto] {$\scriptstyle{\simeq}$} (A0_1);     \path (A1_0) edge [<->,gray] node [auto] {$\scriptstyle{}$} (A1_1);   \end{tikzpicture}   \] induces
a smooth epimorphism $X\arr\stX_{\phi}$ such that $\stX_{\phi}\simeq[X/\Di{\Z^{r}}]$.
In particular $\stX_{\phi}$ is an algebraic stack of finite type
over $\Z$.\end{prop}
\begin{proof}
It's enough to note that an object of $[X/\Di{\Z^{r}}](U)$ is given
by invertible sheaves $\shL_{1},\dots,\shL_{r}$ with a $\Di{\Z^{r}}$-equivariant
map $\Spec\ISym\underline{\shL}\arr\Spec\Z[T_{+}]$ which exactly
corresponds to an additive map $T_{+}\arr\ISym\underline{\shL}$ as
in the definition of $\stX_{\phi}$. It's easy to check that the map
$X\arr[X/\Di{\Z^{r}}]\arr\stX_{\phi}$ is the same defined in the
statement.\end{proof}
\begin{rem}
\label{rem:description of isomorphism for local objects of X phi}Given
a map $U\arrdi aX=\Spec\Z[T_{+}]$, i.e. a monoid map $T_{+}\arrdi a\odi U$,
the induced object $U\arrdi aX\arr\stX_{\phi}$ is the pair $(\underline{\shL},\tilde{a})$
where $\shL_{i}=\odi U$ and for any $t\in T_{+}$   \[   \begin{tikzpicture}[xscale=1.8,yscale=-0.5]     \node (A0_0) at (0, 0) {$\odi{U}$};     \node (A0_1) at (1, 0) {$\underline{\shL}^{\phi(t)}$};     \node (A1_0) at (0, 1) {$a(t)$};     \node (A1_1) at (1, 1) {$\tilde{a}(t)$};     \path (A0_0) edge [->]node [auto] {$\scriptstyle{\simeq}$} (A0_1);     \path (A1_0) edge [|->,gray]node [auto] {$\scriptstyle{}$} (A1_1);   \end{tikzpicture}   \]
We will denote by $a$ also the object $(\underline{\shL},\tilde{a})\in\stX_{\phi}(U)$.

Given two elements $a,b\colon T_{+}\arr\odi U\in\stX_{\phi}(U)$ we
have
\[
\Iso_{U}(a,b)=\{\sigma_{1},\dots,\sigma_{r}\in\odi U^{*}\;|\;\underline{\sigma}^{\phi(t)}a(t)=b(t)\;\forall t\in T_{+}\}
\]
\end{rem}
\begin{lem}
\label{lem:morphisms of stack Xphi} Consider a commutative diagram
  \[   \begin{tikzpicture}[xscale=1.7,yscale=-1.2]     \node (A0_0) at (0, 0) {$T_+$};     \node (A0_1) at (1, 0) {$T_+'$};     \node (A1_0) at (0, 1) {$\Z^r$};     \node (A1_1) at (1, 1) {$\Z^s$};     \path (A0_0) edge [->]node [auto] {$\scriptstyle{h}$} (A0_1);     \path (A0_0) edge [->]node [auto,swap] {$\scriptstyle{\phi}$} (A1_0);     \path (A0_1) edge [->]node [auto] {$\scriptstyle{\psi}$} (A1_1);     \path (A1_0) edge [->]node [auto] {$\scriptstyle{g}$} (A1_1);   \end{tikzpicture}   \] 
where $T{}_{+},T{}_{+}'$ are commutative  monoids and $\phi\comma\psi\comma h\comma g$
are additive maps. Then we have a $2$-commutative diagram \begin{align}\label{diag:commutativity fro stack X}
\begin{tikzpicture}[xscale=3.7,yscale=-0.7]     \node (A0_0) at (0, 0) {$\Spec \Z[T_+']$};     \node (A0_1) at (1, 0) {$\Spec \Z[T_+]$};     \node (A2_0) at (0, 2) {$\stX_\psi$};     \node (A2_1) at (1, 2) {$\stX_\phi$};     \node (A3_0) at (0, 3) {$(\underline{\shL},T_+'\arrdi{a}\ISym \underline{\shL})$};     \node (A3_1) at (1, 3) {$(\underline{\shM},T_+ \arrdi{b}\ISym \underline{\shM})$};     \path (A0_0) edge [->] node [auto] {$\scriptstyle{h^*}$} (A0_1);     \path (A0_1) edge [->] node [auto] {$\scriptstyle{}$} (A2_1);     \path (A0_0) edge [->] node [auto] {$\scriptstyle{}$} (A2_0);     \path (A2_0) edge [->] node [auto] {$\scriptstyle{\Lambda}$} (A2_1);     \path (A3_0) edge [|->,gray] node [auto] {$\scriptstyle{}$} (A3_1);   \end{tikzpicture}   
\end{align} where, for $i=1,\dots,r$, $\shM_{i}=\underline{\shL}^{g(e_{i})}$
and $b$ is the unique map such that \begin{center}
  \[   \begin{tikzpicture}[xscale=1.2,yscale=-1.2]     
\node (A0_0) at (0, 0) {$T_+$};     
\node (A0_1) at (1, 0) {};     
\node (A0_2) at (2, 0) {$\ISym \underline{\shM}$};     
\node (A0_3) at (3, 0) {$\underline{\shM}^v$};     
\node (A0_4) at (4, 0) {$\underline{\shL}^{g(v)}$};     
\node (A1_0) at (0, 1) {$T_+'$};     
\node (A1_2) at (2, 1) {$\ISym \underline{\shL}$};     
\node (A1_3) at (3, 1) {$\underline{\shL}^{g(v)}$};    
\node (simeq) at (3.43, 0) {$\simeq$}; 
\path (A1_0) edge [->] node [auto,swap] {$\scriptstyle{a}$} (A1_2);     
\path (A0_3) edge [->] node [auto] {$\scriptstyle{}$} (A1_3);     
\path (A0_2) edge [->] node [auto] {$\scriptstyle{}$} (A1_2);     
\path (A0_4) edge [->] node [auto] {$\scriptstyle{\id}$} (A1_3);     
\path (A0_0) edge [->] node [auto] {$\scriptstyle{b}$} (A0_2);     
\path (A0_0) edge [->] node [auto,swap] {$\scriptstyle{h}$} (A1_0);   
\end{tikzpicture}   \] 
\par\end{center}\end{lem}
\begin{proof}
An easy computation shows that there is a canonical isomorphism $\shM^{v}\simeq\shL^{g(v)}$
and so $b(t)$ corresponds under this isomorphism to $a(h(t))\in\underline{\shL}^{\psi(h(t))}=\underline{\shL}^{g(\phi(t))}\simeq\underline{\shM}^{\phi(t)}$.
So the functor $\Lambda$ is well defined and we have only to check
the commutativity of the second diagram in the sentence. The map $\Spec\Z[T{}_{+}']\arr\Spec\Z[T{}_{+}]\arr\stX_{\phi}$
is given by trivial invertible sheaves and the additive map   \[   \begin{tikzpicture}[xscale=3.6,yscale=-0.7]     \node (A0_0) at (0.3, 0) {$T_+$};     \node (A0_1) at (1, 0) {$\Z[T_+][x_1,\dots,x_r]_{\prod_{i}x_{i}}$};     \node (A0_2) at (2, 0) {$\Z[T_+'][x_1,\dots,x_r]_{\prod_{i}x_{i}}$};     \node (A1_0) at (0.3, 1) {$t$};     \node (A1_1) at (1, 1) {$x_t x^{\phi(t)}$};     \node (A1_2) at (2, 1) {$x_{h(t)} x^{\phi(t)}$};     \path (A0_0) edge [->] node [auto] {$\scriptstyle{}$} (A0_1);     \path (A1_0) edge [|->,gray] node [auto] {$\scriptstyle{}$} (A1_1);     \path (A0_1) edge [->] node [auto] {$\scriptstyle{}$} (A0_2);     \path (A1_1) edge [|->,gray] node [auto] {$\scriptstyle{}$} (A1_2);   \end{tikzpicture}   \] 
Instead the map $\Spec\Z[T{}_{+}']\arr\stX_{\psi}\arr\stX_{\phi}$
is given by trivial invertible sheaves and the map $b$ that makes
the following diagram commutative   \[   \begin{tikzpicture}[xscale=2.0,yscale=-0.7]     \node (A0_0) at (0, 0) {$T_+$};     \node (A0_2) at (2, 0) {$\Z[T_+'][x_1,\dots,x_r]_{\prod_{i}x_{i}}$};     \node (A0_3) at (3, 0) {$x^v$};     \node (A2_0) at (0, 2) {$T_+'$};     \node (A2_2) at (2, 2) {$\Z[T_+'][y_1,\dots,y_s]_{\prod_{i}y_{i}}$};     \node (A2_3) at (3, 2) {$y^{g(v)}$};     \node (A3_0) at (0, 3) {$t$};     \node (A3_2) at (2, 3) {$x_t y^{\psi(t)}$};     \path (A0_3) edge [|->,gray] node [auto] {$\scriptstyle{}$} (A2_3);     \path (A0_0) edge [->] node [auto] {$\scriptstyle{b}$} (A0_2);     \path (A0_2) edge [->] node [auto] {$\scriptstyle{}$} (A2_2);     \path (A2_0) edge [->] node [auto] {$\scriptstyle{a}$} (A2_2);     \path (A0_0) edge [->] node [auto,swap] {$\scriptstyle{h}$} (A2_0);     \path (A3_0) edge [|->,gray] node [auto] {$\scriptstyle{}$} (A3_2);   \end{tikzpicture}   \] 
Since $x_{h(t)}x^{\phi(t)}$ is sent to $x_{h(t)}y^{g(\phi(t))}=x_{h(t)}y^{\psi(h(t))}=a(h(t))$
we find again $b(t)=x_{h(t)}x^{\phi(t)}$.\end{proof}
\begin{rem}
\label{rem: description of functors of Xphi on local objects}The
functor $\stX_{\psi}\arr\stX_{\phi}$ sends an element $a\colon T_{+}'\arr\odi U\in\stX_{\psi}(U)$
to the element $a\circ h\in\stX_{\phi}(U)$. Moreover, taking into
account the description given in \ref{rem:description of isomorphism for local objects of X phi},
if $a,b\colon T_{+}'\arr\odi U\in\stX_{\psi}(U)$ we have   \[   \begin{tikzpicture}[xscale=3.0,yscale=-0.6]     \node (A0_0) at (0, 0) {$\Iso_U(a,b)$};     \node (A0_1) at (1, 0) {$\Iso_U(a\circ h,b\circ h)$};     \node (A1_0) at (0, 1) {$\underline\sigma$};     \node (A1_1) at (1, 1) {$\underline{\sigma}^{g(e_1)},\dots,\underline{\sigma}^{g(e_r)}$};     \path (A0_0) edge [->] node [auto] {$\scriptstyle{}$} (A0_1);     \path (A1_0) edge [|->,gray] node [auto] {$\scriptstyle{}$} (A1_1);   \end{tikzpicture}   \]

\end{rem}

\subsection{The main irreducible component $\stZ_{\phi}$ of $\stX_{\phi}$.}
\begin{notation}
\label{not:notation for a monoid}A monoid will be called integral
if it satisfies the cancellative rule, i.e. 
\[
\forall a,b,c\ a+b=a+c\then b=c
\]
We will call $T$ the associated group and $T_{+}^{int}=\Imm(T_{+}\arr T)$
the associated integral monoid of $T_{+}$. This means that any monoid
map $T_{+}\arr S_{+}$, where $S_{+}$ is a group (integral monoid),
factors uniquely through $T$$(T_{+}^{int})$.

From now on $T_{+}$ will be a finitely generated monoid whose associated
group is a free $\Z$-module of finite rank. We will also call $\phi$
the induced map $T\arr\Z^{r}$, but with the convention that the stack
$\stX_{\phi}$ will always refer to the map $T_{+}\arr\Z^{r}$.\end{notation}
\begin{rem}
\label{lem:the domain monoid and group monoid associated to T +: ring}If
$D$ is a domain, then $\Spec D[T]$ is an open subscheme of $\Spec D[T_{+}]$,
while $\Spec D[T_{+}^{int}]$ is an irreducible component of it. In
particular we have\end{rem}
\begin{prop}
\label{cor:the domain monoid and group monoid associated to T +: stack}Let
$\hat{\phi}\colon T\arr\Z^{r}$ be the extension of $\phi$ and set
$\phi^{int}=\hat{\phi}_{|T_{+}^{int}}$. Then $\stB_{\phi}=\stX_{\hat{\phi}}\arr\stX_{\phi}$
is an open immersion, while $\stZ_{\phi}=\stX_{\phi^{int}}\arr\stX_{\phi}$
is a closed one. Moreover $\stZ_{\phi}$ is the reduced closed stack
associated to the closure of $\stB_{\phi}$, it is an irreducible
component of $\stX_{\phi}$ and 
\[
\stB_{\phi}\simeq[\Spec\Z[T]/\Di{\Z^{r}}]\text{ and }\stZ_{\phi}\simeq[\Spec\Z[T_{+}^{int}]/\Di{\Z^{r}}]
\]
\end{prop}
\begin{defn}
With notation above we will call respectively $\stB_{\phi}$ and $\stZ_{\phi}$
the \emph{principal open substack} and the \emph{main irreducible
component} of $\stX_{\phi}$.\end{defn}
\begin{notation}
We set
\[
\duale T_{+}=\Hom(T_{+},\N)=\{\E\in T^{*}\:|\:\E(T_{+})\subseteq\N\}
\]
and we will call its elements the \emph{integral rays }for $T_{+}$,
or simply rays. Note that $\duale T_{+}=\duale{T_{+}^{int}}$. Given
$\underline{\E}=\E^{1},\dots,\E^{s}\in\duale T_{+}$ we will denote
by $\underline{\E}$ also the induced map $T\arr\Z^{s}$. Moreover
we set
\[
\Supp\underline{\E}=\{v\in T_{+}\;|\;\exists i\;\E^{i}(v)>0\}
\]
\end{notation}
\begin{defn}
Given a sequence $\underline{\E}=\E^{1},\dots,\E^{s}\in\duale T_{+}$
set   \[   \begin{tikzpicture}[xscale=3.3,yscale=-0.6]     \node (A0_0) at (0, 0) {$\N^s \oplus T$};     \node (A0_1) at (1, 0) {$\Z^s \oplus \Z^r$};     \node (A1_0) at (0, 1) {$e_i$};     \node (A1_1) at (1, 1) {$e_i$};     \node (A2_0) at (0, 2) {$t$};     \node (A2_1) at (1, 2) {$\displaystyle(\underline\E(t),- \phi(t))$};     \path (A0_0) edge [->] node [auto] {$\scriptstyle{\sigma_{\underline \E}}$} (A0_1);     \path (A1_0) edge [|->,gray] node [auto] {$\scriptstyle{}$} (A1_1);     \path (A2_0) edge [|->,gray] node [auto] {$\scriptstyle{}$} (A2_1);   \end{tikzpicture}   \] where
$e_{1},\dots,e_{s}$ is the canonical basis of $\Z^{s}$. We will
call $\stF_{\underline{\E}}=\stX_{\sigma_{\underline{\E}}}$.\end{defn}
\begin{rem}
\label{rem:description of objects of FE}An object of $\stF_{\underline{\E}}$
over a scheme $U$ is given by a sequence $(\underline{\shL},\underline{\shM},\underline{z},\lambda)$
where:
\begin{itemize}
\item $\underline{\shL}=\shL_{1},\dots,\shL_{r}$ and $\underline{\shM}=(\shM_{\E})_{\E\in\underline{\E}}=\shM_{1},\dots,\shM_{s}$
are invertible sheaves on $U$;
\item $\underline{z}=(z_{\E})_{\E\in\underline{\E}}=z_{1},\dots,z_{s}$
are sections $z_{i}\in\shM_{i}$;
\item for any $t\in T$, $\lambda(t)=\lambda_{t}$ is an isomorphism $\underline{\shL}^{\phi(t)}\arrdi{\simeq}\underline{\shM}^{\underline{\E}(t)}$
additive in $t$.
\end{itemize}
An isomorphism $(\underline{\shL},\underline{\shM},\underline{z},\lambda)\arr(\underline{\shL}',\underline{\shM}',\underline{z}',\lambda')$
is a pair $(\underline{\omega},\underline{\tau})$ where $\underline{\omega}=\omega_{1},\dots,\omega_{r}\comma\underline{\tau}=\tau_{1},\dots,\tau_{s}$
are sequences of isomorphisms $\shL_{i}\arrdi{\omega_{i}}\shL_{i}'\comma\shM_{j}\arrdi{\tau_{j}}\shM_{j}'$
such that $\tau_{j}(z_{j})=z_{j}'$ and for any $t\in T$ we have
a commutative diagram   \[   \begin{tikzpicture}[xscale=1.9,yscale=-1.2]     \node (A0_0) at (0, 0) {$\underline{\shL}^{\phi(t)}$};     \node (A0_1) at (1, 0) {$\underline{\shM}^{\E(t)}$};     \node (A1_0) at (0, 1) {$\underline{\shL}'^{\phi(t)}$};     \node (A1_1) at (1, 1) {$\underline{\shM}'^{\E(t)}$};     \path (A0_0) edge [->]node [auto] {$\scriptstyle{\lambda_t}$} (A0_1);     \path (A1_0) edge [->]node [auto] {$\scriptstyle{\lambda_t'}$} (A1_1);     \path (A0_1) edge [->]node [auto] {$\scriptstyle{\underline{\tau}^{\phi(t)}}$} (A1_1);     \path (A0_0) edge [->]node [auto,swap] {$\scriptstyle{\underline{\omega}^{\phi(t)}}$} (A1_0);   \end{tikzpicture}   \] An
object over $U$ coming from the atlas $\Spec\Z[\N^{s}\oplus T]$
is a pair $(\underline{z},\lambda)$ where $\underline{z}=z_{1},\dots,z_{s}\in\odi U$
and $\lambda\colon T\arr\odi U^{*}$ is a group homomorphism. Given
$(\underline{z},\lambda),(\underline{z}',\lambda')\in\stF_{\underline{\E}}(U)$
we have 
\[
\Iso_{U}((\underline{z},\lambda),(\underline{z}',\lambda'))=\{(\underline{\omega},\underline{\tau})\in(\odi U^{*})^{r}\times(\odi U^{*})^{s}\st\tau_{i}z_{i}=z_{i}'\comma\underline{\tau}^{\underline{\E}(t)}\lambda(t)=\underline{\omega}^{\phi(t)}\lambda'(t)\}
\]
\end{rem}
\begin{defn}
Given a sequence $\underline{\E}=\E^{1},\dots,\E^{s}$ of elements
of $\duale T_{+}$ we define the map
\[
\pi_{\underline{\E}}\colon\stF_{\underline{\E}}\arr\stX_{\phi}
\]
 induced by the commutative diagram   \[   \begin{tikzpicture}[xscale=1.6,yscale=-1.5]     \node (A0_0) at (0, 0) {$\scriptstyle t$};     \node (A0_1) at (1, 0) {$T_+$};     \node (A0_3) at (3, 0) {$\Z^r$};     \node (A1_0) at (0, 1) {$\scriptstyle(\underline \E(t),-t)$};     \node (A1_1) at (1, 1) {$\N^s \oplus T$};     \node (A1_3) at (3, 1) {$\Z^s \oplus \Z^r$};     \path (A0_1) edge [->] node [auto] {$\scriptstyle{}$} (A1_1);     \path (A0_0) edge [|->,gray] node [auto] {$\scriptstyle{}$} (A1_0);     \path (A0_3) edge [right hook->] node [auto] {$\scriptstyle{}$} (A1_3);     \path (A0_1) edge [->] node [auto] {$\scriptstyle{\phi}$} (A0_3);     \path (A1_1) edge [->] node [auto,swap] {$\scriptstyle{\sigma_{\underline \E}}$} (A1_3);   \end{tikzpicture}   \] \end{defn}
\begin{rem}
\label{rem:description of piE}We can describe the functor $\pi_{\underline{\E}}$
explicitly. So suppose to have an object $\chi=(\underline{\shL},\underline{\shM},\underline{z},\lambda)\in\stF_{\underline{\E}}(U)$.
We have $\pi_{\underline{\E}}(\chi)=(\underline{\shL},a)\in\stX_{\phi}(U)$
where $a$ is given, for any $t\in T_{+}$, by   \[   \begin{tikzpicture}[xscale=2.0,yscale=-0.6]     \node (A0_0) at (0, 0) {$\underline{\shL}^{\phi(t)}$};     \node (A0_1) at (1, 0) {$\underline{\shM}^{\E(t)}$};     \node (A1_0) at (0, 1) {$a(t)$};     \node (A1_1) at (1, 1) {$\underline{z}^{\E(t)}$};     \path (A0_0) edge [->]node [auto] {$\scriptstyle{\lambda_t}$} (A0_1);     \path (A1_0) edge [|->,gray]node [auto] {$\scriptstyle{}$} (A1_1);   \end{tikzpicture}   \] 
Instead, if $(\underline{\omega},\underline{\tau})$ is an isomorphism
in $\stF_{\underline{\E}}$, then $\pi_{\underline{\E}}(\underline{\omega},\underline{\tau})=\underline{\omega}$.

If $(\underline{z},\lambda)\in\stF_{\underline{\E}}(U)$ then $a=\pi_{\underline{\E}}(\underline{z},\lambda)\in\stX_{\phi}(U)$
is given by   \[   \begin{tikzpicture}[xscale=3.2,yscale=-0.6]     \node (A0_0) at (0, 0) {$T_+$};     \node (A0_1) at (1, 0) {$\odi{U}$};     \node (A1_0) at (0, 1) {$t$};     \node (A1_1) at (1, 1) {$ {\underline z}^{\underline \E(t)}/\lambda_t= z_1^{\E^1(t)} \cdots z_s^{\E^s(t)}/\lambda_t$};     \path (A0_0) edge [->] node [auto] {$\scriptstyle{}$} (A0_1);     \path (A1_0) edge [|->,gray] node [auto] {$\scriptstyle{}$} (A1_1);   \end{tikzpicture}   \] 
\end{rem}

\begin{rem}
\label{rem:Fdelta is an open substack of Fepsilon se delta sottosequenza di epsilon}If
$\underline{\E}=(\E^{i})_{i\in I}$ is a sequence of elements of $\duale T_{+}$,
$J\subseteq I$ and we set $\underline{\delta}=(\E^{j})_{j\in J}$
we can define a map over $\stX_{\phi}$ as   \[   \begin{tikzpicture}[xscale=2.8,yscale=-0.3]     \node (A0_0) at (0, 0) {$\stF_{\underline{\delta}}$};     \node (A0_1) at (1, 0) {$\stF_{\underline{\E}}$};     
\node (A1_3) at (2.6, 1) {$\shM_{i}'=\left\{ \begin{array}{cc} \shM_{i} & i\in J\\ \odi{} & i\notin J\end{array}\right. z_{i}'=\left\{ \begin{array}{cc} z_{i} & i\in J\\ 1 & i\notin J\end{array}\right.$};     \node (A2_0) at (0, 2) {$(\underline{\shL},\underline{\shM},\underline{z},\lambda)$};     \node (A2_1) at (1, 2) {$(\underline{\shL},\underline{\shM}',\underline{z}',\lambda)$};     \path (A0_0) edge [->]node [auto] {$\scriptstyle{\rho}$} (A0_1);     \path (A2_0) edge [|->,gray]node [auto] {$\scriptstyle{}$} (A2_1);   \end{tikzpicture}   \]  $\rho$ comes from the monoid map $T\oplus\N^{I}\arr T\oplus\N^{J}$
induced by the projection. $\rho$ is an open immersion, whose image
is the open substack of $\stF_{\underline{\E}}$ of objects $(\underline{\shL},\underline{\shM},\underline{z},\lambda)$
such that $z_{i}$ generates $\shM_{i}$ for any $i\notin J$. We
will often consider $\stF_{\underline{\delta}}$ as an open substack
of $\stF_{\underline{\E}}$.\end{rem}
\begin{defn}
\label{def:T+epsilon sottolineato}Given a sequence $\underline{\E}=\E^{1},\dots,\E^{s}$
of elements of $\duale T_{+}$ we define
\[
T_{+}^{\underline{\E}}=T_{+}^{\E^{1},\dots,\E^{s}}=\{v\in T\:|\:\forall i\:\E^{i}(v)\geq0\}
\]
 We also consider the case $s=0$, so that $T_{+}^{\underline{\E}}=T$.
If we denote by $\hat{\phi}\colon T_{+}^{\underline{\E}}\arr\Z^{r}$
the extension of $\phi$, we also define $\stX_{\phi}^{\underline{\E}}=\stZ_{\phi}^{\underline{\E}}=\stX_{\hat{\phi}}$.\end{defn}
\begin{rem}
\label{rem:change of monoid for stack X} Assume to have a monoid
map $T_{+}\arr T_{+}'$ inducing an isomorphism on the associated
groups. If $\underline{\E}=\E^{1},\dots,\E^{s}\in\duale{T'}_{+}\subseteq\duale T_{+}$,
then we have a $2$-commutative diagram   \[   \begin{tikzpicture}[xscale=1.8,yscale=-1.3]     \node (A0_0) at (0, 0) {$\stF_{\underline \E}'$};     \node (A0_1) at (1, 0) {$\stF_{\underline \E}$};     \node (A1_0) at (0, 1) {$\stX_{\phi'}$};     \node (A1_1) at (1, 1) {$\stX_\phi$};     \path (A0_0) edge [->] node [auto] {$\scriptstyle{\simeq}$} (A0_1);     \path (A0_0) edge [->] node [auto,swap] {$\scriptstyle{\pi_{\underline \E}'}$} (A1_0);     \path (A0_1) edge [->] node [auto] {$\scriptstyle{\pi_{\underline \E}}$} (A1_1);     \path (A1_0) edge [->] node [auto] {$\scriptstyle{}$} (A1_1);   \end{tikzpicture}   \] 
where $\stF_{\underline{\E}}'$ is the stack obtained from $T_{+}'$
with respect to $\underline{\E}$.\end{rem}
\begin{prop}
The map $\pi_{\underline{\E}}\colon\stF_{\underline{\E}}\arr\stX_{\phi}$
has a natural factorization
\[
\stF_{\underline{\E}}\arr\stX_{\phi}^{\underline{\E}}\arr\stZ_{\phi}\arr\stX_{\phi}
\]
\end{prop}
\begin{proof}
The factorization follows from \ref{rem:change of monoid for stack X}
taking monoid maps $T_{+}\arr T_{+}^{int}\arr T_{+}^{\underline{\E}}$.\end{proof}
\begin{rem}
This shows that $\pi_{\underline{\E}}$ has image in $\stZ_{\phi}$.We
will call with the same symbol $\pi_{\underline{\E}}$ the factorization
$\stF_{\underline{\E}}\arr\stZ_{\phi}$.
\end{rem}
We want now to show how the rays of $T_{+}$ can be used to describe
the objects of $\stZ_{\phi}$ over a field. Here is the result.
\begin{thm}
\label{pro:characterization of points of Zphi}Let $k$ be a field
and $T_{+}\arrdi ak\in\stX_{\phi}(k)$. Then $a\in\stZ_{\phi}(k)$
if and only if there exists $\lambda:T\arr\overline{k}^{*}$ and $\E\in\duale T_{+}$
such that 
\[
a(t)=\lambda_{t}0^{\E(t)}
\]
In particular if $\underline{\E}=\E^{1},\dots,\E^{r}$ generate $\duale T_{+}\otimes\Q$
then $\pi_{\underline{\E}}\colon\stF_{\underline{\E}}(\overline{k})\arr\stZ_{\phi}(\overline{k})$
is essentially surjective and so $\pi_{\underline{\E}}\colon|\stF_{\underline{\E}}|\arr|\stZ_{\phi}|$
is surjective. Finally, if the map $\phi\colon T\arr\Z^{r}$ is injective,
we have a one to one correspondence   \[   \begin{tikzpicture}[xscale=3.2,yscale=-0.5]     \node (A0_0) at (0, 0) {$\stZ_\phi(\overline k)/\simeq$};     \node (A0_1) at (1, 0) {$\{\Supp \E \text{ for } \E\in\duale{T}_+\}$};     \node (A1_0) at (0, 1) {$a$};     \node (A1_1) at (1, 1) {$\{a=0\}$};     \path (A0_0) edge [->] node [auto] {$\scriptstyle{\gamma}$} (A0_1);     \path (A1_0) edge [|->,gray] node [auto] {$\scriptstyle{}$} (A1_1);   \end{tikzpicture}   \] 
In particular $|\stZ_{\phi}|=\stZ_{\phi}(\overline{\Q})/\simeq\bigsqcup(\bigsqcup_{\textup{primes }p}\stZ_{\phi}(\overline{\F_{p}})/\simeq)$.
\end{thm}
Before proving this Theorem we need some preliminary results, that
we will be useful also later.
\begin{defn}
If $T_{+}$ is integral, $\E\in\duale T_{+}$ and $k$ is a field
we define x
\[
p_{\E}=\bigoplus_{v\in T_{+},\E(v)>0}kx_{v}\subseteq k[T_{+}]
\]
If $p\in\Spec k[T_{+}]$ we set $p^{om}=\bigoplus_{x_{v}\in p}kx_{v}$.\end{defn}
\begin{lem}
\label{lem:properties of pE}Let $k$ be a field and assume that $T_{+}$
is integral. Then:
\begin{enumerate}
\item if $\E\in\duale T_{+}$, $p_{\E}$ is prime and $k[\{v\in T_{+}\st\E(v)=0\}]\arr k[T_{+}]\arr k[T_{+}]/p_{\E}$
is an isomorphism.
\item If $p\in\Spec k[T_{+}]$ then $p^{om}=p_{\E}$ for some $\E\in\duale T_{+}$.
\end{enumerate}
\end{lem}
\begin{proof}
$(1)$ It's obvious.

$(2)$ $p^{om}$ is a prime thanks to \cite[Proposition  1.7.12]{Kreuzer2005}
and therefore $p^{om}=p_{\E}$ for some $\E\in\duale T_{+}$ thanks
to \cite[Corollary 2.2.4]{Ogus2006}.\end{proof}
\begin{rem}
\label{rem:differ by torsor implies iso}If $k$ is an algebraically
closed field, $\phi\colon T\arr\Z^{r}$ is injective and $a,b\in\stX_{\phi}(k)$
differ by a torsor, i.e. there exists $\lambda\colon T_{+}\arr k^{*}$
such that $a=\lambda b$, then $a\simeq b$ in $\stZ_{\phi}(k)$.
Indeed $\lambda$ extends to a map $T\arr k^{*}$ and, since $k$
is algebraically closed, it extends again to a map $\lambda\colon\Z^{r}\arr k^{*}$.\end{rem}
\begin{proof}
(of Theorem \ref{pro:characterization of points of Zphi}) We can
assume that $k$ is algebraically closed and that $T_{+}$ is integral,
since if $a$ has a writing as in the statement then clearly $a\in\stZ_{\phi}(k)$.
Consider $p=\Ker a$. Thanks to \ref{lem:properties of pE}, we can
write $p^{om}=p_{\E}$ for some $\E\in\duale T_{+}$. Set $T_{+}'=\{v\in T_{+}\st\E(v)=0\}$
and $T'=<T'>_{\Z}$ $ $ . Since $a\colon T_{+}'\arr k^{*}$, there
exists an extension $\lambda\colon T'\arr k^{*}$. On the other hand,
since $k$ is algebraically closed, the inclusion $T'\arr T$ yields
a surjection
\[
\Hom(T,k^{*})\arr\Hom(T',k^{*})
\]
and so we can extend again to an element $\lambda\colon T\arr k^{*}$.
Since one has $\Supp\E=\{a=0\}$ by construction, it is easy to check
that $a(t)=\lambda_{t}0^{\E(t)}$ for any $t\in T_{+}$.

Now consider the last part of the statement and so assume $\phi\colon T\arr\Z^{r}$
injective. The map $\gamma$ is well defined thanks to above and surjective
since, given $\E\in\duale T_{+}$, one can always define $a(t)=0^{\E(t)}$.
For the injectivity, let $a,b\in\stZ_{\phi}(k)$ such that $\{a=0\}=\{b=0\}$.
We can write $a(t)=\lambda_{t}0^{\E(t)}\comma b(t)=\mu_{t}0^{\E(t)}$,
where $\lambda,\mu\colon T\arr k^{*}$, so that $a,b$ differ by a
torsor and are therefore isomorphic thanks to \ref{rem:differ by torsor implies iso}.
Finally, since any point of $|\stZ_{\phi}|$ comes from $\Z$, we
also have the last equality.
\end{proof}
In some cases the description of the objects of $\stF_{\underline{\E}}$
can be simplified, regardless of $\underline{\E}$, in the sense that
there exist a stack of reduced data $\stF_{\underline{\E}}^{\textup{red}}$,
whose objects can be described by less data, and an isomorphism $\stF_{\underline{\E}}\simeq\stF_{\underline{\E}}^{\textup{red}}$.
This kind of simplification could be very useful when we have to deal
with an explicit map of monoid $\phi\colon T_{+}\arr\Z^{r}$, as we
will see in \ref{pro:stack of reduced data for M-covers}. The idea
is that in order to define an object $(\underline{\shL},\underline{\shM},\underline{z},\lambda)\in\stF_{\underline{\E}}$
we don't really need all the invertible sheaves $\shL_{1},\dots,\shL_{r}$,
because they are uniquely determined by a subset of them and the other
data.
\begin{defn}
\label{def:stack of reduced data}Assume $T\arrdi{\phi}\Z^{r}$ injective.
Let $V\subseteq\Z^{r}$ be a submodule with a given basis $v_{1},\dots,v_{q}$
and $\sigma\colon\Z^{r}\arr V$ be a map such that $(\id-\sigma)\Z^{r}\subseteq T$
(or equivalently $ $$\pi=\pi\circ\sigma$ where $\pi$ is the projection
$\Z^{r}\arr\Coker\phi$). Define $W=<(\id-\sigma)V,\sigma T>$. Given
$\underline{\E}=\E^{1},\dots\E^{l}\in\duale T_{+}$ consider the map
  \[   \begin{tikzpicture}[xscale=2.8,yscale=-0.5]     \node (A0_0) at (0, 0) {$W\oplus \N^s$};     \node (A0_1) at (1, 0) {$\Z^q \oplus \Z^s$};     \node (A1_0) at (0, 1) {$(w,z)$};     \node (A1_1) at (1, 1) {$(-w,\underline{\E}(w)+z)$};     \path (A0_0) edge [->]node [auto] {$\scriptstyle{\psi_{\underline \E,\sigma}}$} (A0_1);     \path (A1_0) edge [|->,gray]node [auto] {$\scriptstyle{}$} (A1_1);   \end{tikzpicture}   \] We
define $\stF_{\underline{\E}}^{\textup{red},\sigma}=\stX_{\psi_{\underline{\E},\sigma}}$
and we call it the stack of reduced data of $\underline{\E}$.\end{defn}
\begin{lem}
\label{lem:for the stack of reduced data}Consider a submodule $U\subseteq\Z^{p}$,
a map $\underline{\E}\colon U\arr\Z^{l}$ and $\tau\colon\Z^{p}\arr\Z^{p}$
such that $(\id-\tau)\Z^{p}\subseteq U$. Consider the commutative
diagram   \[   \begin{tikzpicture}[xscale=1.9,yscale=-1.4]     
\node (A0_0) at (0, 0) {$(u,z)$};     
\node (A0_1) at (1, 0) {$U\oplus \N^l$};     
\node (A0_3) at (2.5, 0) {$U\oplus \N^l$};     
\node (A1_0) at (0, 1) {$(-u,\underline{\E}(u)+z)$};     
\node (A1_1) at (1, 1) {$\Z^p \oplus \Z^l$};     
\node (A1_3) at (2.5, 1) {$\Z^p \oplus \Z^l$};     
\node (A2_1) at (1, 1.4) {$(u,z)$};    
\node (A2_3) at (2.5, 1.4) {$(\tau u,\underline{\E}(u-\tau u)+z)$};     

\path (A0_1) edge [->]node [auto] {$\scriptstyle{\tau\oplus \id}$} (A0_3);     \path (A0_3) edge [->]node [auto] {$\scriptstyle{\psi}$} (A1_3);     \path (A1_1) edge [->]node [auto] {$\scriptstyle{}$} (A1_3);     \path (A2_1) edge [|->,gray]node [auto] {$\scriptstyle{}$} (A2_3);     \path (A0_0) edge [|->,gray]node [auto] {$\scriptstyle{}$} (A1_0);     \path (A0_1) edge [->]node [auto,swap] {$\scriptstyle{\psi}$} (A1_1);   \end{tikzpicture}   \] Then the induced map $\varphi\colon\stX_{\psi}\arr\stX_{\psi}$ is
isomorphic to $\id_{\stX_{\psi}}$.\end{lem}
\begin{proof}
Let $x_{1},\dots,x_{p}$ be a $\Z$-basis of $\Z^{p}$ with $a_{1},\dots,a_{k}\in\N$
such that $a_{1}x_{1},\dots,a_{k}x_{k}$ is a $\Z$-basis of $U$.
We want to define a natural isomorphism $\id_{\stX_{\psi}}\arrdi{\omega}\varphi$.
First note that it is enough to define it on the objects of $\stX_{\psi}$
coming from the atlas $\Spec\Z[U\oplus\N^{l}]$, prove the naturality
between such objects on a fixed scheme $T$ and for the restrictions.
An object coming from the atlas is of the form $(\lambda,\underline{z})$
where $\lambda\colon U\arr\odi T^{*}$ is an additive map and $\underline{z}=z_{1},\dots,z_{l}\in\odi T$.
Moreover $\varphi(\lambda,\underline{z})=(\tilde{\lambda},\underline{z})$
where $\tilde{\lambda}=\lambda\circ\tau$. Let $\underline{\eta}\in\Di{\Z^{p}}(T)$
the only elements such that $\underline{\eta}^{x_{i}}=\lambda(x_{i}-\tau x_{i})$
for $i=1,\dots,p$. These objects are well defined since $(\id-\tau)\Z^{p}\subseteq U$.
We claim that $\omega_{T,(\lambda,\underline{z})}=(\underline{\eta},\underline{1})$
is an isomorphism $(\lambda,\underline{z})\arr\varphi(\lambda,\underline{z})$
and define a natural transformation. It is an isomorphism since $1z_{j}=z_{j}$
and the condition
\[
\underline{\eta}^{-u}\underline{1}^{\underline{\E}(u)}\lambda(u)=\lambda(\tau u)\ \forall u\in U
\]
holds by construction checking it on the basis $a_{1}x_{1},\dots,a_{k}x_{k}$
of $U$ (see \ref{rem:description of isomorphism for local objects of X phi}).
It's also easy to check that this isomorphisms commute with the change
of basis. So it remains to prove that, if $(\underline{\sigma},\underline{\mu})$
is an isomorphism $(\lambda,\underline{z})\arr(\lambda',\underline{z}')$
then we have a commutative diagram   \[   \begin{tikzpicture}[xscale=2.8,yscale=-1.4]     \node (A0_0) at (0, 0) {$(\lambda,\underline z)$};     \node (A0_1) at (1, 0) {$(\lambda',\underline z')$};     \node (A1_0) at (0, 1) {$\varphi(\lambda,\underline z)$};     \node (A1_1) at (1, 1) {$\varphi(\lambda',\underline z')$};     \path (A0_0) edge [->]node [auto] {$\scriptstyle{(\underline \sigma,\underline \mu)}$} (A0_1);     \path (A0_1) edge [->]node [auto] {$\scriptstyle{\omega_{T,(\lambda',\underline z')}}$} (A1_1);     \path (A1_0) edge [->]node [auto] {$\scriptstyle{\varphi(\underline \sigma,\underline \mu)}$} (A1_1);     \path (A0_0) edge [->]node [auto,swap] {$\scriptstyle{\omega_{T,(\lambda,\underline z)}}$} (A1_0);   \end{tikzpicture}   \] 
We have $\varphi(\underline{\sigma},\underline{\mu})=(\tilde{\underline{\sigma}},\tilde{\underline{\mu}})$
with $\tilde{\underline{\mu}}=\underline{\mu}$ and $\tilde{\underline{\sigma}}^{x_{i}}=\underline{\sigma}^{\tau x_{i}}\underline{\mu}^{\underline{\E}(x_{i}-\tau x_{i})}$
(see \ref{rem: description of functors of Xphi on local objects}).
So it is easy to check that the commutativity in the second member
holds. For the first, the condition is $\tilde{\underline{\sigma}}\underline{\eta}=\underline{\eta}'\underline{\sigma}$,
which is equivalent to
\[
(\tilde{\underline{\sigma}}\underline{\eta})^{x_{i}}=\underline{\sigma}^{\tau x_{i}}\underline{\mu}^{\underline{\E}(x_{i}-\tau x_{i})}\lambda(x_{i}-\tau x_{i})=(\underline{\eta}'\underline{\sigma})^{x_{i}}=\lambda'(x_{i}-\tau x_{i})\underline{\sigma}^{x_{i}}
\]
 and to $\underline{\sigma}^{-(x_{i}-\tau x_{i})}\underline{\mu}^{\underline{\E}(x_{i}-\tau x_{i})}\lambda(x_{i}-\tau x_{i})=\lambda'(x_{i}-\tau x_{i})$
for any $i$. But, since $(\underline{\sigma},\underline{\mu})$ is
an isomorphism $(\lambda,\underline{z})\arr(\lambda',\underline{z}')$,
the condition 
\[
\underline{\sigma}^{-u}\underline{\mu}^{\underline{\E}(u)}\lambda(u)=\lambda'(u)\ \forall u\in U
\]
has to be satisfied.\end{proof}
\begin{prop}
\label{pro:isomorphism with the stack of reduced data}Assume $T\arrdi{\phi}\Z^{r}$
injective and let $\underline{\E}=\E^{1},\dots\E^{r}\in\duale T_{+}$
and $\sigma\comma V\comma v_{1},\dots,v_{q}$ be as in \ref{def:stack of reduced data}.
Then we have functors   \[   \begin{tikzpicture}[xscale=5.6,yscale=-0.5]     \node (A0_0) at (0, 0) {$((\underline \shN^{\sigma e_i}\otimes \underline \shM^{\underline \E(e_i-\sigma e_i)})_{i=1,\dots,r},\underline \shM, \underline z,\tilde \lambda)$};     \node (A0_1) at (1, 0) {$(\underline \shN,\underline \shM, \underline z,\lambda)$};     \node (A1_0) at (0, 1) {$\stF_{\underline \E}$};     \node (A1_1) at (1, 1) {$\stF_{\underline \E}^{red,\sigma}$};     \node (A2_0) at (0, 2) {$(\underline \shL,\underline \shM, \underline z,\lambda)$};     \node (A2_1) at (1, 2) {$((\underline \shL^{v_i})_{i=1,\dots,q},\underline \shM, \underline z, \lambda_{|W})$};     \path (A1_1) edge [->]node [auto] {$\scriptstyle{}$} (A1_0);     \path (A1_0) edge [->]node [auto] {$\scriptstyle{}$} (A1_1);     \path (A2_0) edge [|->,gray]node [auto] {$\scriptstyle{}$} (A2_1);     \path (A0_1) edge [|->,gray]node [auto] {$\scriptstyle{}$} (A0_0);   \end{tikzpicture}   \] 
for appropriate choices of $\tilde{\lambda}$ that are inverses of
each other.\end{prop}
\begin{proof}
Consider the commutative diagrams   \[   \begin{tikzpicture}[xscale=3.0,yscale=-1.4]     \node (A1_0) at (0, 1) {$W\oplus \N^s$};     \node (A1_1) at (1, 1) {$T\oplus \N^s$};     \node (A1_2) at (2, 1) {$T\oplus \N^s$};     \node (A1_3) at (3, 1) {$W\oplus \N^s$};     \node (A2_0) at (0, 2) {$\Z^q\oplus \Z^s$};     \node (A2_1) at (1, 2) {$\Z^r\oplus \Z^s$};     \node (A2_2) at (2, 2) {$\Z^r\oplus \Z^s$};     \node (A2_3) at (3, 2) {$\Z^q\oplus \Z^s$};     
\node (A3_2) at (2, 2.4) {$(x,y)$};     
\node (A3_3) at (3, 2.4) {$(\sigma x,\underline \E(x-\sigma x)+y)$};     
\path (A1_0) edge [right hook->]node [auto] {$\scriptstyle{}$} (A1_1);     \path (A1_3) edge [->]node [auto] {$\scriptstyle{\psi}$} (A2_3);     \path (A2_2) edge [->]node [auto] {$\scriptstyle{}$} (A2_3);     \path (A1_0) edge [->]node [auto,swap] {$\scriptstyle{\psi}$} (A2_0);     \path (A1_1) edge [->]node [auto] {$\scriptstyle{\phi_{\underline \E}}$} (A2_1);     \path (A1_2) edge [->]node [auto] {$\scriptstyle{\sigma\oplus \id}$} (A1_3);     \path (A2_0) edge [right hook->]node [auto] {$\scriptstyle{}$} (A2_1);     \path (A3_2) edge [|->,gray]node [auto] {$\scriptstyle{}$} (A3_3);     \path (A1_2) edge [->]node [auto,swap] {$\scriptstyle{\phi_{\underline \E}}$} (A2_2);   \end{tikzpicture}   \] They induce functors $\Lambda\colon\stF_{\underline{\E}}\arr\stF_{\underline{\E}}^{\textup{red},\sigma}$
and $\Delta\colon\stF_{\underline{\E}}^{\textup{red},\sigma}\arr\stF_{\underline{\E}}$
respectively, that behave as the functors of the statement thanks
to description given in \ref{lem:morphisms of stack Xphi}. Finally,
applying \ref{lem:for the stack of reduced data}, we obtain that
$\Lambda\circ\Delta\simeq\id$ and $\Delta\circ\Lambda\simeq\id$.
\end{proof}

\subsection{Integral extremal rays and smooth sequences.}

We continue to use notation from \ref{not:notation for a monoid}.
We have seen that given a collection $\underline{\E}=\E^{1},\dots,\E^{r}\in\duale T_{+}$
we can associate to it a stack $\stF_{\underline{\E}}$ and a 'parametrization'
map $\stF_{\underline{\E}}\arr\stX_{\phi}$. The stack $\stF_{\underline{\E}}$
could be 'too big' if we don't make an appropriate choice of the collection
$\underline{\E}$. This happens for example if the rays in $\underline{\E}$
are not distinct or, more generally, if a ray in $\underline{\E}$
belongs to the submonoid generated by the other rays in $\underline{\E}$.
Thus we want to restrict our attention to a special class of integral
rays, called extremal and to special sequences of them.
\begin{defn}
An integral \emph{extremal} ray for $T_{+}$ is an element $\E\in\duale T_{+}$
such that
\begin{itemize}
\item $\E$ is minimal for the condition $\Supp\E\neq\emptyset$;
\item $\E$ is normalized, i.e. $\E\colon T\arr\Z$ is surjective.
\end{itemize}
\end{defn}
\begin{lem}
Assume that $T_{+}$ is an integral monoid and let $v_{1},\dots,v_{l}$
be a system of generators of $T_{+}$. Then the integral extremal
rays are the normalized $\E\in\duale T_{+}-\{0\}$ such that $\Ker\E$
contains $\rk T-1$ $\Q$-independent vectors among the $v_{1},\dots,v_{l}$.
In particular they are finitely many and they generate $\Q_{+}\duale T_{+}$.\end{lem}
\begin{proof}
Denote by $\Omega\subseteq\duale T_{+}$ the set of elements defined
in the statement. From \cite[Section 1.2, (9)]{Fulton1993} it follows
that $\Q_{+}\Omega=\Q_{+}\duale T_{+}$. If $\E\in\Omega$ then it
is an integral extremal ray. Indeed 
\[
\emptyset\neq\Supp\E'\subseteq\Supp\E\then\exists\lambda\in\Q_{+}\text{ s.t. }\E'=\lambda\E\then\Supp\E'=\Supp\E
\]

Conversely let $\E$ be an integral extremal ray and consider a writing
\[
\E=\sum_{\delta\in\Omega}\lambda_{\delta}\delta\qquad\text{with }\lambda_{\delta}\in\Q_{\geq0}
\]
 There must exists $\delta$ such that $\lambda_{\delta}\neq0$. So
\[
\Supp\delta\subseteq\Supp\E\then\Supp\delta=\Supp\E\then\exists\mu\in\Q_{+}\text{ s.t. }\E=\mu\delta\then\E=\delta
\]
\end{proof}
\begin{cor}
For an integral extremal ray $\E$ and $\E'\in\duale T_{+}$ we have
\[
\Supp\E'=\Supp\E\iff\exists\lambda\in\Q_{+}\text{ s.t. }\E'=\lambda\E\iff\exists\lambda\in\N_{+}\text{ s.t. }\E'=\lambda\E
\]
\end{cor}
\begin{defn}
An element $v\in T_{+}$ is said \emph{indecomposable} if whenever
$v=v'+v''$ with $v',v''\in T_{+}$ it follows that $v'=0$ or $v''=0$.\end{defn}
\begin{prop}
$\duale T_{+}$ has a unique minimal system of generators composed
by the indecomposable elements. Moreover any integral extremal ray
is indecomposable.\end{prop}
\begin{proof}
The first claim of the statement follows from \cite[Proposition 2.1.2]{Ogus2006}
since $\duale T_{+}$ is sharp, i.e. doesn't contain invertible elements.
For the second consider an extremal ray $\E$ and assume $\E=\E'+\E''$.
We have
\[
\Supp\E',\Supp\E''\subseteq\Supp\E\then\E'=\lambda\E,\E''=\mu\E\text{ with }\lambda,\mu\in\N
\]
 and so $\E=(\lambda+\mu)\E\then\lambda+\mu=1\then\lambda=0\text{ or }\mu=0\then\E'=0\text{ or }\E''=0$.\end{proof}
\begin{defn}
\label{def:definition of smooth sequence and smooth elements}A \emph{smooth
sequence} for $T_{+}$ is a sequence $\underline{\E}=\E^{1},\dots,\E^{s}\in\duale T_{+}$
for which there exist elements $v_{1},\dots,v_{s}$ in the associated
integral monoid $T_{+}^{int}$ of $T_{+}$ such that 
\[
T_{+}^{int}\cap\Ker\underline{\E}\:\text{ generates }\Ker\underline{\E}\qquad\text{and}\qquad\E^{i}(v_{j})=\delta_{i,j}
\]

We will also say that a ray $\E\in\duale T_{+}-\{0\}$ is \emph{smooth}
if there exists a smooth sequence as above such that $\E\in<\E^{1},\dots,\E^{s}>_{\N}$
or, equivalently, such that $\Supp\E\subseteq\Supp\underline{\E}$.\end{defn}
\begin{rem}
\label{rem:the equivalently in definition of smooth sequence}If $T_{+}$
is integral and $\Omega$ is a system of generators of it, one can
always assume that $v_{i}\in\Omega.$ Moreover we also have that $\Omega\cap\Ker\underline{\E}$
generates $\Ker\underline{\E}$. 

Finally the {}``equivalently'' in definition \ref{def:definition of smooth sequence and smooth elements}
follows from the fact that, since $\Ker\underline{\E}$ is generated
by elements in $T_{+}^{int}$, then the inclusion of the supports
implies that $\E_{|\Ker\underline{\E}}=0$ and therefore $\E=\sum_{i}\E(v_{i})\E^{i}$.\end{rem}
\begin{lem}
\label{lem:decomposition of T+E and open smooth subscheme}Let $\underline{\E}=\E^{1},\dots,\E^{r}$
be a smooth sequence. Then
\[
T_{+}^{\underline{\E}}=\Ker\underline{\E}\oplus<v_{1},\dots,v_{r}>_{\N}\subseteq T\text{ where }v_{1}\dots,v_{r}\in T_{+}^{int}\comma\E^{i}(v_{j})=\delta_{i,j}
\]
Moreover, if $z_{1},\dots,z_{s}\in T_{+}^{int}$ generate $T_{+}^{int}$,
then $\Z[T_{+}^{\underline{\E}}]=\Z[T_{+}^{int}]_{\prod_{\underline{\E}(z_{i})=0}x_{z_{i}}}$
so that $\Spec\Z[T_{+}^{\underline{\E}}]$ $(\stX_{\phi}^{\underline{\E}})$
is a smooth open subscheme (substack) of $\Spec\Z[T_{+}^{int}]$ $(\stZ_{\phi})$.\end{lem}
\begin{proof}
We have $T=\Ker\underline{\E}\oplus<v_{1},\dots,v_{r}>_{\Z}$ and
clearly $\Ker\underline{\E}\oplus<v_{1},\dots,v_{q}>_{\N}\subseteq T_{+}^{\underline{\E}}$.
Conversely if $v\in T_{+}^{\underline{\E}}$ we can write 
\[
v=z+\sum_{i}\E^{i}(v)v_{i}\text{ with }z\in\Ker\underline{\E}\then v\in\Ker\underline{\E}\oplus<v_{1},\dots,v_{q}>_{\N}
\]
In particular $\Spec\Z[T_{+}^{\underline{\E}}]\simeq\A_{\Z}^{r}\times D_{\Z}(\Ker\underline{\E})$
and so both $\Spec\Z[T_{+}^{\underline{\E}}]$ and $\stX_{\phi}^{\underline{\E}}$
are smooth. Now let 
\[
I=\{i\st\underline{\E}(z_{i})=0\}\text{ and }S_{+}=<T_{+}^{int},-z_{i}\text{ for }i\in I>\subseteq T
\]
We need to prove that $S_{+}=T_{+}^{\underline{\E}}$. Clearly we
have the inclusion $\subseteq$. For the other one, it is enough to
prove that $-\Ker\underline{\E}\cap T_{+}^{int}\subseteq S_{+}$.
But if $v\in\Ker\underline{\E}\cap T_{+}^{int}$ then
\[
v=\sum_{j=1}^{s}a_{j}z_{j}=\sum_{j\in I}a_{j}z_{j}\then-v\in S_{+}
\]
\end{proof}
\begin{rem}
\label{rem:subsequences of smooth sequences are smooth too}Any subsequence
of a smooth sequence is smooth too. Indeed let $\underline{\delta}=\E^{1},\dots,\E^{s}$
a subsequence of a smooth sequence $\underline{\E}=\E^{1},\dots,\E^{r}$,
with $r>s$. We have to prove that $<\Ker\underline{\delta}\cap T_{+}^{int}>_{\Z}=\Ker\underline{\delta}$.
Take $v\in\Ker\underline{\delta}$. So
\[
v-\sum_{j=s+1}^{r}\E^{j}(v)v_{j}\in\Ker\underline{\E}=<\Ker\underline{\E}\cap T_{+}^{int}>_{\Z}\subseteq<\Ker\underline{\delta}\cap T_{+}^{int}>_{\Z}\then v\in<\Ker\underline{\delta}\cap T_{+}^{int}>_{\Z}
\]
\end{rem}
\begin{prop}
\label{lem:equivalent condition for a smooth integral extremal ray}Let
$\E\in\duale T_{+}$. Then $\E$ is a smooth integral extremal ray
if and only if $\E$ is a smooth sequence of one element, i.e. $\Ker\E\cap T_{+}^{int}$
generates $\Ker\E$ and there exists $v\in T_{+}$ such that $\E(v)=1$. 

In particular any element of a smooth sequence is a smooth integral
extremal ray.\end{prop}
\begin{proof}
We can assume $T_{+}$ integral. If $\E$ is smooth and extremal,
then there exists a smooth sequence $\E^{1},\dots,\E^{q}$ such that
$\E\in<\E^{1},\dots,\E^{q}>_{\N}$. Since $\E$ is indecomposable,
it follows that $\E=\E^{i}$ for some $i$. Conversely assume that
$\E$ is a smooth sequence. So it is smooth by definition and it is
normalized since $\E(v)=1$ for some $v$. Finally an inclusion $\Supp\delta\subseteq\Supp\E$
for $\delta\in\duale T_{+}$ means that $\delta\in<\E>_{\N}$, as
remarked in \ref{rem:the equivalently in definition of smooth sequence},
and so $\Supp\delta=\emptyset$ or $\Supp\delta=\Supp\E$.
\end{proof}
We conclude with a lemma that will be useful later.
\begin{lem}
\label{lem:comparison smooth sequences for different monoids}Let
$T_{+}\comma T_{+}'$ be integral monoids and $h\colon T\arr T'$
be an homomorphism such that $h(T_{+})=T_{+}'$ and $\Ker h=<\Ker h\cap T_{+}>$.
If $\underline{\E}=\E^{1},\dots\E^{r}\in\duale{T_{+}'}$ then
\[
\underline{\E}\text{ smooth sequence for }T_{+}'\iff\underline{\E}\circ h\text{ smooth sequence for }T_{+}
\]
\end{lem}
\begin{proof}
Clearly there exist $v_{i}\in T_{+}'$ such that $\E^{i}(v_{j})=\delta_{i,j}$
if and only if there exist $w_{i}\in T_{+}$ such that $\E^{i}\circ h(w_{j})=\delta_{i,j}$.
On the other hand we have a surjective morphism
\[
\Ker\underline{\E}\circ h/<\Ker\underline{\E}\circ h\cap T_{+}>_{\Z}\arr\Ker\underline{\E}/<\Ker\underline{\E}\cap T_{+}'>_{\Z}
\]
In order to conclude it is enough to prove that this map is injective.
So let $v\in T$ such that 
\[
h(v)=\sum_{j}a_{j}z_{j}\text{ with }a_{j}\in\Z\comma z_{j}\in T_{+}'\comma\underline{\E}(z_{j})=0
\]
Since $h(T_{+})=T_{+}'$, there exist $y_{j}\in T_{+}$ such that
$h(y_{j})=z_{j}$. In particular $y=\sum_{j}a_{j}y_{j}\in<\Ker\underline{\E}\circ h\cap T_{+}>_{\Z}$
and 
\[
v-y\in\Ker h=<\Ker h\cap T_{+}>\subseteq<\Ker\underline{\E}\circ h\cap T_{+}>
\]

\end{proof}

\subsection{The smooth locus $\stZ_{\phi}^{\textup{sm}}$ of the main component
$\stZ_{\phi}$.}
\begin{lem}
\label{lem:fundamental lemma for the smooth locus of X phi} Let $\underline{\E}=\E^{1},\dots,\E^{q}$
be a smooth sequence and $\chi$ be a finite sequence of elements
of $\duale T_{+}$. Assume that all the elements of $\chi$ are distinct,
each $\E^{i}$ is an element of $\chi$ and that for any $\delta$
in $\chi$ we have
\[
\delta\in<\E^{1},\dots,\E^{q}>_{\N}\then\exists i\;\delta=\E^{i}
\]
 As usual denote by $\pi_{\chi}$ the map $\stF_{\chi}\arr\stX_{\phi}$.
Then we have an equivalence
\[
\stF_{\underline{\E}}=\pi_{\chi}^{-1}(\stX_{\phi}^{\underline{\E}})\arrdi{\simeq}\stX_{\phi}^{\underline{\E}}
\]
\end{lem}
\begin{proof}
Set $\chi=\E^{1},\dots,\E^{q},\eta^{1},\dots,\eta^{l}=\underline{\E},\underline{\eta}$.
We first prove that $\pi_{\chi}^{-1}(\stX_{\phi}^{\underline{\E}})\subseteq\stF_{\underline{\E}}$.
Since they are open substacks, we can check this over an algebraically
closed field $k$. Let $(\underline{z},\lambda)\in\pi_{\chi}^{-1}(\stX_{\phi}^{\underline{\E}}$)
so that $a=\pi_{\chi}(\underline{z},\lambda)=\underline{z}^{\underline{\E}}/\lambda\colon T_{+}\arr k$
by \ref{rem:description of piE}. We have to prove that $z_{\eta_{j}}\neq0$.
Assume by contradiction that $z_{\eta_{j}}=0$. Since we can write
$a=b0^{\eta_{j}}$ and since $a$ extends to $T_{+}^{\underline{\E}}$
so that $a(t)\neq0$ if $t\in T_{+}\cap\Ker\underline{\E}$, we have
that $\eta_{j}$ is $0$ on $T_{+}\cap\Ker\underline{\E}$. In particular
\[
\Supp\eta^{j}\subseteq\Supp\underline{\E}\then\eta^{j}\in<\E^{1},\dots,\E^{q}>_{\N}\then\exists i\;\eta^{j}=\E^{i}
\]

Thanks to \ref{rem:Fdelta is an open substack of Fepsilon se delta sottosequenza di epsilon},
we can reduce to prove that if $\underline{\E}$ is a smooth sequence
such that $T_{+}=T_{+}^{\underline{\E}}$ then $\pi_{\underline{\E}}$
is an isomorphism. By \ref{lem:decomposition of T+E and open smooth subscheme}
we can write $T_{+}=W\oplus\N^{q}$, where $W$ is a free $\Z$-module
such that $\underline{\E}_{|W}=0$ and, if we denote by $v_{1},\dots,v_{q}$
the canonical base of $\N^{q}$, $\E^{j}(v_{i})=\delta_{i,j}$. Consider
the diagram   \[   \begin{tikzpicture}[xscale=3.0,yscale=-0.8]     
\node (A0_0) at (0, 0) {$\N^q\oplus T$};     
\node (A0_1) at (1, 0) {$T_+$};     
\node (A1_0) at (0, 1) {$\N^q \oplus W \oplus \Z^q$};     
\node (A1_1) at (1, 1) {$W \oplus \N^q$};     
\node (A1_3) at (2.5, 1.5) {$\gamma(e_i)=v_i\comma \gamma_{|W}=-\id_W\comma \gamma(v_i)=0$};     
\node (A2_3) at (2.5, 2.5) {$\delta(e_i)=\phi(v_i)\comma\delta_{|\Z^r}=\id_{\Z^r}$};     
\node (A3_0) at (0, 3) {$\Z^q\oplus\Z^r$};     
\node (A3_1) at (1, 3) {$\Z^r$};     
\node[rotate=-90] (u) at (0, 0.5) {$=$};     
\node[rotate=-90] (uu) at (1, 0.5) {$=$};     

\path (A1_0) edge [->]node [auto,swap] {$\scriptstyle{\sigma_{\underline \E}}$} (A3_0);     \path (A1_0) edge [->]node [auto] {$\scriptstyle{\gamma}$} (A1_1);     \path (A3_0) edge [->]node [auto] {$\scriptstyle{\delta}$} (A3_1);     \path (A1_1) edge [->]node [auto] {$\scriptstyle{\phi}$} (A3_1);   \end{tikzpicture}   \]  One can check directly its commutativity. In this way we get a map
$s\colon\stX_{\phi}\arr\stF_{\underline{\E}}$. Again a direct computation
on the diagrams defining $s$ and $\pi_{\underline{\E}}$ shows that
$\pi_{\underline{\E}}\circ s\simeq\id_{\stX_{\phi}}$ and that the
diagram inducing $G=s\circ\pi_{\underline{\E}}$ is   \[   \begin{tikzpicture}[xscale=3.0,yscale=-0.8]     \node (A0_0) at (0, 0) {$\N^q \oplus W \oplus \Z^q$};     \node (A0_1) at (1, 0) {$\N^q \oplus W \oplus \Z^q$};     \node (A0_3) at (2.5, 0.5) {$\alpha(e_i)=e_i-v_i,\alpha_{|W}=\id_{W}\comma \alpha_{|\Z^q}=0$};     \node (A1_3) at (2.5, 1.5) {$\beta(e_i)=\phi(v_i)\comma\beta_{|\Z^r}=\id_{\Z^r}$};     \node (A2_0) at (0, 2) {$\Z^q\oplus\Z^r$};     \node (A2_1) at (1, 2) {$\Z^q\oplus\Z^r$};     \path (A0_0) edge [->]node [auto] {$\scriptstyle{\alpha}$} (A0_1);     \path (A0_1) edge [->]node [auto] {$\scriptstyle{\sigma_{\underline \E}}$} (A2_1);     \path (A0_0) edge [->]node [auto,swap] {$\scriptstyle{\sigma_{\underline \E}}$} (A2_0);     \path (A2_0) edge [->]node [auto] {$\scriptstyle{\beta}$} (A2_1);   \end{tikzpicture}   \] We
will prove that $G\simeq\id_{\stF_{\underline{\E}}}$. An object of
$\stF_{\underline{\E}}(A)$, where $A$ is a ring, coming from the
atlas is given by $a=(\underline{z},\lambda,\underline{\mu})\colon\N^{q}\oplus W\oplus\Z^{q}\arr A$
where $\underline{z}=(a(e_{i}))_{i}=z_{1},\dots,z_{q}\in A$, $\lambda=a_{|W}\colon W\arr A^{*}$
is an homomorphism and $\underline{\mu}=(\mu(v_{i}))_{i}=\mu_{1}\dots,\mu_{q}\in A^{*}$.
Moreover $Ga=a\circ\alpha$ is $((z_{i}/\mu_{i})_{i},\lambda,\underline{1})$.
It is now easy to check that $(\underline{\mu},1)\colon Ga\arr a$
is an isomorphism and that this map defines an isomorphism $G\arr\id_{\stF_{\underline{\E}}}$.\end{proof}
\begin{cor}
\label{thm:toric open substack of Z phi via smooth sequence}If $\underline{\E}$
is a smooth sequence then $\pi_{\underline{\E}}\colon\stF_{\underline{\E}}\arr\stZ_{\phi}$
is an open immersion with image $\stX_{\phi}^{\underline{\E}}$.
\end{cor}
It turns out that if $\underline{\E}$ is a smooth sequence, then
$\stX_{\phi}^{\underline{\E}}$ has a more explicit description:
\begin{prop}
\label{pro:points of XphiE}Let $\underline{\E}=\E^{1},\dots,\E^{r}$
be a smooth sequence, $k$ be a field and $a\in\stX_{\phi}(k)$
\[
a\in\stX_{\phi}^{\underline{\E}}(k)\iff\exists\E\in<\E^{1},\dots,\E^{r}>_{\N}\comma\lambda\colon T\arr\overline{k}^{*}\text{ s.t. }a=\lambda0^{\E}
\]
Moreover if $\lambda0^{\E}\in\stX_{\phi}^{\underline{\E}}(k)$, for
some $\E\in\duale T_{+}$, $\lambda\colon T\arr\overline{k}^{*}$,
then $\E\in<\E^{1},\dots,\E^{r}>_{\N}$.\end{prop}
\begin{proof}
We can assume $k$ algebraically closed and $T_{+}$ integral. In
this case $a\in\stX_{\phi}^{\underline{\E}}(k)$ if and only if $a\colon T_{+}\arr k$
extends to a map $\Ker\underline{\E}\oplus\N^{r}=T_{+}^{\underline{\E}}\arr k$.
So $\Leftarrow$ holds. Conversely, from \ref{pro:characterization of points of Zphi},
we can write $a=\lambda0^{\E}$ where $\lambda\colon T\arr k^{*}$
and $\E\in\duale{(T_{+}^{\underline{\E}})}$. From \ref{lem:decomposition of T+E and open smooth subscheme}
we see that $\duale{T_{+}^{\underline{\E}}}=<\E^{1},\dots,\E^{r}>_{\N}$.
Finally, if $\lambda0^{\E}\in\stX_{\phi}^{\underline{\E}}$ for some
$\E$, then $\Supp\E\subseteq\Supp\underline{\E}$ and we have done.\end{proof}
\begin{lem}
\label{lem:fundamental lemma for all the classification for h}Let
$\underline{\E}=(\E^{i})_{i\in I}$ be a sequence of distinct smooth
extremal rays and $\Theta$ be a collection of smooth sequences with
rays in $\underline{\E}$. Set
\[
\stF_{\underline{\E}}^{\Theta}=\left\{ (\underline{\shL},\underline{\shM},\underline{z},\delta)\in\stF_{\underline{\E}}\left|\begin{array}{c}
V(z_{i_{1}})\cap\cdots\cap V(z_{i_{s}})\neq\emptyset\\
\text{iff }\exists\underline{\delta}\in\Theta\text{ s.t. }\E^{i_{1}},\dots,\E^{i_{s}}\subseteq\underline{\delta}
\end{array}\right.\right\} 
\]
Then, taking into account the identification made in \ref{rem:Fdelta is an open substack of Fepsilon se delta sottosequenza di epsilon},
we have 
\[
\stF_{\underline{\E}}^{\Theta}=\bigcup_{\underline{\delta}\in\Theta}\stF_{\underline{\delta}}
\]
\end{lem}
\begin{proof}
Let $\chi=(\underline{\shL},\underline{\shM},\underline{z},\lambda)\in\bigcup_{\underline{\delta}\in\Theta}\stF_{\underline{\delta}}(T)$,
for some scheme $T$ and let $p\in V(z_{i_{1}})\cap\cdots\cap V(z_{i_{s}})$.
This means that the pullback of $\pi_{\underline{\E}}(\chi)$ to $\overline{k(p)}$
is given by $a=b0^{\E^{i_{1}}+\cdots+\E^{i_{r}}}$ for some $b\colon T_{+}\arr\overline{k(p)}$.
By definition there exists $\underline{\delta}\in\Theta$ such that
$a\in\stF_{\underline{\delta}}(\overline{k(p)})$, i.e. $a=\mu0^{\delta}$
for some $\delta\in<\underline{\delta}>_{\N}$, $\mu\colon T\arr\overline{k(p)}^{*}$.
So
\[
\Supp\E^{i_{j}}\subseteq\{a=0\}=\Supp\delta\subseteq\Supp\underline{\delta}\then\E^{i_{j}}\in<\underline{\delta}>_{\N}
\]

For the other inclusion, since all the $\stF_{\underline{\delta}}$
are open substacks of $\stF_{\underline{\E}}$, we can reduce to the
case of an algebraically closed field $k$. So let $(\underline{z},\lambda)\in\stF_{\underline{\E}}^{\Theta}(k)$
and set $J=\{i\in I\st z_{i}=0\}$. By definition of $\stF_{\underline{\E}}^{\Theta}$
there exists $\underline{\delta}\in\Theta$ such that $\underline{\eta}=(\E^{j})_{j\in J}\subseteq\underline{\delta}$
and, taking into account \ref{rem:Fdelta is an open substack of Fepsilon se delta sottosequenza di epsilon},
this means that $a\in\stF_{\underline{\eta}}(k)\subseteq\stF_{\underline{\delta}}(k)$.\end{proof}
\begin{defn}
Let $\Theta$ be a collection of smooth sequences. We define
\[
X_{\phi}^{\Theta}=\bigcup_{\underline{\delta}\in\Theta}\Spec\Z[T_{+}^{\underline{\delta}}]\subseteq\Spec\Z[T_{+}]\text{ and }\stX_{\phi}^{\Theta}=\bigcup_{\underline{\delta}\in\Theta}\stX_{\phi}^{\underline{\delta}}\subseteq\stZ_{\phi}
\]
\end{defn}
\begin{thm}
\label{pro:piE for theta isomorphism}Let $\underline{\E}=(\E^{i})_{i\in I}$
be a sequence of distinct smooth integral rays and $\Theta$ be a
collection of smooth sequences with rays in $\underline{\E}$. Then
we have an isomorphism
\[
\stF_{\underline{\E}}^{\Theta}=\pi_{\underline{\E}}^{-1}(\stX_{\phi}^{\Theta})\arrdi{\simeq}\stX_{\phi}^{\Theta}
\]
\end{thm}
\begin{proof}
Taking into account \ref{lem:fundamental lemma for all the classification for h},
it is enough to note that
\[
\pi_{\underline{\E}}^{-1}(\stX_{\phi}^{\Theta})=\pi_{\underline{\E}}^{-1}(\bigcup_{\underline{\delta}\in\Theta}\stX_{\phi}^{\underline{\delta}})=\bigcup_{\underline{\delta}\in\Theta}\stF_{\underline{\E}\cap\underline{\delta}}=\bigcup_{\underline{\delta}\in\Theta}\stF_{\underline{\delta}}\arrdi{\simeq}\stX_{\phi}^{\Theta}
\]
\end{proof}
\begin{prop}
\label{pro:Xphitheta is a smooth toric stack}Let $\underline{\E}=(\E^{i})_{i\in I}$
be a sequence of distinct smooth integral rays and $\Theta$ be a
collection of smooth sequences with rays in $\underline{\E}$. Then
the set
\[
\Delta^{\Theta}=\{<\eta_{1},\dots,\eta_{r}>_{\Q_{+}}\st\exists\underline{\delta}\in\Theta\text{ s.t. }\eta_{1},\dots,\eta_{r}\subseteq\underline{\delta}\}
\]
is a toric fan in $\duale T\otimes\Q$ whose associated toric variety
over $\Z$ is $X_{\phi}^{\Theta}$. Moreover
\[
\stX_{\phi}^{\Theta}\simeq[X_{\phi}^{\Theta}/\Di{\Z^{r}}]
\]
\end{prop}
\begin{proof}
We know that if $\underline{\eta}$ is a smooth sequence then $\Spec\Z[T_{+}^{\underline{\eta}}]$
is a smooth open subset of $\Spec\Z[T_{+}^{int}]$ and it is the affine
toric variety associated to the cone $<\underline{\eta}>_{\Q_{+}}$.
It is then easy to check that $\Delta^{\Theta}$ is a fan whose associated
toric variety is $X_{\phi}^{\Theta}$. Since $\Spec\Z[T_{+}^{\underline{\eta}}]$
is the equivariant open subset of $\Spec\Z[T_{+}^{int}]$ inducing
$\stX_{\phi}^{\underline{\eta}}$ in $\stZ_{\phi}$, then $X^{\Theta}$
is the equivariant open subset of $\Spec\Z[T_{+}^{int}]$ inducing
$\stX_{\phi}^{\Theta}$. In particular we obtain the last isomorphism.\end{proof}
\begin{lem}
\label{pro:smooth locus of Z[T +]} Assume $T_{+}$ integral and set
$\Theta$ for the set of all smooth sequences. Then $X_{\phi}^{\Theta}$
is the smooth locus of $\Spec\Z[T_{+}]$. In particular $\stZ_{\phi}^{\textup{sm}}=\stX_{\phi}^{\Theta}\simeq[X_{\phi}^{\Theta}/\Di{\Z^{r}}]$.\end{lem}
\begin{proof}
From \ref{lem:decomposition of T+E and open smooth subscheme} we
know that $\Spec\Z[T_{+}^{\underline{\E}}]$ is smooth over $\Z$
and it is an open subset of $\Spec\Z[T_{+}]$. So we focus on the
converse. Since $\Spec\Z[T_{+}]$ is flat over $\Z$, we can replace
$\Z$ by an algebraically closed field $k$. Let $p\in\Spec k[T_{+}]$
be a smooth point. In particular $p^{om}$ is smooth too. If $p^{om}=0$
then $p\in\Spec k[T]$ and we have done. So we can assume $p^{om}=p_{\E}$
for some $0\neq\E\in\duale T_{+}$ thanks to \ref{lem:properties of pE}.
We claim that there exist a smooth sequence $\E^{1},\dots,\E^{q}$
such that $\E\in<\E^{1},\dots,\E^{q}>_{\N}$. This is enough to conclude
that $p\in\Spec k[T_{+}^{\underline{\E}}]$ . Indeed if $x_{w}\in p$
for some $w\in\Ker\underline{\E}\cap T_{+}$ then it belongs to $p^{om}=p_{\E}$
and so $\E(w)>0$, which is not our case.

So assume to have $\E\in\duale T_{+}$ such that $p_{\E}$ is a regular
point. Set $W=<\Ker\E\cap T_{+}>_{\Z}$ and $T_{+}'=T_{+}+W$. Note
that $\Spec k[T_{+}']$ is an open subset of $\Spec k[T_{+}]$ that
contains $p_{\E}$. Moreover $k[T_{+}']/p_{\E}=k[W]$. Let $v_{1},\dots,v_{q}\in T_{+}$
be elements such that
\[
T_{+}'=<v_{1},\dots,v_{q}>_{\N}+W\qquad\text{and}\qquad\E(v_{i})>0
\]
with $q$ minimal. We claim that $M=p_{\E}/p_{\E}^{2}\simeq k[W]^{q}$,
where $p_{\E}$ is thought in $k[T_{+}']$. Indeed $M$ is a $k$-vector
space over the $x_{v}$, $v\in T_{+}'$ that satisfies: $\E(v)>0$
and whenever we have $v=v'+v''$ with $v',v''\in T_{+}'$ it follows
that $\E(v')=0$ or $\E(v'')=0$. A simple computation shows that
such a $v$ must be of the form $v_{i}+W$ for some $i$. But since
we have chosen $q$ minimal we have $(v_{i}+W)\cap(v_{j}+W)=\emptyset$
if $i\neq j$. This implies that $M$ is a free $k[W]$-module with
basis $x_{v_{1}},\dots,x_{v_{q}}$. This shows that $q=\alt p_{\E}$.

Now set $V=<v_{1},\dots,v_{q}>_{\Z}$. Since $V+W=T$, $\rk V\leq q$
and 
\[
k[W]\simeq k[T_{+}']/p_{\E}\then\rk T=\dim k[T_{+}']=\alt p_{\E}+\dim k[W]=q+\rk W
\]
we obtain that $v_{1},\dots,v_{q}$ are independents. Let $\E^{1},\dots,\E^{q}$
given by $\E^{i}(v_{j})=\delta_{i,j}$ and $\E_{|W}^{i}=0$. In particular
$W=\Ker\underline{\E}$ and it is generated by elements in $T_{+}$.
Since $\E_{|W}=0$ we have 
\[
\E=\sum_{i=1}^{q}\E(v_{i})\E^{i}\qquad\E(v_{i})>0
\]
 Moreover since $T_{+}\subseteq T_{+}'$ and $\E^{i}\in\duale{T'}_{+}$
we get that $\E^{i}\in\duale T_{+}$, as required.\end{proof}
\begin{thm}
\label{thm:fundamental theorem for the smooth locus of ZM} If $\underline{\E}$
is a sequence of distinct indecomposable rays containing the smooth
integral extremal rays then $\pi_{\underline{\E}}$ induces an equivalence
\[
\left\{ (\underline{\shL},\underline{\shM},\underline{z},\delta)\in\stF_{\underline{\E}}\left|\begin{array}{c}
V(z_{i_{1}})\cap\cdots\cap V(z_{i_{s}})=\emptyset\\
\text{if }\E^{i_{1}},\dots\E^{i_{s}}\text{ is not a}\\
\text{smooth sequence}
\end{array}\right.\right\} =\pi_{\underline{\E}}^{-1}(\stZ_{\phi}^{\textup{sm}})\arrdi{\simeq}\stZ_{\phi}^{\textup{sm}}
\]
\end{thm}
\begin{proof}
\ref{pro:smooth locus of Z[T +]} says us that $\stZ_{\phi}^{\textup{sm}}=\stX_{\phi}^{\Theta}$,
where $\Theta$ is the collection of all smooth sequences, while \ref{lem:fundamental lemma for the smooth locus of X phi}
allows us to replace $\underline{\E}$ with the sequence of all smooth
extremal rays. Therefore it is enough to apply \ref{pro:piE for theta isomorphism}
and \ref{pro:Xphitheta is a smooth toric stack}.\end{proof}
\begin{prop}
\label{pro:equivalent condition for belonging in the smooth locus of the main irreducible component of X phi}Let
$a\colon T_{+}\arr k\in\stX_{\phi}(k)$, where $k$ is a field. Then
$a$ lies in $\stZ_{\phi}^{\textup{sm}}$ if and only if there exists
a smooth ray $\E\in\duale T_{+}$ and $\lambda\colon T\arr\overline{k}^{*}$
such that $a=\lambda0^{\E}$. \end{prop}
\begin{proof}
Apply \ref{thm:fundamental theorem for the smooth locus of ZM} and
\ref{pro:points of XphiE}.
\end{proof}

\subsection{Extension of objects from codimension $1$.}

In this subsection we want to explain how it is possible, in certain
cases, to check that an object of $\stX_{\phi}$ over a 'good' scheme
$X$ comes (uniquely) from $\stF_{\underline{\E}}$ only checking
what happens in codimension $1$.
\begin{notation}
Given a scheme $X$ we will denote by $\Picsh X$ the category whose
objects are invertible sheaves and whose arrows are maps between them.\end{notation}
\begin{prop}
\label{pro:the map X(Y) --> X(X) is fully faithful (equivalence) is it so between Pic}Let
$X\arrdi fY$ be a map of scheme. If $\Picsh Y\arrdi{f^{*}}\Picsh X$
is fully faithful (an equivalence) then also $\stX_{\phi}(Y)\arrdi{f^{*}}\stX_{\phi}(X)$
is so.\end{prop}
\begin{proof}
Let $(\underline{\shL},a),(\underline{\shL}',a')\in\stX_{\phi}(Y)$
and $\underline{\sigma}\colon f^{*}(\underline{\shL},a)\arr f^{*}(\underline{\shL}',a')$
be a map in $\stX_{\phi}(X)$. Any map $\sigma_{i}\colon f^{*}\shL_{i}\arr f^{*}\shL_{i}$
comes from a unique map $\tau_{i}\colon\shL_{i}\arr\shL_{i}$, i.e.
$\sigma_{i}=f^{*}\tau_{i}$. Since 
\[
f^{*}(\underline{\tau}^{\phi(t)}(a(t)))=\underline{\sigma}^{\phi(t)}(f^{*}a(t))=f^{*}(a'(t))\then\underline{\tau}^{\phi(t)}(a(t))=a'(t)
\]
$\underline{\tau}$ is a map $(\underline{\shL},a)\arr(\underline{\shL}',a')$
such that $f^{*}\underline{\tau}=\underline{\sigma}$. We can conclude
that $f^{*}\colon\stX_{\phi}(Y)\arr\stX_{\phi}(X)$ is fully faithful.

Now assume that $\Picsh Y\arrdi{f^{*}}\Picsh X$ is an equivalence.
We have to prove that $\stX_{\phi}(Y)\arrdi{f^{*}}\stX_{\phi}(X)$
is essentially surjective. So let $(\underline{\shM},b)\in\stX_{\phi}(X)$.
Since $f^{*}$ is an equivalence we can assume $\shM_{i}=f^{*}\shL_{i}$
for some invertible sheaf $\shL_{i}$ on $Y$. Since for any invertible
sheaf $\shL$ on Y one has that $\shL(Y)\simeq(f^{*}\shL)(X)$, any
section $b(t)\in\underline{\shM}^{\phi(t)}$ extends to a unique section
$a(t)\in\underline{\shL}^{\phi(t)}$. Since
\[
f^{*}(a(t)\otimes a(s))=b(t)\otimes b(s)=b(t+s)=f^{*}(a(t+s))\then a(t)\otimes a(s)=a(t+s)
\]
for any $t,s\in T_{+}$ and $a(0)=1$, it follows that $(\underline{\shL},a)\in\stX_{\phi}(Y)$
and $f^{*}(\underline{\shL},a)=(\underline{\shM},b)$.\end{proof}
\begin{cor}
\label{cor:lift when Picsh is the same}Let $X\arrdi fY$ be a map
of schemes and consider a commutative diagram    \[   \begin{tikzpicture}[xscale=1.3,yscale=-1.0]     \node (A0_0) at (0, 0) {$X$};     \node (A0_1) at (1, 0) {$\stF_{\underline \E}$};     \node (A1_0) at (0, 1) {$Y$};     \node (A1_1) at (1, 1) {$\stX_\phi$};     \path (A0_0) edge [->]node [auto] {$\scriptstyle{}$} (A0_1);     \path (A1_0) edge [->,dashed]node [auto] {$\scriptstyle{}$} (A0_1);     \path (A0_0) edge [->]node [auto,swap] {$\scriptstyle{f}$} (A1_0);     \path (A0_1) edge [->]node [auto] {$\scriptstyle{\pi_{\underline \E}}$} (A1_1);     \path (A1_0) edge [->]node [auto] {$\scriptstyle{}$} (A1_1);   \end{tikzpicture}   \] where
$\underline{\E}$ is a sequence of elements of $\duale T_{+}$. Then
if $\Picsh X\arrdi{f^{*}}\Picsh Y$ is fully faithful (an equivalence)
the dashed lifting is unique (exists).\end{cor}
\begin{proof}
It is enough to consider the $2$-commutative diagram   \[   \begin{tikzpicture}[xscale=2.0,yscale=-1.0]     \node (A0_0) at (0, 0) {$\stF_{\underline \E}(Y)$};     \node (A0_1) at (1, 0) {$\stF_{\underline \E}(X)$};     \node (A1_0) at (0, 1) {$\stX_\phi(Y)$};     \node (A1_1) at (1, 1) {$\stX_\phi(X)$};     \path (A0_0) edge [->]node [auto] {$\scriptstyle{f^*}$} (A0_1);     \path (A0_0) edge [->]node [auto,swap] {$\scriptstyle{\pi_{\underline \E}}$} (A1_0);     \path (A0_1) edge [->]node [auto] {$\scriptstyle{\pi_{\underline \E}}$} (A1_1);     \path (A1_0) edge [->]node [auto] {$\scriptstyle{f^*}$} (A1_1);   \end{tikzpicture}   \] and
note that $f^{*}$ is fully faithful (an equivalence) in both cases.\end{proof}
\begin{thm}
\label{thm:fundamental theorem for locally factorial schemes}Let
$X$ be a locally noetherian and locally factorial scheme, $\underline{\E}=(\E^{i})_{i\in I}$
be a sequence of distinct smooth integral rays and $\Theta$ be a
collection of smooth sequences with rays in $\underline{\E}$. Set
\[
\catC_{X}^{\Theta}=\left\{ (\underline{\shL},\underline{\shM},\underline{z},\delta)\in\stF_{\underline{\E}}(X)\left|\begin{array}{c}
\codim_{X}V(z_{i_{1}})\cap\cdots\cap V(z_{i_{s}})\geq2\\
\text{if }\nexists\underline{\delta}\in\Theta\text{ s.t. }\E^{i_{1}},\dots\E^{i_{s}}\subseteq\underline{\delta}
\end{array}\right.\right\} 
\]
and 
\[
\catD_{X}^{\Theta}=\left\{ \chi\in\stX_{\phi}(X)\left|\begin{array}{c}
\forall p\in X\text{ with }\codim_{p}X\leq1\\
\chi_{|\overline{k(p)}}\in\stX_{\phi}^{\Theta}
\end{array}\right.\right\} 
\]
Then $\pi_{\underline{\E}}$ induces an equivalence of categories
\[
\catC_{X}^{\Theta}=\pi_{\underline{\E}}^{-1}(\catD_{X}^{\Theta})\arrdi{\simeq}\catD_{X}^{\Theta}
\]
\end{thm}
\begin{proof}
We claim that
\[
\catC_{X}^{\Theta}=\{\chi\in\stF_{\underline{\E}}(X)\st\exists U\subseteq X\text{ open subset s.t. }\codim_{X}X-U\geq2\comma\chi_{|U}\in\stF_{\underline{\E}}^{\Theta}(U)\}
\]

$\subseteq$ Taking into account the definition of $\stF_{\underline{\E}}^{\Theta}$
in \ref{lem:fundamental lemma for all the classification for h},
it is enough to consider 
\[
U=X-\bigcup_{\nexists\underline{\delta}\in\Theta\text{ s.t. }\E^{i_{1}},\dots\E^{i_{s}}\subseteq\underline{\delta}}V(z_{i_{1}})\cap\cdots\cap V(z_{i_{s}})
\]

$\supseteq$ If $p\in V(z_{i_{1}})\cap\cdots\cap V(z_{i_{s}})$ and
$\codim_{p}X\leq1$ then $p\in U$ and again by definition of $\stF_{\underline{\E}}^{\Theta}$
there exists $\underline{\delta}\in\Theta$ such that $\E^{i_{1}},\dots,\E^{i_{s}}\subseteq\underline{\delta}$.

We also claim that 
\[
\catD_{X}^{\Theta}=\{\chi\in\stX_{\phi}(X)\st\exists U\subseteq X\text{ open subset s.t. }\codim_{X}X-U\geq2\comma\chi_{|U}\in\stX_{\phi}^{\Theta}(U)\}
\]

$\supseteq$ Such a $U$ contains all the codimension $1$ or $0$
points of $X$.

$\subseteq$ Let $\chi\in\catD_{X}^{\Theta}$ and $X\arrdi g\stX_{\phi}$
be the induced map. If $\xi$ is a generic point of $X$, we know
that $f(\xi)\in|\stX_{\phi}^{\Theta}|\subseteq|\stZ_{\phi}|$. In
particular $f(|X|)\subseteq|\stZ_{\phi}|$. Since both $X$ and $\stZ_{\phi}$
are reduced $g$ factors to a map $X\arrdi g\stZ_{\phi}$. Since $\stX_{\phi}^{\Theta}$
is an open substack of $\stZ_{\phi}$, it follows that $U=g^{-1}(\stX_{\phi}^{\Theta})$
is an open subscheme of $X$, $\chi_{|U}\in\stX_{\phi}^{\Theta}(U)$
and, by definition of $\catD_{X}^{\Theta}$, $\codim_{X}X-U\geq2$.

Taking into account \ref{pro:piE for theta isomorphism} it is clear
that $\catC_{X}^{\Theta}=\pi_{\underline{\E}}^{-1}(\catD_{X}^{\Theta})$.
We will make use of the fact that if $U\subseteq X$ is an open subscheme
such that $\codim_{X}X-U\geq2$ then the restriction yields an equivalence
$\Picsh X\simeq\Picsh U$. The map $\catC_{X}^{\Theta}\arr\catD_{X}^{\Theta}$
is essentially surjective since, given an object of $\catD_{X}^{\Theta}$,
the associated map $X\arrdi g\stX_{\phi}$ fits in a $2$-commutative
diagram   \[   \begin{tikzpicture}[xscale=1.5,yscale=-1.2]     \node (A0_0) at (0, 0) {$U$};     \node (A0_1) at (1, 0) {$\stF_{\underline \E}^\Theta\subseteq \stF_{\underline \E}$};     \node (A1_0) at (0, 1) {$X$};     \node (A1_1) at (1, 1) {$\stX_\phi$};     \path (A0_0) edge [->]node [auto] {$\scriptstyle{}$} (A0_1);     \path (A1_0) edge [->]node [auto] {$\scriptstyle{g}$} (A1_1);     \path (A0_1) edge [->]node [auto] {$\scriptstyle{\pi_{\underline \E}}$} (A1_1);     \path (A0_0) edge [->]node [auto] {$\scriptstyle{}$} (A1_0);   \end{tikzpicture}   \] 
and so lifts to a map $X\arr\stF_{\underline{\E}}$ thanks to \ref{cor:lift when Picsh is the same}.

It remains to show that $\catC_{X}^{\Theta}\arr\catD_{X}^{\Theta}$
is fully faithful. Let $\chi,\chi'\in\catC_{X}^{\Theta}$ and $U,U'$
be the open subscheme given in the definition of $\catC_{X}^{\Theta}$.
Set $V=U\cap U'$. Taking into account \ref{pro:the map X(Y) --> X(X) is fully faithful (equivalence) is it so between Pic}
and \ref{pro:piE for theta isomorphism} we have   \[   \begin{tikzpicture}[xscale=4.3,yscale=-0.9]     \node (A0_0) at (0, 0) {$\Hom_{\stF_{\underline \E}(X)}(\chi,\chi')$};     \node (A0_1) at (1, 0) {$\Hom_{\stX_\phi(X)}(\chi,\chi')$};     \node (A1_0) at (0, 1) {$\Hom_{\stF_{\underline \E}(V)}(\chi_{|V},\chi'_{|V})$};     \node (A1_1) at (1, 1) {$\Hom_{\stX_\phi(V)}(\chi_{|V},\chi'_{|V})$};     \node (A2_0) at (0, 2) {$\Hom_{\stF_{\underline \E}^\Theta(V)}(\chi_{|V},\chi'_{|V})$};     \node (A2_1) at (1, 2) {$\Hom_{\stX_\phi^\Theta(V)}(\chi_{|V},\chi'_{|V})$};     
\node[rotate=-90] (s) at (0, 0.5) {$\simeq$};
\node[rotate=-90] (ss) at (1, 0.5) {$\simeq$};
\node[rotate=-90] (s2) at (0, 1.5) {$\simeq$};
\node[rotate=-90] (ss2) at (1, 1.5) {$\simeq$};

\path (A0_0) edge [->]node [auto] {$\scriptstyle{}$} (A0_1);     \path (A1_0) edge [->]node [auto] {$\scriptstyle{}$} (A1_1);     \path (A2_0) edge [->]node [auto,swap] {$\scriptstyle{\simeq}$} (A2_1);   \end{tikzpicture}   \] 
\end{proof}

\section{Galois covers for a diagonalizable group}

In this section we will fix a finite diagonalizable group scheme $G$
over $\Z$ and we will call $M=\Hom(G,\Gm)$ its dual group. So $M$
is a finite abelian group and $G=\Di M$. With abuse of notation we
will write $\odi U[M]=\odi U[G_{U}]$ and $\stZ_{M}=\stZ_{\Di M}$,
the main component of $\MCov$. It turns out that in this case $\Di M$-covers
have a nice and more explicit description.

In the first subsection we will show that $\MCov\simeq\stX_{\phi}$
for an explicit map $T_{+}\arrdi{\phi}\Z^{M}/<e_{0}>$ and that this
isomorphism preserves the main irreducible components of both stacks.
Moreover we will study the connection between $\MCov$ and the equivariant
Hilbert schemes $\MHilb^{\underline{m}}$ and prove some results about
their geometry.

Then we will introduce an upper semicontinuous map $|\MCov|\arrdi h\N$
that yields a stratification by open substacks of $\MCov$. We will
also see that $\{h=0\}$ coincides with the open substack of $\Di M$-torsors,
while $\{h\leq1\}$ lies in the smooth locus of $\stZ_{M}$ and can
be described by a particular set of smooth integral extremal rays.
This will allow to describe the $\Di M$-covers over locally noetherian
and locally factorial scheme $X$ with $(\car X,|M|)=1$ whose total
space is regular in codimension $1$ (which, a posteriori, is equivalent
to the normal condition).

\subsection{The stack $\MCov$ and its main irreducible component $\stZ_{M}$.}

Consider a scheme $U$ and a cover $X=\Spec\alA$ on it. An action
of $\Di M$ on it consists of a decomposition 
\[
\alA=\bigoplus_{m\in M}\alA_{m}
\]
 such that $\odi U\subseteq\alA_{0}$ and the multiplication satisfies
the rules $\alA_{m}\otimes\alA_{n}\arr\alA_{m+n}$. If $X/U$ is a
$\Di M$-cover there exists an fppf covering $\{U_{i}\arr U\}$ such
that $\alA_{|U_{i}}\simeq\odi{U_{i}}[M]$ as $\Di M$-comodule. This
means that for any $m\in M$ we have
\[
\forall i\;(\alA_{m})_{|U_{i}}\simeq\odi{U_{i}}\then\alA_{m}\text{ invertible}
\]
Conversely any $M$-graded quasi-coherent algebra $\alA=\bigoplus_{m\in M}\alA_{m}$
with $\alA_{0}=\odi U$ and $\alA_{m}$ invertible for any $m$ yields
a $\Di M$-cover $\Spec\alA$.

So the stack $\MCov$ can be described as follows. An object of $\MCov(U)$
is given by a collection of invertible sheaves $\shL_{m}$ for $m\in M$
with maps
\[
\psi_{m,n}\colon\shL_{m}\otimes\shL_{n}\arr\shL_{m+n}
\]
and an isomorphism $\odi U\simeq\shL_{0}$ satisfying the following
relations:   \[   \begin{tikzpicture}[xscale=1.5,yscale=-1.2]     
\node (A0_1) at (1, 0.4) {$\textup{Commutativity}$};     
\node (A0_5) at (5, 0.4) {$\textup{Associativity}$};     
\node (A1_0) at (0, 1) {$\shL_m\otimes \shL_n$};     
\node (A1_2) at (2, 1) {$\shL_n \otimes \shL_m$};     
\node (A1_4) at (4, 1) {$\shL_m \otimes \shL_n \otimes \shL_t$};     
\node (A1_6) at (6, 1) {$\shL_m \otimes \shL_{n+t}$};     \node (A2_1) at (1, 2) {$\shL_{m+n}$};     
\node (A2_4) at (4, 2) {$\shL_{m+n}\otimes \shL_t$};     
\node (A2_6) at (6, 2) {$\shL_{m+n+t}$};
\node (name) at (0.2, 3.3) {$\begin{array}{c}
\textup{Neutral}\\
\textup{Element}
\end{array}
$};
\node (A3_3) at (1.2, 3) {$\shL_m$};     
\node (A3_4) at (2.7, 3) {$\shL_m \otimes \odi{U}$};     
\node (A3_5) at (4.2, 3) {$\shL_m \otimes \shL_0$};     
\node (A3_6) at (5.7, 3) {$\shL_m$};

\path (A3_4) edge [->]node [auto] {$\scriptstyle{\simeq}$} (A3_5);     \path (A1_0) edge [->]node [auto,swap] {$\scriptstyle{\psi_{m,n}}$} (A2_1);     \path (A1_6) edge [->]node [auto] {$\scriptstyle{\psi_{m,n+t}}$} (A2_6);     \path (A1_0) edge [->]node [auto] {$\scriptstyle{\simeq}$} (A1_2);     \path (A1_2) edge [->]node [auto] {$\scriptstyle{\psi_{n,m}}$} (A2_1);     \path (A3_5) edge [->]node [auto] {$\scriptstyle{\psi_{m,0}}$} (A3_6);     \path (A2_4) edge [->]node [auto] {$\scriptstyle{\psi_{m+n,t}}$} (A2_6);     \path (A1_4) edge [->]node [auto] {$\scriptstyle{\id \otimes \psi_{n,t}}$} (A1_6);     \path (A3_3) edge [->]node [auto] {$\scriptstyle{\simeq}$} (A3_4);     \path (A1_4) edge [->]node [auto,swap] {$\scriptstyle{\psi_{m,n}\otimes \id}$} (A2_4);   
\path (A3_3) edge [->,bend left=20]node [auto,swap] {$\scriptstyle{\id}$} (A3_6);\end{tikzpicture}   \] 

If we assume that $\shL_{m}=\odi Uv_{m}$, i.e. to have sections $v_{m}$
generating $\shL_{m}$, the maps $\psi_{m,n}$ can be thought as elements
of $\odi U$ and the algebra structure is given by $v_{m}v_{n}=\psi_{m,n}v_{m+n}$.
In this case we can rewrite the above conditions obtaining 
\begin{equation}
\psi_{m,n}=\psi_{n,m},\quad\psi_{m,0}=1,\quad\psi_{m,n}\psi_{m+n,t}=\psi_{n,t}\psi_{n+t,m}\label{eq:condition on psi}
\end{equation}
The functor that associates to a scheme $U$ the functions $\psi\colon M\times M\arr\odi U$
satisfying the above conditions is clearly representable by the spectrum
of the ring
\begin{equation}
R_{M}=\Z[x_{m,n}]/(x_{m,n}-x_{n,m},x_{m,0}-1,x_{m,n}x_{m+n,t}-x_{n,t}x_{n+t,m})\label{eq:writing of RM}
\end{equation}
In this way we obtain a Zariski epimorphism $\Spec R_{M}\arr\MCov$,
that we will prove to be smooth. We now want to prove that the stack
$\MCov$ is isomorphic to a stack of the form $\stX_{\phi}$. 
\begin{defn}
Define $\tilde{K}_{+}$ as the monoid quotient of $\N^{M\times M}$
by the equivalence relation generated by 
\[
e_{m,n}\sim e_{n,m},\quad e_{m,0}\sim0,\quad e_{m,n}+e_{m+n,t}\sim e_{n,t}+e_{n+t,m}
\]
Also define $\phi_{M}\colon\tilde{K}_{+}\arr\Z^{M}/<e_{0}>$ by $\phi_{M}(e_{m,n})=e_{m}+e_{n}-e_{m+n}$. \end{defn}
\begin{prop}
$R_{M}=\Z[\tilde{K}_{+}]$ and there exists an isomorphism 
\begin{equation}
\stX_{\phi_{M}}\simeq\MCov\label{eq:MCov isomorphic to Xphi}
\end{equation}
 such that $\Spec R_{M}=\Spec\Z[\tilde{K}_{+}]\arr\stX_{\phi_{M}}\simeq\MCov$
is the forgetful map. In particular
\[
\MCov\simeq[\Spec R_{M}/\Di{\Z^{M}/<e_{0}>}]
\]
\end{prop}
\begin{proof}
The required isomorphism sends $(\underline{\shL},\tilde{K}_{+}\arrdi{\psi}\ISym\underline{\shL})\in\stX_{\phi_{M}}$
to the object of $\MCov$ given by invertible sheaves $(\shL'_{m}=\shL_{m}^{-1})$
and $\psi_{m,n}=\psi(e_{m,n})$.
\end{proof}
We want to prove that the isomorphism \ref{eq:MCov isomorphic to Xphi}
sends $\stZ_{\phi_{M}}$ to $\stZ_{M}$ (see def. \ref{def:the main component of GCov})
and $\stB_{\phi_{M}}$ to $\Bi\Di M$. We need the following classical
result on the structure of a $\Di M$-torsor (see \cite[Proposition 4.1 and 4.6]{SGA3Exp8}):
\begin{prop}
\label{pro:equivalent conditions for a D(M)-torsor}Let $M$ be a
finite abelian group and $P\arr U$ a $\Di M$-equivariant map. Then
$P$ is an fppf $\Di M$-torsor if and only if $P\in\MCov(U)$ and
all the multiplication maps are isomorphisms.
\end{prop}
Now consider the exact sequence   \[   \begin{tikzpicture}[xscale=1.5,yscale=-0.5]     \node (A0_0) at (0, 0) {$0$};     \node (A0_1) at (1, 0) {$K$};     \node (A0_2) at (2, 0) {$\Z^M/<e_0>$};     \node (A0_3) at (3, 0) {$M$};     \node (A0_4) at (4, 0) {$0$};     \node (A1_2) at (2, 1) {$e_m$};     \node (A1_3) at (3, 1) {$m$};     \path (A0_0) edge [->] node [auto] {$\scriptstyle{}$} (A0_1);     \path (A0_2) edge [->] node [auto] {$\scriptstyle{}$} (A0_3);     \path (A0_3) edge [->] node [auto] {$\scriptstyle{}$} (A0_4);     \path (A0_1) edge [->] node [auto] {$\scriptstyle{}$} (A0_2);     \path (A1_2) edge [|->,gray] node [auto] {$\scriptstyle{}$} (A1_3);   \end{tikzpicture}   \] 
\begin{defn}
For $m,n\in M$ we define
\[
v_{m,n}=\phi_{M}(e_{m,n})=e_{m}+e_{n}-e_{m+n}\in K
\]
 and $K_{+}$ as the submonoid of $K$ generated by the $v_{m,n}$.
We will set $x_{m,n}=x^{v_{m,n}}\in\Z[K_{+}]$ and, for $\E\in\duale K_{+}$,
$\E_{m,n}=\E(v_{m,n})$.\end{defn}
\begin{lem}
The map   \[   \begin{tikzpicture}[xscale=2.2,yscale=-0.5]     \node (A0_0) at (0, 0) {$\tilde K_+$};     \node (A0_1) at (1, 0) {$K$};     \node (A1_0) at (0, 1) {$e_{m,n}$};     \node (A1_1) at (1, 1) {$v_{m,n}$};     \path (A0_0) edge [->] node [auto] {$\scriptstyle{}$} (A0_1);     \path (A1_0) edge [|->,gray] node [auto] {$\scriptstyle{}$} (A1_1);   \end{tikzpicture}   \] 
is the associated group of $\tilde{K}_{+}$ and $K_{+}$ is its associated
integral monoid. In particular we have a $2$-cartesian diagram   \[   \begin{tikzpicture}[xscale=2.6,yscale=-1.5]     \node (A0_0) at (0, 0) {$\Spec \Z[K]$};     \node (A0_1) at (1, 0) {$\Spec \Z[K_+]$};     \node (A0_2) at (2, 0) {$\Spec R_M$};     \node (A1_0) at (0, 1) {$\Bi \Di{M}$};     \node (A1_1) at (1, 1) {$\stZ_M$};     \node (A1_2) at (2, 1) {$\MCov$};     \path (A0_0) edge [->] node [auto] {$\scriptstyle{}$} (A0_1);     \path (A0_1) edge [->] node [auto] {$\scriptstyle{}$} (A0_2);     \path (A1_0) edge [->] node [auto] {$\scriptstyle{}$} (A1_1);     \path (A0_2) edge [->] node [auto] {$\scriptstyle{}$} (A1_2);     \path (A1_1) edge [->] node [auto] {$\scriptstyle{}$} (A1_2);     \path (A0_0) edge [->] node [auto] {$\scriptstyle{}$} (A1_0);     \path (A0_1) edge [->] node [auto] {$\scriptstyle{}$} (A1_1);   \end{tikzpicture}   \] \end{lem}
\begin{proof}
Set $x=\prod_{m,n}x_{m,n}$. Since an object $\psi\in\Spec R_{M}(U)$
is a torsor if and only if $\psi_{m,n}\in\odi U^{*}$ for any $m,n$,
it follows that $(\Spec R_{M})_{x}=\Bi\Di M\times_{\MCov}\Spec R_{M}$.
We want to define an inverse to $(R_{M})_{x}\arr\Z[K]$. If $S_{M}$
is the universal algebra over $R_{M}$ and we call $w_{m}$ a graded
basis of $S_{M}$ with $w_{0}=1$, $(S_{M})_{x}$ is an $\Di M$-torsor
over $(R_{M})_{x}$ and so $w_{m}\in(S_{M})_{x}^{*}$ for any $m$.
In particular we can define a group homomorphism   \[   \begin{tikzpicture}[xscale=2.6,yscale=-0.7]     \node (A0_0) at (0, 0) {$\Z^M/<e_0>$};     \node (A0_1) at (1, 0) {$(S_M)_x^*$};     \node (A1_0) at (0, 1) {$e_m$};     \node (A1_1) at (1, 1) {$w_m$};     \path (A0_0) edge [->] node [auto] {$\scriptstyle{}$} (A0_1);     \path (A1_0) edge [|->,gray] node [auto] {$\scriptstyle{}$} (A1_1);   \end{tikzpicture}   \] 
which restricts to a map $K\arr(R_{M})_{x}$ that sends $v_{m,n}$
to $x_{m,n}$. In particular the map $\tilde{K}_{+}\arr K$ defined
in the statement gives the associated group of $\tilde{K}_{+}$ and
has as image exactly $K_{+}$, which means that $K_{+}$ is the integral
monoid associated to $\tilde{K}_{+}$.

In order to conclude the proof it's enough to apply \ref{lem:the domain monoid and group monoid associated to T +: ring}
and \ref{cor:the domain monoid and group monoid associated to T +: stack}.\end{proof}
\begin{cor}
\label{cor:MCov as global quotient}The isomorphism $\stX_{\phi_{M}}\simeq\MCov$
(\ref{eq:MCov isomorphic to Xphi}) induces isomorphisms $\stB_{\phi_{M}}\simeq\Bi\Di M$
and $\stZ_{\phi_{M}}\simeq\stZ_{M}$. In particular $\stZ_{M}$ is
an irreducible component of $\MCov$ and
\[
\Bi\Di M\simeq[\Spec\Z[K]/\Di{\Z^{M}/<e_{0}>}]\text{ and }\stZ_{M}\simeq[\Spec\Z[K_{+}]/\Di{\Z^{M}/<e_{0}>}]
\]

\end{cor}
Note that the induced map $\phi_{M}\colon K\arr\Z^{M}/<e_{0}>$ is
just the inclusion and so it is injective. This means that any result
obtained in section \ref{sec:stack Xphi} applies naturally in the
context of $\Di M$-covers. In particular now we show how we can describe
the objects of $\stF_{\underline{\E}}$, for a sequence of rays in
$\duale{\tilde{K}}_{+}$, in a simpler way.
\begin{prop}
\label{pro:stack of reduced data for M-covers}Let $M\simeq\prod_{i=1}^{n}\Z/l_{i}\Z$
be a decomposition and let $m_{1},\dots,m_{n}$ be the associated
generators. Given $\underline{\E}=\E^{1},\dots,\E^{r}\in\duale K_{+}$
define $\stF_{\underline{\E}}^{\textup{red}}$ as the stack whose
objects over a scheme $X$ are sequences $\underline{\shL}=\shL_{1},\dots,\shL_{n},\underline{\shM}=\shM_{1},\dots,\shM_{r},\underline{z}=z_{1},\dots,z_{r},\underline{\mu}=\mu_{1},\dots,\mu_{n}$
where $\underline{\shL}\comma\underline{\shM}$ are invertible sheaves
over $X$, $z_{i}\in\shM_{i}$ and $\underline{\mu}$ are isomorphisms
\[
\mu_{i}\colon\shL_{i}^{-l_{i}}\arrdi{\simeq}\underline{\shM}^{\underline{\E}(l_{i}e_{m_{i}})}=\shM_{1}^{\E^{1}(l_{i}e_{m_{i}})}\otimes\cdots\otimes\shM_{r}^{\E^{r}(l_{i}e_{m_{i}})}
\]

Then we have an isomorphism of stacks   \[   \begin{tikzpicture}[xscale=4.7,yscale=-0.6]     \node (A0_0) at (0, 0) {$\stF_{\underline{\E}}$};     \node (A0_1) at (1, 0) {$\stF_{\underline{\E}}^{\textup{red}}$};     \node (A1_0) at (0, 1) {$(\underline \shL,\underline \shM, \underline z,\lambda)$};     \node (A1_1) at (1, 1) {$( (\shL_{m_i})_{i=1,\dots,n},\underline \shM, \underline z,(\lambda(l_ie_{m_i}))_{i=1,\dots,n})$};     \path (A0_0) edge [->]node [auto] {$\scriptstyle{}$} (A0_1);     \path (A1_0) edge [|->,gray]node [auto] {$\scriptstyle{}$} (A1_1);   \end{tikzpicture}   \] \end{prop}
\begin{proof}
We want to find $\sigma\comma V\comma v_{1},\dots,v_{q}$ as in \ref{def:stack of reduced data}
such that $\stF_{\underline{\E}}^{\textup{red},\sigma}=\stF_{\underline{\E}}^{\textup{red}}$
and that the map in the statement coincide with the one defined in
\ref{pro:isomorphism with the stack of reduced data}. Set $\delta^{i}\colon M\arr\{0,\dots,l_{i}-1\}$
as the map such that $\pi_{i}(m)=\pi_{i}(\delta_{m}^{i}m_{i})$, where
$\pi_{i}\colon M\arr\Z/l_{i}\Z$, and think it also as a map $\delta^{i}\colon\Z^{M}/<e_{0}>\arr\Z$.
Set $V=\bigoplus_{i=1}^{n}\Z e_{m_{i}}$, $v_{i}=e_{m_{i}}$ and $\sigma\colon\Z^{M}/<e_{0}>\arr V$
as $\sigma(e_{m})=\sum_{i=1}^{n}\delta_{m}^{i}v_{i}$. Clearly $(\id-\sigma)\Z^{M}/<e_{0}>\subseteq K$
and $(\id-\sigma)V=0$. So $W=\sigma K$. We have
\[
\sigma(v_{m,n})=\sum_{i=1}^{n}\delta_{m,n}^{i}v_{i}\in\bigoplus_{i=1}^{n}l_{i}\Z v_{i}
\]
since for any $i$ $\delta_{m,n}^{i}\in\{0,l_{i}\}$. On the other
hand $\sigma(v_{(l_{i}-1)m_{i},m_{i}})=l_{i}v_{i}$. Therefore we
have $W=\bigoplus_{i=1}^{n}l_{i}\Z v_{i}$. It's now easy to check
that all the definitions agree.
\end{proof}
We now want to express the relation between $\MCov$ and the equivariant
Hilbert scheme, that can be defined as follows. Given $\underline{m}=m_{1},\dots,m_{r}\in M$,
so that $\Di M$ acts on $\A_{\Z}^{r}=\Spec\Z[x_{1},\dots,x_{r}]$
with graduation $\deg x_{i}=m_{i}$, we define $\MHilb^{\underline{m}}\colon\Sch^{\textup{op}}\arr\set$
as the functor that associates to a scheme $Y$ the pairs $(X\arrdi fY,j)$
where $X\in\MCov(Y)$ and $j\colon X\arr\A_{Y}^{r}$ is an equivariant
closed immersion over $Y$. Such a pair can be also thought as a coherent
sheaf of algebras $\alA\in\MCov(Y)$ together with a graded surjective
map $\odi Y[x_{1},\dots,x_{r}]\arr\alA$. This functor is proved to
be a scheme of finite type in \cite{Haiman2002}.
\begin{prop}
\label{pro:MHlb --> MCov has irreducible fibers}Let $\underline{m}=m_{1},\dots,m_{r}\in M$.
The forgetful map $\vartheta_{\underline{m}}\colon\MHilb^{\underline{m}}\arr\MCov$
is a smooth zariski epimorphism onto the open substack $\MCov^{\underline{m}}$
of $\MCov$ of sheaves of algebras $\alA$ such that, for any $y\in Y$,
$\alA\otimes k(y)$ is generated in the degrees $m_{1},\dots,m_{r}$
as a $k(y)$-algebra. Moreover $\MHilb^{\underline{m}}$ is an open
subscheme of a vector bundle over $\MCov^{\underline{m}}$.\end{prop}
\begin{proof}
Let $\alA=\oplus_{m\in M}\alA_{m}\in\MCov$ and consider the map 
\[
\eta_{\alA}\colon\Sym(\alA_{m_{1}}\oplus\cdots\oplus\alA_{m_{r}})\arr\alA
\]
It's easy to check that $\eta_{\alA}$ is surjective if and only if
$\alA\in\MCov^{\underline{m}}$. Therefore $\MCov^{\underline{m}}$
is an open substack of $\MCov$ and clearly contains the image of
$\vartheta_{\underline{m}}$. Consider now a cartesian diagram   \[   \begin{tikzpicture}[xscale=2.0,yscale=-1.2]     \node (A0_0) at (0, 0) {$F$};     \node (A0_1) at (1, 0) {$\MHilb^{\underline m}$};     \node (A1_0) at (0, 1) {$T$};     \node (A1_1) at (1, 1) {$\MCov^{\underline m}$};     \path (A0_0) edge [->] node [auto] {$\scriptstyle{}$} (A0_1);     \path (A1_0) edge [->] node [auto] {$\scriptstyle{\alA}$} (A1_1);     \path (A0_1) edge [->] node [auto] {$\scriptstyle{\vartheta_{\underline m}}$} (A1_1);     \path (A0_0) edge [->] node [auto] {$\scriptstyle{}$} (A1_0);   \end{tikzpicture}   \] and
let $U\arrdi{\phi}T$ be a map. $F(U)$ is given by a graded surjection
$\odi U[x_{1},\dots,x_{r}]\arr\alB$ and an isomorphism $\alB\simeq\phi^{*}\alA$.
This is equivalent to giving a graded surjection $\odi U[x_{1},\dots,x_{r}]\arr\phi^{*}\alA$.
In this way we obtain a map
\[
F\arrdi{g_{T}}\prod_{i}\Homsh_{T}(\odi T,\alA_{m_{i}})\simeq\Spec\Sym(\bigoplus_{i}\alA_{m_{i}}^{-1})
\]
 We claim that this is an open immersion. Indeed given $(a_{i})_{i}\colon U\arr\prod_{i}\Homsh_{T}(\odi T,\alA_{m_{i}})$,
the fiber product with $F$ is the locus where the induced graded
map $\odi U[x_{1},\dots,x_{r}]\arr\alA\otimes\odi U$ is surjective,
that is an open subscheme of $U$. In particular $F$ is smooth and
so $\vartheta_{\underline{m}}$ is smooth too. It's easy to check
that it is also a zariski epimorphism. Finally the vector bundle $\shN$
of the statement is defined over any $U\arr\MCov^{\underline{m}}$
given by $\alA=\bigoplus_{m}\alA_{m}$ by $\shN_{|U}=\oplus_{i}\alA_{m_{i}}^{-1}$.\end{proof}
\begin{rem}
\label{rem:relation MHilb MCov}If the sequence $\underline{m}$ contains
any elements of $M-\{0\}$, then $\MCov^{\underline{m}}=\MCov$. Therefore
in this case $\MHilb^{\underline{m}}$ is an atlas for $\MCov$.
\end{rem}

\begin{rem}
\label{rem:relation Mhilbm MCovm}We have cartesian diagrams   \[   \begin{tikzpicture}[xscale=2.5,yscale=-1.2]     
\node (A0_0) at (0, 0) {$W_{\underline m}$};     
\node (A0_1) at (1, 0) {$V_{\underline m}$};    
\node (A0_2) at (2, 0) {$U_{\underline m}$};     
\node (A0_3) at (3, 0) {$\Spec R_M$};     
\node (A1_0) at (0, 1) {$\MHilb^{\underline m}$};     
\node (A1_1) at (1, 1) {$H_{\underline m}$};     
\node (A1_2) at (2, 1) {$\MCov^{\underline m}$};     
\node (A1_3) at (3, 1) {$\MCov$};
\node (o) at (0.5, 0.4) {$\scriptstyle{\textup{open}}$};
\node (i) at (0.5, 0.6) {$\scriptstyle{\textup{immersions}}$};
\node (o2) at (2.5, 0.4) {$\scriptstyle{\textup{open}}$};
\node (i2) at (2.5, 0.6) {$\scriptstyle{\textup{immersions}}$};
\node (tor1) at (-0.5, 0.35) {$\scriptstyle{\Di{\Z^M/<e_0>}}$};
\node (tor2) at (-0.5, 0.65) {$\scriptstyle{\textup{torsors}}$};
\node (v) at (1.5, 0.4) {$\scriptstyle{\textup{vector}}$};
\node (b) at (1.5, 0.6) {$\scriptstyle{\textup{bundles}}$};

\path (A0_1) edge [->]node [auto] {$\scriptstyle{}$} (A1_1);     \path (A0_0) edge [->]node [auto] {$\scriptstyle{}$} (A0_1);     \path (A0_1) edge [->]node [auto] {$\scriptstyle{}$} (A0_2);     \path (A1_0) edge [->]node [auto] {$\scriptstyle{}$} (A1_1);     \path (A0_3) edge [->]node [auto] {$\scriptstyle{}$} (A1_3);     \path (A0_2) edge [->]node [auto] {$\scriptstyle{}$} (A1_2);     \path (A1_1) edge [->]node [auto] {$\scriptstyle{}$} (A1_2);     \path (A0_0) edge [->]node [auto] {$\scriptstyle{}$} (A1_0);     \path (A0_2) edge [->]node [auto] {$\scriptstyle{}$} (A0_3);     \path (A1_2) edge [->]node [auto] {$\scriptstyle{}$} (A1_3);   \end{tikzpicture}   \]  In particular, since $\Bi\Di M\subseteq\MCov^{\underline{m}}$, we
can conclude that $\vartheta_{\underline{m}}^{-1}(\stZ_{M})$ is the
main irreducible component of $\MHilb^{\underline{m}}$. Moreover
the above diagram shows that $\MHilb^{\underline{m}}$ and $\MCov^{\underline{m}}$,
as well as their main irreducible components, share many properties
like smoothness, connection, integrality, reducibility.
\end{rem}
We want now study some geometrical properties of the stack $\MCov$
and, therefore, of the equivariant Hilbert schemes.
\begin{rem}
\label{rem:reduced relations for RM}The ring $R_{M}$ can be written
as quotient of the ring $\Z[x_{m,n}]_{(m,n)\in J}$, where $J$ is
$\{(m,n)\in M^{2}\st m,n,m+n\neq0\}$ divided by the equivalence relation
$(m,n)\sim(n,m)$, by the ideal
\[
I=\left(\begin{array}{c}
x_{m,n}x_{m+n,t}-x_{n,t}x_{n+t,m}\text{ with }m,n,t,m+n,n+t,m+n+t\neq0\text{ and }m\neq t,\\
x_{-m,t}x_{-m+t,m}-x_{-m,s}x_{-m+s,m}\text{ with }m,s,t\neq0\text{ and distinct}
\end{array}\right)
\]
Indeed the first relations are trivial when one of $m,n,t$ is zero
or $m=t$, while if $m+n=0$ yield relations $x_{m,-m}=x_{-m,t}x_{-m+t,m}$.
Using these last relations we can remove all the variable $x_{m,n}$
with $0\in\{m,n,m+n\}$.
\end{rem}

\begin{rem}
\label{rem:N-graduation of RM}There exist a map $f\colon\tilde{K}_{+}\arr\N$
such that for any $m,n\neq0$ we have $f(e_{m,n})=1$ if $m+n\neq0$,
$f(e_{m,-m})=2$ otherwise. In particular $f(v)=0$ only if $v=0$.
Moreover $f$ induces an $\N$-graduation on both $(R_{M}\otimes A)$
and $\Z[K_{+}]\otimes A$, where $A$ is a ring, such that the degree
zero part is $A$ and that the elements $x_{m,n}$ with $m+n\neq0$
are homogeneous of degree $1$. $f$ is obtained as composition $\tilde{K}_{+}\arr K\subseteq\Z^{M}/<e_{0}>\arrdi h\Z$,
where $h(e_{m})=1$ if $m\neq0$.
\end{rem}
One of the open problems in the theory of equivariant Hilbert schemes
is whether those schemes are connected. As said above $\MHilb^{\underline{m}}$
is connected if and only if $\MCov^{\underline{m}}$ is so. What we
can say here is:
\begin{thm}
\label{thm:Mcov geom connected}The stack $\MCov$ is connected with
geometrically connected fibers. If $M-\{0\}\subseteq\underline{m}$,
then $\MHilb^{\underline{m}}$ has the same properties.\end{thm}
\begin{proof}
It's enough to prove that $\Spec R_{M}\otimes k$ is connected for
any field $k$. But $R_{M}\otimes k$ has an $\N$-graduation such
that $(R_{M}\otimes k)_{0}=k$ by \ref{rem:N-graduation of RM} and
it is a general fact that such an algebra doesn't contain non trivial
idempotents.
\end{proof}
We now want to discuss the problem of the reducibility of $\MCov$.
\begin{defn}
\label{def:universally reducible}Let $X$ be a scheme over a base
scheme $S$. $X$ is said universally reducible over $S$ if for any
base change $S'\arr S$ the scheme $X\times_{S}S'$ is reducible.
A scheme is universally reducible if it is so over $\Z$.\end{defn}
\begin{rem}
It's easy to check that $X$ is universally reducible over $S$ if
and only if all the fibers are reducible.\end{rem}
\begin{lem}
\label{lem:necessary condition MCov reducible}If there exist $m,n,t,a\in M$
such that
\begin{enumerate}
\item $m,n,t$ are distinct and not zero;
\item $a\neq0,m,n,t,m-n,n-m,n-t,t-n,m-t,2m-t,2n-t,m+n-t,m+n-2t$;
\item $2a\neq m+n-t$;
\end{enumerate}
then $\Spec R_{M}$ is universally reducible.\end{lem}
\begin{proof}
Let $k$ be a field and $I=(\underline{x}^{\alpha_{i}}-\underline{x}^{\beta_{i}})$
be an ideal of $k[x_{1},\dots,x_{r}]=k[\underline{x}]$. We will say
that $\alpha\in\N^{r}$ is transformable (with respect to $I$) if
there exists $i$ such that $\alpha_{i}\leq\alpha$ or $\beta_{i}\leq\alpha$.
Here by $\alpha\leq\beta\in\N^{r}$ we mean $\alpha_{j}\leq\beta_{j}$
for all $j$. A direct computation shows that if $\underline{x}^{\alpha}-x^{\beta}\in I$
and $\alpha\neq\beta$, then both $\alpha$ and $\beta$ are transformable. 

We will use the above notation for the ideal $I$ defining $R_{M}\otimes k$
as in \ref{rem:reduced relations for RM}. In particular the elements
$\alpha_{i},\beta_{i}\in\N^{J}$ associated to the ideal $I$ are
of the form $e_{u,v}+e_{u+v,w}$ with $u,v,u+v,w,u+v+w\neq0$. 

Set $\mu=\prod_{m,n}x_{m,n}$. Since $R_{M}\otimes k\arr k[K_{+}]\subseteq k[K]=(R_{M}\otimes k)_{\mu}$,
there exists $N>0$ such that $P=\Ker(R_{M}\otimes k\arr k[K_{+}])=\Ann\mu^{N}$.
Our strategy will be to find an element of $P$ which is not nilpotent.
Since $P$ is a minimal prime, being $\Spec k[K_{+}]$ an irreducible
component of $\Spec R_{M}\otimes k$, it follows that $R_{M}\otimes k$
is reducible. Now consider $\alpha=e_{a,m-a}+e_{m+n-t-a,t+a-m}+e_{t+a-n,n-a}$,
$\beta=e_{m+n-t-a,t+a-n}+e_{a,n-a}+e_{m-a,t+a-m}\in\N^{J}$ and $z=\underline{x}^{\alpha}-\underline{x}^{\beta}$.
We will show that $\mu z=0$, i.e. $z\in P$ and that $z$ is not
nilpotent. First of all note that $z$ is well defined since for any
$e_{u,v}$ in $\alpha$ or $\beta$ we have $u,v\neq0$ and $0\neq u+v\in\{m,n,t\}$
thanks to $1)$, $2)$. Let $S_{M}$ be the universal algebra over
$R_{M}$, i.e. $S_{M}=\bigoplus_{m\in M}R_{M}v_{m}$ with $v_{m}v_{n}=x_{m,n}v_{m+n}$
and $v_{0}=1$. By construction we have
\begin{alignat*}{1}
(v_{a}v_{m-a})(v_{m+n-t-a}v_{t+a-m})(v_{t+a-n}v_{n-a}) & =\underline{x}^{\alpha}v_{m}v_{n}v_{t}=\\
(v_{m+n-t-a}v_{t+a-n})(v_{a}v_{n-a})(v_{m-a,t+a-m}) & =\underline{x}^{\beta}v_{m}v_{n}v_{t}
\end{alignat*}
So $\underline{x}^{\alpha}x_{m,n}x_{m+n,t}v_{m+n+t}=\underline{x}^{\beta}x_{m,n}x_{m+n,t}v_{m+n+t}$
and therefore $z\mu=0$, i.e. $z\in P$.

Now we want to prove that any linear combination $\gamma=a\alpha+b\beta\in\N^{J}$
with $a,b\in\N$ is not transformable. First remember that each $e_{u,v}$
in $\gamma$ is such that $u+v\in\{m,n,t\}$. If we will have $e_{u,v}+e_{u+v,w}\leq\gamma$
then there must exist $e_{i,j}\leq\gamma$ such that $i\in\{m,n,t\}$
or $j\in\{m,n,t\}$. Condition $2)$ is exactly what we need to avoid
this situation and can be written as $\{a,m-a,m+n-t-a,t+a-m,t+a-n,n-a\}\cap\{m,n,t\}=\emptyset$.

In particular, if we think $\tilde{K}_{+}$ as a quotient of $\N^{J}$,
we have that $a\alpha+b\beta=a'\alpha+b'\beta$ in $\tilde{K}_{+}$
if and only if they are equals in $\N^{J}$. Assume for a moment that
$\alpha\neq\beta$ in $\N^{J}$. Clearly this means that $\alpha$
and $\beta$ are $\Z$-independent in $\Z^{J}$. Since any linear
combination of $\alpha$ and $\beta$ is not transformable, it follows
that $\underline{x}^{\alpha},\underline{x}^{\beta}$ are algebraically
independent over $k$ in $R_{M}\otimes k$ and, in particular, that
$z=\underline{x}^{\alpha}-\underline{x}^{\beta}$ cannot be nilpotent.
So it remains to prove that $\alpha\neq\beta$ in $\N^{J}$. Note
that for any $i\in\{m,n,t\}$ there exists only one $e_{u,v}$ in
$\alpha$ such that $u+v=i$ and the same happens for $\beta$. So,
if $\alpha=\beta$ and since $m,n,t$ are distinct, those terms have
to be equal, for instance $e_{a,m-a}=e_{m+n-t-a,t+a-n}$. But $a\neq m+n-t-a$
by $3)$, while $a\neq t+a-n$ since $t\neq n$. Therefore $\alpha\neq\beta$.\end{proof}
\begin{cor}
\label{cor:Mcov reducible}If $|M|>7$ and $M\not\simeq(\Z/2\Z)^{3}$
then $\MCov$ is universally reducible and the same holds for $\MHilb^{\underline{m}}$
provided that $\underline{m}$ contains any element of $M-\{0\}$.\end{cor}
\begin{proof}
We have to show that $R_{M}$ is universally reducible and so we will
apply \ref{lem:necessary condition MCov reducible}. If $M=C\times T$,
where $C$ is cyclic with $|C|\geq4$ and $T\neq0$ we can choose:
$m$ a generator of $C$, $n=3m$, $t=2m$ and $a\in T-\{0\}$. If
$M$ cannot be written as above, there are four remaining cases. $1)$
$M\simeq\Z/8\Z$: choose $m=2\comma n=4\comma t=6\comma a=1$. $2)$
$M$ cyclic with $|M|>8$ and $|M|\neq10$: choose $m=1\comma n=2\comma t=3\comma a=5$.
$3)$ $M\simeq(\Z/2\Z)^{l}$ with $l\geq4$: choose $m=e_{1}\comma n=e_{2}\comma t=e_{3}\comma a=e_{4}$.
$4)$ $M\simeq(\Z/3\Z)^{l}$ with $l\geq2$: choose $m=e_{1}\comma n=2e_{1}\comma t=e_{2}\comma a=m+t=e_{1}+e_{2}$.\end{proof}
\begin{prop}
\label{pro:smooth DMCov}$\MCov$ is smooth if and only if $\stZ_{M}$
is so. This happens if and only if $M\simeq\Z/2\Z,\Z/3\Z,\Z/2\Z\times\Z/2\Z$
and in this case $\MCov=\stZ_{M}$. To be more precise $R_{M}=\Z[x_{m,n}]_{(m,n)\in J}$,
where $J$ is the set defined in \ref{rem:reduced relations for RM}.

In particular $\MHilb^{\underline{m}}$ is smooth and irreducible
for any sequence $\underline{m}$ if $M$ is as above. Otherwise,
if $M-\{0\}\subseteq\underline{m}$, $\MHilb^{\underline{m}}$ is
not smooth.\end{prop}
\begin{proof}
Let $k$ be a field. Note that 
\[
\MCov\text{ smooth}\iff R_{M}\text{ smooth}\then\stZ_{M}\text{ smooth }\then k[K_{+}]/k\text{ smooth}
\]
We first prove that if $k[K_{+}]$ is smooth then $M$ has to be one
of the groups of the statement. We have $K_{+}\simeq\N^{r}\oplus\Z^{s}$
and therefore $k[K_{+}]$ is UFD. We will consider $k[K_{+}]$ endowed
with the $\N$-graduation defined in \ref{rem:N-graduation of RM}.
Since any of the $x_{m,n}$ has degree $1$, it is irreducible and
so prime. If we have a relation $x_{m,n}x_{m+n,t}=x_{n,t}x_{n+t,m}$
with $m,n,t,m+n,n+t,m+n+t\neq0$ and $m\neq t$, then $x_{m,n}\mid x_{n,t}x_{n+t,m}$
implies that $x_{m,n}=x_{n,t}$ or $x_{m,n}=x_{n+t,m}$, which is
impossible thanks to our assumptions. We will prove that if $M$ is
not isomorphic to one of the group in the statement, then such a relation
exists. Clearly it is enough to find this relation in a subgroup of
$M$. So it is enough to consider the following cases. $1)$ $M$
cyclic with $|M|\geq5$: choose $m=n=1\comma t=2$. $2)$ $M\simeq\Z/4\Z$:
choose $m=1\comma n=2\comma t=3$. $3)$ $M\simeq(\Z/2\Z)^{3}$: choose
$m=e_{1}\comma n=e_{2}\comma t=e_{3}$. $4)$ $M\simeq(\Z/3\Z)^{2}$:
choose $m=n=e_{1}\comma t=e_{2}$.

We now want to prove that when $M$ is as in the statement, then the
ideal $I$ of \ref{rem:reduced relations for RM} is zero. If we have
a relation as in the first row, since $m\neq t$ we have $|M|\geq3$.
If $M\simeq\Z/3\Z$ then $t=2m$ and $m+t=0$. If $M\simeq(\Z/2\Z)^{2}$,
if $m,n,t$ are distinct then $m+n+t=0$, otherwise $m=n$ and $m+n=0$.
If we have a relation as in the second row, since $m,t,s$ are distinct,
we must have $M\simeq(\Z/2\Z)^{2}$. Therefore $m+t=s$ and the relation
become trivial.\end{proof}
\begin{cor}
$\GrCov{\Z/2\Z\times\Z/2\Z}$ is isomorphic to the stack of sequences
$(\shL_{i},\psi_{i})_{i=1,2,3}$, where $\shL_{1},\shL_{2},\shL_{3}$
are invertible sheaves and $\psi_{1}\colon\shL_{2}\otimes\shL_{3}\arr\shL_{1}$,
$\psi_{2}\colon\shL_{1}\otimes\shL_{3}\arr\shL_{2}$, $\psi_{3}\colon\shL_{1}\otimes\shL_{2}\arr\shL_{3}$
are maps.\end{cor}
\begin{proof}
Set $M=(\Z/2\Z)^{2}$. Thanks to \ref{pro:smooth DMCov}, we know
that $\tilde{K}_{+}=K_{+}\simeq\N v_{e_{1},e_{2}}\oplus\N v_{e_{1},e_{1}+e_{2}}\oplus\N v_{e_{2},e_{1}+e_{2}}$.
So an object of $\MCov$ is given by invertible sheaves $\shL_{1}=\shL_{e_{1}}\comma\shL_{2}=\shL_{e_{2}}\comma\shL_{3}=\shL_{e_{1}+e_{2}}$
and maps $\psi_{1}=\psi_{e_{2},e_{1}+e_{2}}\comma\psi_{2}=\psi_{e_{1},e_{1}+e_{2}}\comma\psi_{3}=\psi_{e_{1},e_{2}}$.\end{proof}
\begin{rem}
\label{rem:MCov for M=00003DZfour}$\GrCov{\Z/4\Z}$ and $\Z/4\Z\textup{-Hilb}^{\underline{m}}$,
for any sequence $\underline{m}$, are integral and normal since one
can check directly that $R_{\Z/4\Z}=\Z[x_{1,2},x_{3,3},x_{2,3},x_{1,1}]/(x_{1,2}x_{3,3}-x_{2,3}x_{1,1})$.
I'm not able to prove that $\MCov$ is irreducible when $M$ is one
of $\Z/5\Z\comma\Z/6\Z$, $\Z/7\Z\comma(\Z/2\Z)^{3}$. Anyway the
first two cases seems to be integral thanks to a computer program,
while for the last ones there are some techniques that can be used
to study this problem but they are too complicated to be explained
here.
\end{rem}

\subsection{The invariant $h\colon|\MCov|\arr\N$.}

In this subsection we investigate the local structure of a $\Di M$-cover,
especially over a local ring. In particular we will be able to define
an upper semicontinuous map $h\colon|\MCov|\arr\N$.
\begin{notation}
Given a ring $A$, we will write $B\in\Spec R_{M}(A)$ meaning that
$B$ is an $M$-graded $A$-algebra with a given $M$-graded basis,
usually denoted by $\{v_{m}\}_{m\in M}$ with $v_{0}=1$, and a given
multiplication $\psi$ such that
\[
B=\bigoplus_{m\in M}Av_{m}\comma v_{m}v_{n}=\psi_{m,n}v_{m+n}
\]
\end{notation}
\begin{lem}
\label{lem:factorization of covers through torsors on local rings}Let
$A$ be a ring and $B\in\Spec R_{M}(A)$, with graded basis $v_{m}$
and multiplication map $\psi$. Then the set
\[
H_{\psi}=H_{B/A}=\{m\in M\;|\; v_{m}\in B^{*}\}=\{m\in M\;|\;\psi_{m,-m}\in A^{*}\}
\]
 is a subgroup of $M$. Moreover if $m,n\in M$ and $h\in H_{\psi}$
then $\psi_{m,n}$ and $\psi_{m,n+h}$ differs by an element of $A^{*}$.
If $H<H_{\psi}$ then $C=\bigoplus_{m\in H}Av_{m}\in\Bi\Di H(A)$.
Moreover if $\sigma\colon M/H\arr M$ gives representatives of $M/H$
in $M$ and we set $w_{m}=v_{\sigma(m)}$ for $m\in M/H$ we have
\[
B=\bigoplus_{m\in M/H}Cw_{m}\in\Spec R_{M/H}(C)
\]
 Finally if we denote by $\psi'$ the induced multiplication on $B$
over $C$ we have $H_{\psi'}=H_{\psi}/H$ and for any $m,n\in M$
$\psi'_{m,n}$ and $\psi_{m,n}$ differ by an element of $C^{*}$.\end{lem}
\begin{proof}
From the relations $v_{m}v_{-m}=\psi_{m,-m}$, $v_{m}^{|M|-1}=\lambda v_{-m}$,
$v_{m}^{|M|}=\lambda\psi_{m,-m}$ and $v_{m}v_{n}=\psi_{m,n}v_{m+n}$
we see that $v_{m}\in B^{*}\iff\psi_{m,-m}\in A^{*}$ and that $H_{\psi}<M$.
From \ref{eq:condition on psi} we get the relations $\psi_{-h,h}=\psi_{h,u}\psi_{h+u,-h}$
and $\psi_{m,n}\psi_{m+n,h}=\psi_{n,h}\psi_{m,n+h}$. So if $h\in H$
then $\psi_{h,u}\in A^{*}$ for any $u$ and $\psi_{m,n}$ and $\psi_{m,n+h}$
differ by an element of $A^{*}$.

Now consider the second part of the statement. From \ref{pro:equivalent conditions for a D(M)-torsor}
we know that $C$ is a torsor over $A$. Since for any $m$ we have
$v_{m}=(\psi_{h,m}/v_{h})v_{\sigma(\overline{m})}$, where $h=\sigma(\overline{m})-m\in H$
we obtain the writing of $B$ as $M/H$ graded $C$-algebra and that
\[
\psi'_{m,n}=\psi_{\sigma(m),\sigma(n)}(\psi_{h,\sigma(m)+\sigma(n)}/v_{h})\text{ where }h=\sigma(m+n)-\sigma(m)-\sigma(n)
\]
From the above equation it is easy to conclude the proof.\end{proof}
\begin{defn}
\label{def:of the maximal torsor and the subgroup associated}Given
a ring $A$ and $B\in\Spec R_{M}(A)$ we continue to use the notation
$H_{B/A}$ introduced in \ref{lem:factorization of covers through torsors on local rings}
and we will call the algebra $C$ obtained in it setting $H=H_{B/A}$
the \emph{maximal torsor} of the extension $B/A$. If $\E\in\duale K_{+}$
we will write $H_{\E}=H_{B/k}$ where $B$ is the algebra induced
by the multiplication $0^{\E}$ and $k$ is any field. In particular
\[
H_{\E}=\{m\in M\st\E_{m,-m}=0\}
\]

Finally if $f\colon X\arr Y\in\MCov(Y)$ and $q\in Y$ we define $\shH_{f}(q)=H_{\odi{X,q}/\odi{Y,q}}$.\end{defn}
\begin{prop}
\label{pro:the H* is well defined as a map from MCov}We have a map
  \[   \begin{tikzpicture}[xscale=3.5,yscale=-0.5]     \node (A0_0) at (0, 0) {$|\MCov|$};     \node (A0_1) at (1, 0) {$\{\text{subgroups of }M\}$};     \node (A1_0) at (0, 1) {$B/k$};     \node (A1_1) at (1, 1) {$H_{B/k}$};     \path (A0_0) edge [->]node [auto] {$\scriptstyle{\shH}$} (A0_1);     \path (A1_0) edge [|->,gray]node [auto] {$\scriptstyle{}$} (A1_1);   \end{tikzpicture}   \] such
that, if $Y\arrdi u\MCov$ is given by $X\arrdi fY$, then $\shH_{f}=\shH\circ u$.\end{prop}
\begin{proof}
It's enough to note that if $A$ is a local ring, $B\in\MCov(A)$
is given by multiplications $\psi$ and $\pi\colon A\arr A/m_{A}\arr k$
is a morphism, where $k$ is a field, then $\psi_{m,-m}\in A^{*}\iff\pi(\psi_{m,-m})\neq0$.\end{proof}
\begin{rem}
Let $(A,m_{A})$ be a local ring and $B\in\Spec R_{M}(A)$ with $M$-graded
basis $\{v_{m}\}_{m\in M}$. Then $H_{B/A}=\shH_{B/A}(m_{A})$. If
$H_{B/A}=0$ then $B$ is local with maximal ideal
\[
m_{B}=m_{A}\oplus\bigoplus_{m\in M-\{0\}}Av_{m}
\]
and residue field $B/m_{B}=A/m_{A}$. In particular $m_{B}/m_{B}^{2}$
is $M$-graded.\end{rem}
\begin{lem}
\label{lem:equivalent condition for an algebra to be generated in degrees m1,...,mr}Let
$A$ be a local ring and $B=\bigoplus_{m\in M}Av_{m}\in\MCov(A)$
such that $H_{B/A}=0$. If $m_{1},\dots,m_{r}\in M$ then $B$ is
generated in degrees $m_{1},\dots,m_{r}$ as an $A$-algebra if and
only if $m_{B}=(m_{A},v_{m_{1}},\dots,v_{m_{r}})_{B}$.\end{lem}
\begin{proof}
We can write $m_{B}=m_{A}\oplus\bigoplus_{m\in M-\{0\}}Av_{m}$. Denote
$\underline{v}=v_{m_{1}},\dots,v_{m_{r}}$ and $\pi(\alpha)=\sum_{i}\alpha_{i}m_{i}$
for $\alpha\in\N^{r}$. The only if follows since given $l\in M-\{0\}$
there exists a relation of the form $v_{l}=\mu\underline{v}^{\alpha}$
with $\mu\in A^{*}$ and $\alpha\neq0$ and so $v_{l}\in(m_{A},v_{m_{1}},\dots,v_{m_{r}})_{B}$.
For the converse note that, given $l\in M-\{0\}$, $v_{l}\in m_{B}=(m_{A},v_{m_{1}},\dots,v_{m_{r}})$
means that we have a relation $v_{l}=\lambda v_{l'}v_{m_{i}}$ for
some $i$, $\lambda\in A^{*}$ and $l'=l-m_{i}$. Moreover $v_{l}\notin A[\underline{v}]$
implies that $v_{l'}\notin A[\underline{v}]$ and $l'\neq0$. If,
by contradiction, we have such an element $l$ we can write $v_{l}=\mu v_{n_{1}}\cdots v_{n_{s}}$
with $n_{i}\in M-\{0\}$ and $s\geq|M|^{2}$. In particular there
must exist $i$ such that $m=n_{i}$ appears at least $|M|$ times
in this product. So $m_{A}\ni v_{m}^{|M|}\mid v_{l}$ and $v_{l}\in m_{A}B$,
which is not the case.
\end{proof}
Assume to have a cover $X\arrdi fY\in\MCov(Y)$. We want to define,
for any $m\in M$ a map $h_{f,m}=h_{X/Y,m}\colon Y\arr\{0,1\}$. Let
$q\in Y$ and denote by $C$ the 'maximal torsor' of $\odi{X,q}/\odi{Y,q}$
(see \ref{def:of the maximal torsor and the subgroup associated}).
Also let $p\in f^{-1}(q)$ and set $p_{C}=p\cap C$. We know that
$(\odi{X,q})_{p}=(\odi{X,q})_{p_{C}}=B$ and that $B\in\GrCov{M/\shH_{f}(q)}(C_{p_{C}})$
with $H_{B/C_{p_{C}}}=0$. If we denote by $\overline{m}$ the image
of $m$ in $M/\shH_{f}(q)$ we can define:
\begin{defn}
With notation above we set 
\[
h_{f,m}(q)=\left\{ \begin{array}{cc}
0 & \text{ if }m\in\shH_{f}(q)\\
\dim_{C_{p_{C}}/p_{C}}(m_{B}/m_{B}^{2})_{\overline{m}} & \text{otherwise}
\end{array}\right.
\]
We also set 
\[
h_{f}(q)=\dim_{C_{p_{C}}/p_{C}}(m_{B}/m_{B}^{2})-\dim_{C_{p_{C}}/p_{C}}(m_{B}/m_{B}^{2})_{0}=(\sum_{m\in M}h_{f,m}(q))/|\shH_{f}(q)|
\]
If $\E\in\duale K_{+}$ we set $h_{\E,m}=h_{f,m}\comma h_{\E}=h_{f}\in\N$
where $f$ is the cover $\Spec A\arr\Spec k$ and $A$ is the algebra
given by multiplication $0^{\E}$ over some field $k$.\end{defn}
\begin{lem}
\label{lem:showing that hm and h are well defined}Let $(A,m_{A})$
be a local ring, $B\in\MCov(A)$ given by multiplication $\psi$ and
$t\in M$. Then $h_{B/A,t}=h_{B/A,t}(m_{A})$ is well defined and
$h_{B/A,t}=1$ if and only if $t\notin H_{B/A}$ and for any $u,n\in M-H_{B/A}$
such that $u+n\equiv t\text{ mod }H_{A/B}$ we have $\psi_{u,n}\notin A^{*}$.\end{lem}
\begin{proof}
Let $C$ be the maximal torsor of the extension $B/A$ and $p$ be
a maximal prime of it. We use notation from \ref{lem:factorization of covers through torsors on local rings}.
For any $l\in M-H_{B/A}$ we have a surjective map
\[
k(p)=(m_{B_{p}}/pC_{p})_{\overline{l}}\arr(m_{B_{p}}/m_{B_{p}}^{2})_{\overline{l}}
\]
and so $\dim_{k(p)}(m_{B_{p}}/m_{B_{p}}^{2})_{\overline{l}}\in\{0,1\}$,
where $\overline{l}$ is the image of $l$ through the projection
$M\arr M/H_{A/B}$. If we prove the last part of the statement clearly
we will also have that $h_{B/A,t}$ is well defined. If $t\in H_{B/A}$
then $h_{B/A,t}=0$, while if there exist $u,n$ as in the statement
such that $\psi_{u,n}\in A^{*}$, then $w_{\overline{t}}\in C_{p}^{*}w_{\overline{u}}w_{\overline{n}}\subseteq m_{B_{p}}^{2}$
and again $h_{B/A,t}=0$. On the other hand if $h_{B/A,t}=0$ and
$t\notin H_{B/A}$ then $w_{\overline{t}}\in m_{B_{p}}^{2}$ and therefore
we have a writing
\[
w_{\overline{t}}=bx+\sum_{\overline{u},\overline{n}\neq0}b_{\overline{u},\overline{n}}w_{\overline{u}}w_{\overline{n}}\text{ with }b,b_{\overline{u},\overline{n}}\in B_{p},x\in m_{C_{p}}
\]
The second sum splits as a sum of products of the form $c_{s,\overline{u},\overline{n}}w_{s}w_{\overline{u}}w_{\overline{n}}$
with $s+\overline{u}+\overline{n}=\overline{t}$ and $c_{s,\overline{u},\overline{n}}\in C_{p}$.
Since $C_{p}$ is local, one of these monomials generates $C_{p}w_{\overline{t}}$.
In this case, if $s+\overline{u}=0$ then $\overline{u}\in H_{B_{p}/C_{p}}=0$
which is not the case. So we have a writing 
\[
w_{\overline{t}}=\lambda w_{\overline{u}}w_{\overline{n}}=\lambda\psi'_{\overline{u},\overline{n}}w_{\overline{t}}\then\psi'_{\overline{u},\overline{n}}\in C_{p}^{*}
\]
where $\overline{u},\overline{n}\neq0$ and $\overline{u}+\overline{n}=\overline{t}$.
Since $\psi'_{\overline{u},\overline{n}}$ and $\psi_{u,n}$ differs
by an element of $C^{*}$ thanks to \ref{lem:factorization of covers through torsors on local rings},
it follows that $\psi_{u,n}\in A^{*}$.\end{proof}
\begin{prop}
We have maps   \[   \begin{tikzpicture}[xscale=3.0,yscale=-0.5]     \node (A0_0) at (0, 0) {$|\MCov|$};     \node (A0_1) at (1, 0) {$\{0,1\}$};     \node (A0_2) at (2, 0) {$|\MCov|$};     \node (A0_3) at (3, 0) {$\N$};     \node (A1_0) at (0, 1) {$B/k$};     \node (A1_1) at (1, 1) {$h_{B/k,m}$};     \node (A1_2) at (2, 1) {$B/k$};     \node (A1_3) at (3, 1) {$h_{B/k}$};     \path (A0_0) edge [->]node [auto] {$\scriptstyle{h_m}$} (A0_1);     \path (A1_0) edge [|->,gray]node [auto] {$\scriptstyle{}$} (A1_1);     \path (A0_2) edge [->]node [auto] {$\scriptstyle{h}$} (A0_3);     \path (A1_2) edge [|->,gray]node [auto] {$\scriptstyle{}$} (A1_3);   \end{tikzpicture}   \] 
such that, if $Y\arrdi u\MCov$ is given by $X\arrdi fY$, then $h_{f,m}=h_{m}\circ u$
and $h_{f}=h\circ g$.\end{prop}
\begin{proof}
Taking into account \ref{lem:showing that hm and h are well defined}
and \ref{pro:the H* is well defined as a map from MCov}, it's enough
to note that if $A$ is a local ring, $B\in\MCov(A)$ is given by
multiplications $\psi$ and $\pi\colon A\arr A/m_{A}\arr k$ is a
morphism, where $k$ is a field, then $\psi_{u,v}\in A^{*}\iff\pi(\psi_{u,v})\neq0$
and $H_{B/A}=H_{B\otimes_{A}k/k}$.\end{proof}
\begin{cor}
\label{cor:An algebra with H =00003D 0 is generated in the degrees where h=00003D1}Under
the hypothesis of \ref{lem:equivalent condition for an algebra to be generated in degrees m1,...,mr},
$\{m\in M\;|\; h_{B/A,m}=1\}$ is the minimum subset of $M$ such
that $B$ is generated as an $A$-algebra in those degrees. In particular
$B$ is generated in $h_{B/A}$ degrees.\end{cor}
\begin{prop}
Let $(A,m_{A})$ be a local ring, $B\in\MCov(A)$ and $C$ the maximal
torsor of $B/A$. Then
\[
h_{B/A}(m_{A})=\dim_{k(p)}\Omega_{B/C}\otimes_{B}k(p)
\]
for any maximal prime $p$ of $B$. In particular if $(|H_{B/A}|,\car A/m_{A})=1$
we also have $h_{B/A}(m_{A})=\dim_{k(p)}\Omega_{B/A}\otimes_{B}k(p)$
for any maximal prime $p$ of $B$.\end{prop}
\begin{proof}
If $A$ is any ring and $B\in\MCov(A)$ is given by basis $\{v_{m}\}_{m\in M}$
and multiplication $\psi$ one sees from the universal property that
\[
\Omega_{B/A}=B^{M}/<e_{0},v_{n}e_{m}+v_{m}e_{n}-\psi_{m,n}e_{m+n}>
\]
Now consider $B\in\GrCov{M/H}(C)$, where $H=H_{B/A}$ and let $p$
be a maximal prime of $B$. Following the notation of \ref{lem:factorization of covers through torsors on local rings},
we have that $w_{m}\in p$ for any $m\in M/H-\{0\}$ and $\psi_{m,n}'\in p\iff\psi_{m,n}\in m_{A}$.
So $\Omega_{B/C}\otimes_{B}k(p)$ is free on the $e_{m}$ for $m\in M/H-\{0\}$
such that for any $u,n\in M/H-\{0\}$, $u+n=m$ implies $\psi_{u,n}\notin A^{*}$,
that are exactly $h_{B/A}(m_{A})$ thanks to \ref{lem:showing that hm and h are well defined}.\end{proof}
\begin{cor}
The function $h$ is upper semicontinuous.\end{cor}
\begin{proof}
Let $X\arrdi fY$ be an $\Di M$-cover and $q\in Y$. Set $r=h_{f}(q)$
and $H=\shH_{f}(q)$. We can assume that $Y=\Spec A$, $X=\Spec B$
with graded basis $\{v_{m}\}_{m\in M}$ and multiplication $\psi$
and that $\psi_{m,-m}\in A^{*}$ for any $m\in H$. Set $C=A[v_{m}]_{m\in H}$.
$C_{q}$ is the maximal torsor of $B_{q}/A_{q}$ and so, if $p\in X$
is a point over $q$, we have $r=\dim_{k(p)}\Omega_{B/C}\otimes_{B}k(p)$.
Finally let $U\subseteq X$ be an open neighborhood of $p$ such that
$\dim_{k(p')}\Omega_{B/C}\otimes_{B}k(p')\leq r$ for any $p'\in U$
and $V=f(U)$. We want to prove that $h\leq r$ on $V$. Indeed given
$q'=f(p')\in V$, if $D$ is the maximal torsor of $B_{q'}/A_{q'}$,
we have $C_{q'}\subseteq D\subseteq B_{q'}$. So 
\[
h_{f}(q')=\dim_{k(p')}\Omega_{B_{q'}/D}\otimes_{B_{q'}}k(p')\leq\dim_{k(p')}\Omega_{B_{q'}/C_{q'}}\otimes_{B_{q'}}k(p')\leq r
\]
\end{proof}
\begin{rem}
\label{rem:The 0 section and the 0 algebras}The $0$ section $R_{M}\arr\Z$
induces a closed immersion
\[
\Picsh^{|M|-1}\simeq\Bi\shT=[\Spec\Z/\shT]\subseteq[\Spec R_{M}/\shT]\simeq\MCov
\]
where $\shT=\Di{\Z^{M}/<e_{0}>}$.\end{rem}
\begin{prop}
The following results hold:
\begin{enumerate}
\item $\{h=0\}=|\Bi\Di M|$;
\item $\{h\geq|M|\}=\emptyset$;
\item $\{h=|M|-1\}=|\Bi\Di{\Z^{M}/<e_{0}>}|$ (see \ref{rem:The 0 section and the 0 algebras})
\end{enumerate}
\end{prop}
\begin{proof}
If $X\arrdi fY$ is a $\Di M$-torsor, clearly $h_{f}=0$. So $1)$
and $2)$ follow from \ref{cor:An algebra with H =00003D 0 is generated in the degrees where h=00003D1}.
Finally, if $B\in\MCov(k)$ with multiplication $\psi$, $h_{B/k}=|M|-1$
if and only if $H_{B/k}=0$ and $h_{B/k,m}=1$ for any $m\in M-\{0\}$.
This means that $\psi_{m,n}=0$ for any $m,n\neq0$ by \ref{lem:showing that hm and h are well defined}.
\end{proof}
In particular, setting $U_{i}=\{h\leq i\}$, we obtain a stratification
$\Bi\Di M=U_{0}\subseteq U_{1}\subseteq\cdots\subseteq U_{|M|-1}=\MCov$
of $\MCov$ by open substacks.

\subsection{The locus $h\leq1$.}

In this subsection we want to describe $\Di M$-covers with $h\leq1$.
This means that 'up to torsors' we have a graded $M$-algebra generated
over the base ring in one degree. We will see that $\{h\leq1\}$ is
a smooth open substack of $\stZ_{M}$ determined by a special class
of explicit smooth integral extremal rays of $K_{+}$. This will allow
to give a description of covers over locally noetherian and locally
factorial scheme $X$ with $(\car X,|M|)=1$ whose total space is
normal. This result, when $X$ is a smooth algebraic variety over
an algebraic closed field $k$ with $(\car k,|M|)=1$, was already
proved in \cite[Theorem 2.1, Corollary 3.1]{Pardini1991}.
\begin{notation}
Given $\E\in\duale K_{+}$ we will write $\E_{m,n}=\E(v_{m,n})$.
Since $K\otimes\Q\simeq\Q^{M}/<e_{0}>$ we will also write $\E_{m}=\E(e_{m})\in\Q$,
so that $\E_{m,n}=\E_{m}+\E_{n}-\E_{m+n}$. Given a group homomorphism
$\eta\colon M\arr N$ we will denote by $\eta_{*}\colon K_{M}\arr K_{N}$
the homomorphism such that $\eta_{*}(v_{m,n})=v_{\eta(m),\eta(n)}$
for all $m,n\in M$.\end{notation}
\begin{rem}
Let $A$ be a ring and consider a sequence $\underline{\E}=\E^{1},\dots,\E^{r}\in\duale K_{+}$.
An element of $\stF_{\underline{\E}}(A)$ coming from the atlas (see
\ref{rem:description of objects of FE}) is given by a pair $(\underline{z},\lambda)$
where $\underline{z}=z_{1},\dots,z_{r}\in A$ and $\lambda\colon K\arr A^{*}$.
The image of this object under $\pi_{\underline{\E}}$ is the algebra
whose multiplication is given by $\psi_{m,n}=\lambda_{m,n}^{-1}z_{1}^{\E_{m,n}^{1}}\cdots z_{r}^{\E_{m,n}^{r}}$.\end{rem}
\begin{lem}
\label{lem:comparison smooth sequences for M covers}Let $\eta\colon M\arr N$
be a surjective morphism and $\underline{\E}$ be a sequence in $\duale{(K_{+N})}$.
Then $\underline{\E}$ is a smooth sequence for $N$ if and only if
$\underline{\E}\circ\eta_{*}$ is a smooth sequence for $M$.\end{lem}
\begin{proof}
We want to apply \ref{lem:comparison smooth sequences for different monoids}.
Therefore we have to prove that $\eta_{*}(K_{+M})=K_{+N}$, which
is clear, and that $\Ker\eta_{*}=<\Ker\eta_{*}\cap K_{+N}>$. Consider
the map $f\colon\Z^{M}/<e_{0}>\arr\Z^{N}/<e_{0}>$ given by $f(e_{m})=e_{\eta(m)}$
and set $H=\Ker\eta$. Clearly $f_{|K_{M}}=\eta_{*}$. It is easy
to check that $G=<v_{m,n}\text{ for }m\in H>_{\Z}\subseteq\Ker\eta^{*}\subseteq\Ker f$
and that $\Ker f/\Ker\eta_{*}\simeq H$. So in order to conclude,
it is enough to note that the map $H\arr\Ker f/G$ sending $h$ to
$e_{h}$ is a surjective group homomorphism since we have relations
$e_{h}+e_{h'}-e_{h+h'}=v_{h,h'}$ and $e_{m+h}-e_{m}=e_{h}-v_{m,h}$
for $m\in M$ and $h,h'\in H$.\end{proof}
\begin{prop}
\label{pro:Rita's smooth integral extremal rays}Let $\eta\colon M\arr\Z/l\Z$
be a surjective homomorphism with $l>1$. Then 
\[
\E^{\eta}(v_{m,n})=\left\{ \begin{array}{cc}
0 & \text{if }\eta(m)+\eta(n)<l\\
1 & \text{otherwise}
\end{array}\right.
\]
defines a smooth integral extremal ray for $K_{+}$.\end{prop}
\begin{proof}
$\E^{\eta}\in\duale K_{+}$ because, if $\sigma\colon\Z/l\Z\arr\N$
is the obvious section, $\E^{\eta}$ is the restriction of the map
$\Z^{M}/<e_{0}>\arr\Z$ sending $e_{m}$ to $\sigma(\eta(m))$. In
order to conclude the proof, we will apply \ref{lem:comparison smooth sequences for M covers}
and \ref{lem:equivalent condition for a smooth integral extremal ray}.
Set $N=\Z/l\Z$. One clearly has $\E^{\eta}=\E^{\id}\circ\eta_{*}$
and so we can assume $M=\Z/l\Z$ and $\eta=\id$. In this case one
can check that $v_{1,1},v_{1,2},\dots,v_{1,l-1}$ is a $\Z$-base
of $K$ such that $\E^{\eta}(v_{1,j})=0$ if $j<l-1$, $\E^{\eta}(v_{1,l-1})=1$.
\end{proof}
Those particular rays have been already defined in \cite[Equation 2.2]{Pardini1991}.
\begin{notation}
\label{not:ZME}If $\phi\colon\tilde{K}_{+}\arr\Z^{M}/<e_{0}>$ is
the usual map we set $\stZ_{M}^{\underline{\E}}=\stX_{\phi}^{\underline{\E}}$
(see definition \ref{def:T+epsilon sottolineato}) for any sequence
$\underline{\E}$ of elements of $\duale K_{+}$. Remember that if
$\underline{\E}$ is a smooth sequence then $\stZ_{M}^{\underline{\E}}$
is a smooth open subset of $\stZ_{M}$ (see \ref{thm:toric open substack of Z phi via smooth sequence})
and its points have the description given in \ref{pro:points of XphiE}.

Set $\Phi_{M}$ for the union over all $d>1$ of the sets of surjective
maps $M\arr\Z/d\Z$.\end{notation}
\begin{thm}
\label{thm:fundamental theorem for hleqone}Let $\underline{\E}=(\E^{\eta})_{\eta\in\Phi_{M}}$.
We have
\[
\{h\leq1\}=\bigcup_{\eta\in\Phi_{M}}\stZ_{M}^{\E^{\eta}}
\]
In particular $\{h\leq1\}\subseteq\stZ_{M}^{\textup{sm}}$ and $\pi_{\underline{\E}}$
induces an equivalence of categories
\[
\{(\underline{\shL},\underline{\shM},\underline{z},\lambda)\in\stF_{\underline{\E}}\;|\; V(z_{\eta})\cap V(z_{\mu})=\emptyset\text{ if }\eta\neq\mu\}=\pi_{\underline{\E}}^{-1}(\{h\leq1\})\arrdi{\simeq}\{h\leq1\}
\]
\end{thm}
\begin{proof}
The last part of the statement follows from the first one just applying
\ref{pro:piE for theta isomorphism} with $\Theta=\{(\E^{\eta})\}_{\eta\in\Phi_{M}}$.
Let $k$ be an algebraically closed field and $B\in\MCov(k)$ with
graded basis $\{v_{m}\}_{m\in M}$ and multiplication $\psi$.

$\supseteq$. Assume $B\in\stZ_{M}^{\E^{\eta}}(k)$. If B is a torsor
we will have $h_{B/k}=0$. Otherwise we can write $\psi=\xi0^{\E^{\eta}}$
for some $\xi\colon K\arr k^{*}$. Up to change $\Spec k$ with a
geometrical point of the maximal torsor of $B/k$, we can assume that
$M=\Z/d\Z$ and $\eta=\id$. In particular $H_{B/k}=0$ and, from
the definition of $\E^{\id}$, we get $B\simeq k[x]/(x^{d})$. So
$h_{B/k}=\dim_{k}m_{B}/m_{B}^{2}=1$.

$\subseteq$. Assume $h_{B/k}=1$. Set $C$ for the maximal torsor
of $B/k$ (see \ref{def:of the maximal torsor and the subgroup associated}),
$H=H_{B/k}$ and $l=|M/H|$. $h_{B/k}=1$ means that there exists
a unique $\overline{r}\in M/H$ (where $r\in M$) such that $h_{B/k,r}=1$
and so $C_{q}[v_{r}]=B_{q}\simeq C_{q}[x]/(x^{l})$ for all (maximal)
primes $q$ of $C$. In particular $B=C[v_{r}]\simeq C[x]/(x^{l})$
and $\overline{r}$ generates $M/H$. Let $\eta\colon M\arr M/H\simeq\Z/l\Z$
be the projection. We want to prove that $B\in\stZ_{M}^{\E^{\eta}}$.
Up to change $k$ with a geometrical point of some fppf extension
of $k$, we can assume $C=k[H]$, i.e. $v_{h}v_{h'}=v_{h+h'}$ if
$h,h'\in H$. Finally the elements $v_{h}v_{r}^{i}$ for $h\in H$
and $0\leq i<l$ define an $M$-graded basis of $B/k$ whose associated
multiplication is $0^{\E^{\eta}}$.\end{proof}
\begin{thm}
\label{thm:for hleqone}Let $\underline{\E}=(\E^{\eta})_{\eta\in\Phi_{M}}$
and let $X$ be a locally noetherian and locally factorial scheme.
Set $\catC_{X}^{1}=\{(\underline{\shL},\underline{\shM},\underline{z},\lambda)\in\stF_{\underline{\E}}(X)\;|\;\codim_{X}V(z_{\eta})\cap V(z_{\mu})\geq2\text{ if }\eta\neq\mu\}$
and $\catD_{X}^{1}=\{Y\arrdi fX\in\MCov(X)\st h_{f}(p)\leq1\ \forall p\in X\text{ with }\codim_{p}X\leq1\}$. 

Then $\pi_{\underline{\E}}$ induces an equivalence of categories
\[
\catD_{X}^{1}=\pi_{\underline{\E}}^{-1}(\catC_{X}^{1})\arrdi{\simeq}\catC_{X}^{1}
\]
\end{thm}
\begin{proof}
Apply \ref{thm:fundamental theorem for locally factorial schemes}
with $\Theta=\{(\E^{\eta})\}_{\eta\in\Phi_{M}}$.\end{proof}
\begin{thm}
\label{thm:regular in codimension 1 covers}Let $\underline{\E}=(\E^{\eta})_{\eta\in\Phi_{M}}$
and let $X$ be a locally noetherian and locally factorial scheme
without isolated points and $(\car X,|M|)=1$, i.e. $1/|M|\in\odi X(X)$.
Set 
\[
Reg_{X}^{1}=\{Y/X\in\MCov(X)\st Y\text{ regular in codimension }1\}
\]
and
\[
\widetilde{Reg}_{X}^{1}=\left\{ (\underline{\shL},\underline{\shM},\underline{z},\lambda)\in\stF_{\underline{\E}}(X)\left|\begin{array}{c}
\forall\E\neq\delta\in\underline{\E}\;\codim_{X}V(z_{\E})\cap V(z_{\delta})\geq2\\
\forall\E\in\underline{\E}\forall p\in X^{(1)}\; v_{p}(z_{\E})\leq1
\end{array}\right.\right\} 
\]
Then we have an equivalence of categories
\[
\widetilde{Reg}_{X}^{1}=\pi_{\underline{\E}}^{-1}(Reg_{X}^{1})\arrdi{\simeq}Reg_{X}^{1}
\]
\end{thm}
\begin{proof}
We will make use of \ref{thm:for hleqone}. If $Y\arrdi fX\in Reg_{X}^{1}$,
$p\in Y^{(1)}$ and $q=f(p)$ then $h_{f}(q)\leq\dim_{k(p)}m_{p}/m_{p}^{2}=1$.
So $Reg_{X}^{1}\subseteq\catD_{X}^{1}$. So we have only to check
that $\widetilde{Reg}_{X}^{1}=\pi_{\underline{\E}}^{-1}(Reg_{X}^{1})\subseteq\catC_{X}^{1}$.
In particular we can assume that $X=\Spec R$, where $R$ is a DVR.
Let $\chi\in\catC_{X}^{1}$, $ $$A/R\in\catD_{X}^{1}$ the associated
covers, $H=H_{A/R}$ and $C$ be the maximal torsor of $A/R$. We
have to prove that $\chi\in\widetilde{Reg}_{X}^{1}$ if and only if
$A$ is regular in codimension $1$. Since $D_{R}(H)$ is etale over
$R$ so is also $\Spec C$. It is so easy to check that, up to change
$R$ with a localization of $C$ and $M$ with $M/H$, we can assume
that $H=0$. Since $\chi\in\catC_{X}^{1}$, the multiplication of
$A$ over $R$ is of the form $\psi=\mu z^{r\E^{\phi}}$, where $\mu\colon K\arr R^{*}$
is an $M$-torsor, $z$ is a parameter of $A$, $\phi\colon M\arr\Z/l\Z$
is an isomorphism and $r=v_{R}(z_{\E^{\phi}})$. Moreover $v_{R}(z_{\E^{\psi}})=0$
if $\psi\neq\phi$. Up to change $M$ with $\Z/l\Z$ through $\phi$
we can assume $\phi=\id$. Finally, since $\mu$ induces an (fppf)
torsor which is etale over $R$, up to change $R$ with an etale neighborhood,
we can assume $\mu=1$. After these reductions we have $A=R[X]/(X^{|M|}-z^{r})$
which is regular in codimension $1$ if and only if $r=1$.\end{proof}
\begin{rem}
In the theorem above one can replace the condition 'regular in codimension
$1$' in the definition of $Reg_{X}^{1}$ with 'normal' thanks to
Serre's conditions, since all the fibers involved are Gorenstein.
\end{rem}

\begin{rem}
Theorem \ref{thm:regular in codimension 1 covers} is a rewriting
of Theorem $2.1$ and Corollary $3.1$ of \cite{Pardini1991} extended
to locally noetherian and locally factorial schemes without isolated
points, where an object of $\stF_{\underline{\E}}(X)$ is called a
building data.
\end{rem}

\section{The locus $h\leq2$}

In this section we want to give a characterization of the open subset
$\{h\leq2\}\subseteq\MCov$ as done in \ref{thm:for hleqone} for
$\{h\leq1\}$. The general problem we want to solve can be stated
as follows.
\begin{problem}
\label{prob: general problem for hleqtwo}Find a sequence of smooth
integral extremal rays $\underline{\E}$ for $M$ and a collection
$\Theta$ of smooth sequences with rays in $\underline{\E}$ such
that (see \ref{not:ZME})
\[
\{h\leq2\}=\bigcup_{\underline{\delta}\in\Theta}\stZ_{M}^{\underline{\delta}}
\]
or, equivalently, such that, for any algebraically closed field $k$,
the algebras $A\in\MCov(k)$ with $h_{A/k}\leq2$ are exactly the
algebras associated to a multiplication of the form $\psi=\omega0^{\E}$
where $\omega\colon K_{+}\arr k^{*}$ and $\E\in<\underline{\delta}>_{\N}$
for some $\underline{\delta}\in\Theta$.
\end{problem}
For example in the case $h\leq1$ the analogous problem is solved
taking $\underline{\E}=(\E^{\phi})_{\phi\in\Phi_{M}}$ and $\Theta=\{(\E)\text{ for }\E\in\underline{\E}\}$
(see \ref{thm:fundamental theorem for hleqone}). Once we have found
a pair $\underline{\E},\Theta$ as in \ref{prob: general problem for hleqtwo}
we can formally apply theorems \ref{pro:piE for theta isomorphism}
and \ref{thm:fundamental theorem for locally factorial schemes}.
This is done in theorems \ref{thm:fundamental thm for hleqtwo} and
\ref{thm:fundamental thm locally factoria hleqtwo}.

The first problem (see \ref{not:starting from a cover, m.n generate M})
to solve is to describe $M$-graded algebras over a field $k$ (yielding
$\Di M$-covers) generated in two degrees. As for the case $h\leq1$,
where we can reduce to study $\Z/d\Z$-graded algebras generated in
degree $1$ and then considering surjective maps $M\arr\Z/d\Z$, also
in the case $h\leq2$ we can restrict our attention to the study of
algebras generated in two degrees $m,n\in M$. It is easy to see (\ref{pro:description group generated by two elements})
that a group with two marked elements generating it is canonically
isomorphic to
\[
M_{r,\alpha,N}=\Z^{2}/(r,-\alpha),(0,N)\comma e_{1}\comma e_{2}
\]
for some natural numbers $0\leq\alpha<N\comma r>0$. 

So our problem become the study of $\Di{M_{r,\alpha,N}}$-covers $A$
over a field $k$ generated in degrees $e_{1},e_{2}$ and such that
$H_{A/k}=0$ (thanks to \ref{cor:An algebra with H =00003D 0 is generated in the degrees where h=00003D1},
when $H_{A/k}=0$ the number $h_{A/k}$ is simply the minimum number
of degrees in which $A$ is generated). These algebras are related
to the following set: given $q\in\Z$, define $d_{q}$ as the only
integer $0<d_{q}\leq N$ such that $d_{q}\equiv-q\alpha$ mod $N$
and define
\[
\Omega_{N-\alpha,N}=\{0<q\leq o(\alpha,\Z/N\Z)=N/(N,\alpha)\;|\; d_{q'}<d_{q}\text{ for any }0<q'<q\}
\]
The point is that if we take $z_{A}$ as the minimum $h>0$ for which
there exist $s\in\N$ and $\lambda\in k$ such that $hm=sn$ and $v_{m}^{h}=\lambda v_{n}^{s}$,
then $\overline{q}_{A}=z_{A}/r\in\Omega_{N-\alpha,N}$ (see \ref{pro:overlineqA is in omeganminusalpha,n}).
We will show that also the converse holds, in the sense that for any
$\overline{q}\in\Omega_{N-\alpha,N}$ there exists a $\Di{M_{r,\alpha,N}}$-cover
$A/k$ as above with $\overline{q}_{A}=\overline{q}$. The precise
statement is \ref{pro:Universal algebras generated by m,n with overline q fixed}.
This gives a solution to the first problem and also suggests how to
proceed for the next one, i.e. find the sequence $\underline{\E}$
of problem \ref{prob: general problem for hleqtwo}.

It turns out that the integral extremal rays $\E$ for $M$ such that
$h_{\E}=2$ correspond to particular sequences of the form $\chi=(r,\alpha,N,\overline{q},\phi)$,
where $0\leq\alpha<N$, $r>0$, $\overline{q}\in\Omega_{N-\alpha,N}$
and $\phi\colon M\arr M_{r,\alpha,N}$ is a surjective map (see \ref{pro:classification sm int ray htwo}).
The sequence of smooth integral extremal rays {}``needed'' to describe
the substack $\{h\leq2\}$ is composed by the {}``old'' rays $(\E^{\eta})_{\eta\in\Phi_{M}}$
and by these new rays. Finally the smooth sequences in the family
$\Theta$ of problem \ref{prob: general problem for hleqtwo} will
be all given by elements of the dual basis of particular $\Z$-basis
of $K$ (see \ref{lem:lambda,delta for overlineq}).

In the last subsection we will see (Theorem \ref{thm:NC in codimension one})
that the $\Di M$-covers of a locally noetherian and locally factorial
scheme with no isolated points and with $(\car X,|M|)=1$ whose total
space is normal crossing in codimension $1$ can be described in the
spirit of classification \ref{thm:regular in codimension 1 covers}
and extending this result.

\subsection{Good sequences.}

In this subsection we provide some general technical results in order
to work with $M$-graded algebras over local rings. So we will consider
given a local ring $D$, a sequence $\underline{m}=m_{1},\dots,m_{r}\in M$
and $C\in\MCov(D)$ generated in degrees $m_{1},\dots,m_{r}$. Since
$\Pic\Spec D$=0 for any $u\in M$ we have $C_{u}\simeq D$. Given
$u\in M$, we will call $v_{u}$ a generator of $C_{u}$ and we will
also use the abbreviation $v_{i}=v_{m_{i}}$. Moreover, if $\underline{A}=(A_{1},\dots,A_{r})\in\N^{r}$
we will also write
\[
v^{\underline{A}}=v_{1}^{A_{1}}\cdots v_{r}^{A_{r}}
\]

\begin{defn}
A sequence for $u\in M$ is a sequence $\underline{A}\in\N^{r}$ such
that $A_{1}m_{1}+\cdots+A_{r}m_{r}=u$. Such a sequence will be called
\emph{good }if the map $C_{m_{1}}^{A_{1}}\otimes\cdots\otimes C_{m_{r}}^{A_{r}}\arr C_{u}$
is surjective, i.e. $v^{\underline{A}}$ generates $C_{u}$. If $r=2$
we will talk about pairs instead of sequences.\end{defn}
\begin{rem}
Any $u\in M$ admits a good sequence since, otherwise, we will have
$C_{u}=(D[v_{1},\dots,v_{r}])_{u}\subseteq m_{D}C_{u}$. If $\underline{A}$
is a good sequence and $\underline{B}\leq\underline{A}$, then also
$\underline{B}$ is a good.\end{rem}
\begin{lem}
\label{lem:fundamental local rings and generators}Let $\underline{A}$,
$\underline{B}$ be two sequences for some element of $M$ and assume
that $\underline{A}$ is good. Set $\underline{E}=\min(\underline{A},\underline{B})=(\min(A_{1},B_{1}),\dots,\min(A_{r},B_{r}))$
and take $\lambda\in D$. Then 
\[
v^{\underline{B}}=\lambda v^{\underline{A}}\then v^{\underline{B}-\underline{E}}=\lambda v^{\underline{A}-\underline{E}}
\]
\end{lem}
\begin{proof}
Clearly we have $v^{\underline{E}}(v^{\underline{B}-\underline{E}}-\lambda v^{\underline{A}-\underline{E}})=0$.
On the other hand, since $\underline{A}-\underline{E}$ is a good
sequence, there exists $\mu\in D$ such that $v^{\underline{B}-\underline{E}}=\mu v^{\underline{A}-\underline{E}}$.
Since $\underline{A}$ is a good sequence, substituting we get $v^{\underline{A}}(\mu-\lambda)=0\then\mu=\lambda$.
\end{proof}

\subsection{\label{sub:algebras gen in two degrees}$M$-graded algebras generated
in two degrees.}
\begin{defn}
Given $0\leq\alpha<N\comma r>0$ set $M_{r,\alpha,N}=\Z^{2}/(r,-\alpha),(0,N)$.\end{defn}
\begin{prop}
\label{pro:description group generated by two elements}A finite abelian
group $M$ with two marked elements $m,n\in M$ generating it is canonically
isomorphic to $(M_{r,\alpha,N},e_{1},e_{2})$ where $r=\min\{s>0\st sm\in<n>\}$,
$rm=\alpha n$ and $N=o(n)$. Moreover we have: $|M|=Nr$, $o(m)=rN/(\alpha,N)$
and 
\[
m,n\neq0\text{ and }m\neq n\iff N>1\text{ and }(r>1\text{ or }\alpha>1)
\]
\end{prop}
\begin{proof}
We have
\[
0\arr\Z^{2}\arrdi{\scriptsize{\left(\begin{array}{cc}
r & 0\\
-\alpha & N
\end{array}\right)}}\Z^{2}\arr M_{r,\alpha,N}\arr0\then|M_{r,\alpha,N}|=\left|\det\left(\begin{array}{cc}
r & 0\\
-\alpha & N
\end{array}\right)\right|=rN
\]
and clearly $e_{1},e_{2}$ generate $M$. Moreover $M_{r,\alpha,N}/<e_{2}>\simeq\Z/r\Z$
and therefore $r$ is the minimum such that $re_{1}\in<e_{2}>$. Finally
it is easy to check that $N=o(e_{2})$. If now $M,r,\alpha,N$ are
as in the statement, there exists a unique map $ $$M_{r,\alpha,N}\arr M$
sending $e_{1},e_{2}$ to $m,n$. This map is an isomorphism since
it is clearly surjective and $|M|=o(m)o(n)/|<m>\cap<n>|=o(n)r=|M_{r,\alpha,N}|$.
The last equivalence in the statement is now easy to prove.\end{proof}
\begin{notation}
\label{not:for alpha,N,r e M}In this subsection we will fix a finite
abelian group $M$ generated by two elements $0\neq m,n\in M$ such
that $m\neq n$. Up to isomorphism, this means $M=M_{r,\alpha,N}$
with $m=e_{1}\comma n=e_{2}$ and conditions $0\leq\alpha<N\comma r>0\comma N>1\comma(r>1\text{ or }\alpha>1)$. 

We will write $d_{q}$ for the only integers $0<d_{q}\leq N$ such
that $qrm+d_{q}n=0$, for $q\in\Z$, or, equivalently, $d_{q}\equiv-q\alpha\text{ mod }(N)$.
\end{notation}

\begin{problem}
\label{not:starting from a cover, m.n generate M}Let $k$ be a field.
We want to describe, up to isomorphism, algebras $A\in\MCov(k)$ such
that $A$ is generated in degrees $m,n$ and $H_{A/k}=0$. Thanks
to \ref{cor:An algebra with H =00003D 0 is generated in the degrees where h=00003D1},
this is equivalent to asking for an algebra $A$ such that $H_{A/k}=0$
and
\[
\{l\in M\st h_{A/k,l}=1\}\subseteq\{m,n\}
\]
The solution of this problem is contained in \ref{pro:Universal algebras generated by m,n with overline q fixed}. 
\end{problem}
In this subsection we will fix an algebra $A$ as in \ref{not:starting from a cover, m.n generate M}.
$\{v_{l}\}_{l\in M}$ will be a graded basis of $A$ and $\psi$ the
associated multiplication. Note that $H_{A/k}=0$ means $v_{m},v_{n}\notin A^{*}$.
\begin{defn}
Define
\[
z=\min\{h>0\;|\; v_{m}^{h}=\lambda v_{n}^{i}\text{ for some }\lambda\in k\text{ and }hm=in\}
\]
\[
x=\min\{h>0\;|\; v_{n}^{h}=\mu v_{m}^{i}\text{ for some }\mu\in k\text{ and }hn=im\}
\]
Denote by $0\leq y<o(n)$, $0\leq w<o(m)$ the elements such that
$zm=yn$, $xn=wm$, by $\lambda,\mu\in k$ the elements such that
$v_{m}^{z}=\lambda v_{n}^{y}$, $v_{n}^{x}=\mu v_{m}^{w}$, with the
convention that $\lambda=0$ if $v_{n}^{y}=0$ and $\mu=0$ if $v_{m}^{w}=0$.
Finally set $\overline{q}=z/r$ and define the map of sets   \[   \begin{tikzpicture}[xscale=3.5,yscale=-0.6]     \node (A0_0) at (0, 0) {$\{0,1,\dots,z-1\}$};     \node (A0_1) at (1, 0) {$\{0,1,\dots,o(n)\}$};     \node (A1_0) at (0, 1) {$c$};     \node (A1_1) at (1, 1) {$\min\{d\in\N \ | \ v_{m}^{c}v_{n}^{d}=0\}$};     \path (A0_0) edge [->] node [auto] {$\scriptstyle{f}$} (A0_1);     \path (A1_0) edge [|->,gray] node [auto] {$\scriptstyle{}$} (A1_1);   \end{tikzpicture}   \] 
We will also write $\overline{q}_{A}\comma z_{A}\comma x_{A}\comma y_{A}\comma w_{A}\comma\lambda_{A}\comma\mu_{A}\comma f_{A}$
if necessary.
\end{defn}
We will see that $A$ is uniquely determined by $\overline{q}$ and
$\lambda$ up to isomorphism.
\begin{lem}
\label{lem:good pair for A<z}Given $l\in M$ there exists a unique
good pair $(a,b)$ for $l$ with $0\leq a<z$. Moreover $0\leq b<f(a)$.\end{lem}
\begin{proof}
\emph{Existence. }We know that there exists a good pair $(a,b)$ for
$l$ and we can assume that $a$ is minimum. If $a\geq z$ we can
write $v_{m}^{a}v_{n}^{b}=\lambda v_{m}^{a-z}v_{n}^{b+y}$. Therefore
$\lambda\neq0$ and $(a-z,b+y)$ is a good pair for l, contradicting
the minimality of $a$. Finally $v_{m}^{a}v_{n}^{b}\neq0$ means $b<f(a)$.

\emph{Uniqueness. }Let $(a,b)$, $(a',b')$ be two good pairs for
$l$ and assume $0\leq a<a'<z$. So there exists $\omega\in k^{*}$
such that
\[
v_{m}^{a}v_{n}^{b}=\omega v_{m}^{a'}v_{n}^{b'}\then v_{n}^{b}=\omega v_{m}^{a'-a}v_{n}^{b'}
\]
If $b\geq b'$ then $a'-a\geq z$ by definition of $z$, while if
$b<b'$ then $v_{n}$ is invertible.\end{proof}
\begin{defn}
\label{not:definition of E delta starting from the cover}Given $l\in M$
we will write the associated good pair as $(\E_{l},\delta_{l})$ with
$\E_{l}<z$. $\E,\delta$ will be considered as maps $\Z^{M}/<e_{0}>\arr\Z$
and, if necessary, we will also write $\E^{A},\delta^{A}$.\end{defn}
\begin{notation}
Up to isomorphism, we can change the given basis to 
\[
v_{l}=v_{m}^{\E_{l}}v_{n}^{\delta_{l}}
\]
so that the multiplication $\psi$ is given by
\begin{equation}
v_{a}v_{b}=v_{m}^{\E_{a}+\E_{b}}v_{n}^{\delta_{a}+\delta_{b}}=\psi_{a,b}v_{m}^{\E_{a+b}}v_{n}^{\delta_{a+b}}=\psi_{a,b}v_{a+b}\label{eq:multiplications psi for m,n first}
\end{equation}
\end{notation}
\begin{cor}
$f$ is a decreasing function and
\begin{equation}
f(0)+\cdots+f(z-1)=|M|\label{eq:the sum of f's is |M|}
\end{equation}
\end{cor}
\begin{proof}
If $(a,b)$ is a pair such that $0\leq a<z$ and $0\leq b<f(a)$ then
$v_{m}^{a}v_{n}^{b}\neq0$, i.e. $(a,b)$ is a good pair for $am+bn$.
So
\[
\sum_{c=0}^{z-1}f(c)=|\{(a,b)\;|\;0\leq a<z,\;0\leq b<f(a)\}|=|M|
\]
\end{proof}
\begin{rem}
\label{rem:some remarks on f}The following pairs are good:
\[
(z-1)m:(z-1,0),\;(x-1)n:(0,x-1),\; zm=yn:(0,y)\comma xn=wm:(w,0)
\]
i.e. $v_{m}^{z-1},v_{n}^{x-1},v_{n}^{y},v_{m}^{w}\neq0$. In particular
$f(0)\geq x,y+1$ and $f(c)>0$ for any $c$. Indeed
\begin{alignat*}{3}
v_{m}^{z-1}=\omega v_{m}^{a}v_{n}^{b} & \then & v_{m}^{z-1-a}=\omega v_{n}^{b} & \then & a=z-1,\; b=0\\
v_{m}^{z}=\omega v_{m}^{a}v_{n}^{b} & \then & v_{m}^{z-a}=\omega v_{n}^{b} & \then & a=0,\; b=y
\end{alignat*}
where $(a,b)$ are good pairs for the given elements and, by symmetry,
we get the result.
\end{rem}

\begin{rem}
\label{rem:lambda is 0 iff mu is 0}If $\lambda\neq0$ or $\mu\neq0$
then $x=y$, $z=w$ and $\lambda\mu=1$. Assume for example $\lambda\neq0$.
If $y=0$ then $v_{m}^{z}=\lambda\neq0$ and so $v_{m}$ is invertible.
So $y>0$ and, since $v_{n}^{y}=\lambda^{-1}v_{m}^{z}$, we also have
$y\geq x$. Now
\[
0\neq v_{m}^{z}=\lambda v_{n}^{y}=\lambda\mu v_{n}^{y-x}v_{m}^{w}
\]
So $\mu\neq0$ and $(y-x,w)$ is a good pair. As before $w\geq z$
and therefore
\[
\lambda\mu v_{n}^{y-x}v_{m}^{w-z}=1\then y=x,\; w=x\text{ and }\lambda\mu=1
\]
\end{rem}
\begin{lem}
\label{lem:computation of E and delta}Let $a,b\in M$.$ $We have:
\begin{itemize}
\item Assume $\E_{a,b}>0$. If $\delta_{a,b}\leq0$ then $\E_{a,b}\geq z\comma\delta_{a,b}\geq-y$.
Moreover $\psi_{a,b}\neq0\iff\lambda\neq0,\E_{a,b}=z,\delta_{a,b}=-y(=-x)$
and in this case $\psi_{a,b}=\lambda$.
\item Assume $\E_{a,b}<0$. Then $\E_{a,b}\geq-w,\delta_{a,b}\geq x$. Moreover
$\psi_{a,b}\neq0\iff\mu\neq0,\E_{a,b}=-w(=-z),\delta_{a,b}=x$ and
in this case $\psi_{a,b}=\mu$.
\item Assume $\E_{a,b}=0$. Then we have $\delta_{a,b}=0$ and $\psi_{a,b}=1$
or $\delta_{a,b}\geq o(n)$ and $\psi_{a,b}=0$.
\end{itemize}
\end{lem}
\begin{proof}
Set $\psi=\psi_{a,b}$. We start with the case $\E_{a,b}>0$. From
\ref{eq:multiplications psi for m,n first} we get
\[
v_{m}^{\E_{a,b}}v_{n}^{\delta_{a}+\delta_{b}}=\psi v_{n}^{\delta_{a+b}}
\]
If $\delta_{a,b}>0$ then $v_{m}^{\E_{a,b}}v_{n}^{\delta_{a,b}}=\psi$
and so $\psi=0$ since $v_{m}\notin A^{*}$. If $\delta_{a,b}\leq0$
we instead have $v_{m}^{\E_{a,b}}=\psi v_{n}^{-\delta_{a,b}}$ and
so $\E_{a,b}\geq z$. If $-\delta_{a,b}<y$ then $(0,-\delta_{a,b})$
is good. So we can write
\[
v_{m}^{\E_{a,b}-z}\lambda v_{n}^{y+\delta_{a,b}}=\psi\then\psi=0
\]
since $v_{n}$ is not invertible. If $\delta_{a,b}\leq-y$ we have
\[
0\leq\E_{a,b}-z<z,\;0\leq-\delta_{a,b}-y<f(0),\;(\E_{a,b}-z)m=(-\delta_{a,b}-y)n,\; v_{m}^{\E_{a,b}-z}\lambda=\psi v_{n}^{-\delta_{a,b}-y}
\]
and so both $(\E_{a,b}-z,0)$ and $(0,-\delta_{a,b}-y)$ are good
pair for the same element of $M$. Therefore we must have $\E_{a,b}=z$,
$\delta_{a,b}=-y$ and $\psi=\lambda$.

Now assume $\E_{a,b}=0$. If $\delta_{a,b}<0$ then $v_{n}^{-\delta_{a,b}}\psi=1$
which is impossible. So $\delta_{a,b}\geq0$. If $\delta_{a,b}=0$
clearly $\psi=1$. If $\delta_{a,b}>0$ then $v_{n}^{\delta_{a,b}}=\psi$
and so $\psi=0$ and $\delta_{a,b}\geq o(n)$.

Finally assume $\E_{a,b}<0$. From \ref{eq:multiplications psi for m,n first}
we get
\[
v_{n}^{\delta_{a}+\delta_{b}}=\psi v_{m}^{-\E_{a,b}}v_{n}^{\delta_{a+b}}
\]
We must have $\delta_{a,b}>0$ since $v_{m}$ is not invertible. So
$v_{n}^{\delta_{a,b}}=\psi v_{m}^{-\E_{a,b}}$ and $\delta_{a,b}\geq x$,
from which
\[
v_{n}^{\delta_{a,b}-x}\mu v_{m}^{w}=\psi v_{m}^{-\E_{a,b}}
\]
Note that, since $0\leq-\E_{a,b}\leq\E_{a+b}<z$, $(-\E_{a,b},0)$
is a good pair. If $w>-\E_{a,b}$ then $\psi=0$. So assume $w\leq-\E_{a,b}$.
Arguing as above we must have $\delta_{a,b}=x$, $\E_{a,b}=-w$ and
$\psi=\mu$.
\end{proof}

\begin{lem}
\label{lem:From lambda not zero to zero A'}Define
\[
A'=k[s,t]/(s^{z},s^{c}t^{f(c)}\:\text{for}\:0\leq c<z)
\]
Then $A'\in\MCov(k)$ with graduation $\deg s=m$, $\deg t=n$ and
it satisfies the requests of \ref{not:starting from a cover, m.n generate M},
i.e. $A'$ is generated in degrees $m,n$ and $H_{A'/k}=0$. Moreover
we have
\[
\overline{q}_{A'}=\overline{q}_{A}\comma z_{A'}=z_{A},\; y_{A'}=y_{A},\;\E^{A'}=\E^{A},\;\delta^{A'}=\delta^{A},\;\lambda_{A'}=\mu_{A'}=0\comma f_{A'}=f_{A}
\]
\end{lem}
\begin{proof}
Clearly the elements $s^{c}t^{d}$ for $0\leq c<z$, $0\leq d<f(c)$
generates $A'$ as a $k$-space. Since they are $\sum_{c=0}^{z-1}f(c)=|M|$
and they all have different degrees, it's enough to prove that any
of them are non-zero. So let $(c',d')$ a pair as always. It is enough
to show that $B=k[s,t]/(s^{c'+1},t^{d'+1})\arr A'/(s^{c'+1},t^{d'+1})$
is an isomorphism. But $c'<z$ implies that $s^{z}=0$ in $B$. If
$c'<c$ then $s^{c}t^{f(c)}=0$ in $B$ and finally if $c'\geq c$
then $d'+1\leq f(c')\leq f(c)$ and so $s^{c}t^{f(c)}=0$ in B.

$A'$ is clearly generated in degrees $m,n$ and $H_{A'/k}=0$ since
$s^{z}=t^{f(0)}=0$ and $z,f(0)>0$. $s^{z}=0t^{y}$ implies that
$z'=z_{A'}\leq z$. Assume by contradiction $z'<z$. From $0\neq s^{z'}=\lambda't^{y'}$
we know that $t^{y'}\neq0$ so that $y'<f(0)$. Therefore $(\E_{z'm},\delta_{z'm})=(z',0)=(0,y')$
and so $z'=0$, which is a contradiction. Then $z'=z$, $y_{A'}=y'=y$.
Also $s^{z}=0t^{y}$ and $t^{y}\neq0$ imply $\lambda_{A'}=0$ and,
thanks to \ref{rem:lambda is 0 iff mu is 0}, $\mu_{A'}=0$. Finally
by construction we also have $\E^{A'}=\E$, $\delta^{A'}=\delta$
and $f_{A'}=f$.\end{proof}
\begin{lem}
\label{lem:overline q from a cover}We have
\[
{\displaystyle d_{\overline{q}}=\max_{1\leq q\leq\overline{q}}d_{q}}
\]
\end{lem}
\begin{proof}
Thanks to \ref{lem:From lambda not zero to zero A'} we can assume
$\lambda=0$ and, therefore, $\mu=0$. So $v_{n}^{x}=0$, $v_{n}^{x-1}\neq0$
and $v_{n}^{y}\neq0$ imply $y<x=f(0)$. Let $1\leq q<\overline{q}$
and $l=qr$. We have $(\E_{l},\delta_{l})=(qr,0)$. If $N-d_{q}<x=f(0)$
then we will also have $(\E_{l},\delta_{l})=(0,N-d_{q})$ and so $q=0$,
which is not the case. So $N-d_{q}\geq x>y=N-d_{\overline{q}}\then d_{q}<d_{\overline{q}}$.\end{proof}
\begin{lem}
\label{lem:From a cover we get overline q}Define $\hat{q}$ as the
only integers $0\leq\hat{q}<\overline{q}$ such that
\[
d_{\hat{q}}=\min_{0\leq q<\overline{q}}d_{q}
\]
If $\lambda=0$ we have $d_{\hat{q}}\leq x=f(0)\text{ and }f(c)=\left\{ \begin{array}{cc}
x & \text{if }0\leq c<\hat{q}r\\
d_{\hat{q}} & \text{if }\hat{q}r\leq c<z
\end{array}\right.$\end{lem}
\begin{proof}
We want first prove that $f(c)=\min(x,d_{q}\;\text{for}\;0\leq qr\leq c)$.
Clearly we have the inequality $\leq$ since $v_{n}^{x}=v_{m}^{qr}v_{n}^{d_{q}}=0$.
Set $d=f(c)$ and let $(a,b)$ a good pair for $cm+dn$, so that $v_{m}^{c}v_{n}^{d}=0v_{m}^{a}v_{n}^{b}$.
We cannot have $b\geq d$ since otherwise $v_{m}^{c}=0$ implies $c\geq z$.
If $a\geq c$ then $v_{n}^{d}=0$ and so $d=f(c)\geq x$. Conversely
if $a<c$ then $0\leq c-a=qr\leq c<z$ and $0<d-b=d_{q}\leq d=f(c)$.

We are now ready to prove the writing of $f$. Note that the pairs
$(qr,d_{q}-1)$, with $0\leq q<\overline{q}$, are all the possible
pairs for $-n$. So there exists a unique $0\leq\tilde{q}<\overline{q}$
such that $(\tilde{q}r,d_{\tilde{q}}-1)$ is good. In particular if
$0\leq q\neq\tilde{q}<\overline{q}$ we have a writing 
\[
v_{m}^{qr}v_{n}^{d_{q}-1}=0v_{m}^{\tilde{q}}v_{n}^{d_{\tilde{q}}-1}\then\left\{ \begin{array}{ccccc}
q<\tilde{q} & \then & v_{n}^{d_{q}-1}=0 & \then & d_{q}\geq x\\
q>\tilde{q} & \then & d_{q}>d_{\tilde{q}}
\end{array}\right.
\]
Since $v_{n}^{d_{\tilde{q}}-1}\neq0$ we must have $d_{\hat{q}}\leq x$.
This shows that $\tilde{q}=\hat{q}$ and the writing of $f$. Finally
If $\overline{q}>1$ then $\hat{q}>0$ and so $d_{\hat{q}}\leq x=f(0)$
since $f$ is a decreasing function. If $\overline{q}=1$ then $\hat{q}=0$
and so $N=d_{\hat{q}}=f(0)\leq x\leq N$.\end{proof}
\begin{defn}
We will continue to use notation from \ref{lem:From a cover we get overline q}
for $\hat{q}$ and we will also write $\hat{q}_{A}$ if necessary.
\end{defn}

\subsection{\label{sub:the invariant overlineq}The invariant $\overline{q}$.}
\begin{lem}
\label{lem:fundamental for the case m,n}Let $\beta,N\in\N$, with
$N>1$, and define $d_{q}^{\beta}=d_{q}$, for $q\in\Z$, the only
integer $0<d_{q}\leq N$ such that $d_{q}\equiv q\beta$ mod $N$.
Set
\[
\Omega_{\beta,N}=\{0<q\leq o(\beta,\Z/N\Z)=N/(N,\beta)\;|\; d_{q'}<d_{q}\text{ for any }0<q'<q\},
\]
set $q_{n}$ for the $n$-th element of it and denote by $0\leq\hat{q}<q_{n}$
the only number such that 
\[
d_{\hat{q}}=\min_{0\leq q<q_{n}}d_{q}
\]
Then we have relations $\hat{q}N+q_{n}d_{\hat{q}}-\hat{q}d_{q_{n}}=N$
and, if $n>1$, $q_{n}=q_{n-1}+\hat{q}$, $d_{q_{n}}=d_{q_{n-1}}+d_{\hat{q}}$
and $d_{q_{n-1}}+d_{q}>N$ for $q<\hat{q}$.\end{lem}
\begin{proof}
First of all note that all is defined also in the extremal case $\beta=0$.
In this case $\Omega_{\beta,N}=\{1\}$. Assume first $n>1$. Set $\tilde{q}=q_{n}-q_{n-1}$
so that $d_{q_{n}}=d_{q_{n-1}}+d_{\tilde{q}}$ since $d_{q_{n}}>d_{q_{n-1}}$.
Assume by contradiction that $\tilde{q}\neq\hat{q}$. Since $\tilde{q}<q_{n}$
we have $d_{\hat{q}}<d_{\tilde{q}}$. Let also $q'=q_{n}-\hat{q}$
and, as above, we can write $d_{q_{n}}=d_{q'}+d_{\hat{q}}$. Now
\[
d_{q_{n}}-d_{q'}=d_{\hat{q}}<d_{\tilde{q}}=d_{q_{n}}-d_{q_{n-1}}\then d_{q_{n-1}}<d_{q'}
\]
Since $q_{n-1}\in\Omega_{\beta,N}$ we must have $q'>q_{n-1}$, which
is a contradiction because otherwise, being $q'<q_{n}$, we must have
$q'=q_{n}$. So $\tilde{q}=\hat{q}$. For the last relation note that,
since $q_{n}$ is the first $q>q_{n-1}$ such that $d_{q}>d_{q_{n-1}}$,
then $\hat{q}$ is the first such that $d_{q_{n-1}}+d_{\hat{q}}\leq N$.

Now consider the first relation. We need to do induction on all the
$\beta$. So we will write $d_{q}^{\beta}$ and $q_{n}^{\beta}$ in
order to remember that those numbers depend on to $\beta$. The induction
statement on $1\leq q<N$ is: for any $0\leq\beta<N$ and for any
$n$ such that $q_{n}^{\beta}\leq q$ the required formula holds.
The base step is $q=1$. In this case we have $n=1$, $q_{1}=1$,
$\hat{q}=0$, $d_{0}=N$ and the formula can be proven directly. For
the induction step we can assume $q>1$ and $n>1$. We will write
$\hat{q}_{n}^{\beta}$ for the $\hat{q}$ associated to $n$ and $\beta$.
First of all note that, by the relations proved above, we can write
\[
\hat{q}_{n}^{\beta}N+q_{n}^{\beta}d_{\hat{q}_{n}^{\beta}}^{\beta}-\hat{q}_{n}^{\beta}d_{q_{n}^{\beta}}^{\beta}=\hat{q}_{n}^{\beta}N+q_{n-1}^{\beta}d_{\hat{q}_{n}^{\beta}}^{\beta}-\hat{q}_{n}^{\beta}d_{q_{n-1}^{\beta}}^{\beta}
\]
and so we have to prove that the second member equals $N$. If $\hat{q}_{n}^{\beta}\leq q_{n-1}^{\beta}$
then $\hat{q}_{n-1}^{\beta}=\hat{q}_{n}^{\beta}$ and the formula
is true by induction on $q-1\geq q_{n-1}^{\beta}$. So assume $\hat{q}_{n}^{\beta}>q_{n-1}^{\beta}$
and set $\alpha=N-\beta$. Clearly we will have
\[
o=o(\alpha,\Z/N\Z)=o(\beta,\Z/N\Z)\text{ and }d_{q}^{\beta}+d_{q}^{\alpha}=N\text{ for any }0<q<o
\]
Moreover
\[
d_{\hat{q}_{n}^{\beta}}^{\beta}<d_{q}^{\beta}\text{ for any }0<q<q_{n}^{\beta}\then d_{\hat{q}_{n}^{\beta}}^{\alpha}>d_{q}^{\alpha}\text{ for any }0<q<\hat{q}_{n}^{\beta}\then\exists l\text{ s.t. }q_{l}^{\alpha}=\hat{q}_{n}^{\beta}
\]
and
\[
d_{q_{n-1}^{\beta}}^{\beta}\geq d_{q}^{\beta}\text{ for any }0<q<q_{n}^{\beta}\then d_{q_{n-1}^{\beta}}^{\alpha}\leq d_{q}^{\alpha}\text{ for any }0\leq q<q_{l}^{\alpha}=\hat{q}_{n}^{\beta}\then\hat{q}_{l}^{\alpha}=q_{n-1}^{\beta}
\]
Using induction on $q_{l}^{\alpha}=\hat{q}_{n}^{\beta}<q_{n}^{\beta}\leq q$
we can finally write
\begin{alignat*}{1}
N= & \hat{q}_{l}^{\alpha}N+q_{l}^{\alpha}d_{\hat{q}_{l}^{\alpha}}^{\alpha}-\hat{q}_{l}^{\alpha}d_{q_{l}^{\alpha}}^{\alpha}=q_{n-1}^{\beta}N+\hat{q}_{n}^{\beta}d_{q_{n-1}^{\beta}}^{\alpha}-q_{n-1}^{\beta}d_{\hat{q}_{n}^{\beta}}^{\alpha}\\
= & q_{n-1}^{\beta}N+\hat{q}_{n}^{\beta}(N-d_{q_{n-1}^{\beta}}^{\beta})-q_{n-1}^{\beta}(N-d_{\hat{q}_{n}^{\beta}}^{\beta})=\hat{q}_{n}^{\beta}N+q_{n-1}^{\beta}d_{\hat{q}_{n}^{\beta}}^{\beta}-\hat{q}_{n}^{\beta}d_{q_{n-1}^{\beta}}^{\beta}
\end{alignat*}

\end{proof}
We continue to keep notation from \ref{not:for alpha,N,r e M}. With
$d_{q}$ we will always mean $d_{q}^{N-\alpha}$ as in \ref{lem:fundamental for the case m,n}.
Lemma \ref{lem:overline q from a cover} can be restated as:
\begin{prop}
\label{pro:overlineqA is in omeganminusalpha,n}Let $A$ be an algebra
as in \ref{not:starting from a cover, m.n generate M}. Then $\overline{q}_{A}\in\Omega_{N-\alpha,N}$.
\end{prop}
So given an algebra $A$ as in \ref{not:starting from a cover, m.n generate M}
we can associate to it the number $\overline{q}_{A}\in\Omega_{N-\alpha,N}$.
Conversely we will see that any $\overline{q}\in\Omega_{N-\alpha,N}$
admits an algebra $A$ as in \ref{not:starting from a cover, m.n generate M}
such that $\overline{q}=\overline{q}_{A}$. It turns out that all
the objects $z_{A}$, $y_{A}$, $f_{A}$, $\E^{A}$, $\delta^{A}$,
$\hat{q}_{A}$ and, if $\lambda_{A}=0$, $x_{A}$, $w_{A}$ associated
to $A$ only depend on $\overline{q}_{A}$. Therefore in this subsection,
given $\overline{q}\in\Omega_{N-\alpha,N}$, we will see how to define
such objects independently from an algebra $A$.

In this subsection we will consider given an element $\overline{q}\in\Omega_{N-\alpha,N}$.
\begin{defn}
\label{not:alpha, N, r, overq, q', hatq, f}Set $\hat{q}$ for the
only integer $0\leq\hat{q}<\overline{q}$ such that $d_{\hat{q}}=\min_{0\leq q<\overline{q}}d_{q}$,
$q'=\overline{q}-\hat{q}$, $z=\overline{q}r$, $y=N-d_{\overline{q}}$,
\[
x=\left\{ \begin{array}{cc}
N-d_{q'} & \text{if }\overline{q}>1\\
N & \text{if }\overline{q}=1
\end{array}\right.\comma w=\left\{ \begin{array}{cc}
q'r & \text{if }\overline{q}>1\\
0 & \text{if }\overline{q}=1
\end{array}\right.\comma f(c)=\left\{ \begin{array}{cc}
x & \text{if }0\leq c<\hat{q}r\\
d_{\hat{q}} & \text{if }\hat{q}r\leq c<z
\end{array}\right.
\]
We will also write $\hat{q}_{\overline{q}}\comma q_{\overline{q}}'\comma z_{\overline{q}}\comma x_{\overline{q}}\comma f_{\overline{q}}\comma y_{\overline{q}}\comma w_{\overline{q}}$
if necessary.\end{defn}
\begin{rem}
Using notation from \ref{lem:fundamental for the case m,n} we have
$\overline{q}=q_{n}$ for some $n$ and, if $n>1$, i.e. $\overline{q}>1$,
$q_{n-1}=q'$. Note that $zm=yn$, $wm=xn$, $y<x$, $w<z$. Moreover,
from \ref{lem:fundamental for the case m,n} and from a direct computation
if $\overline{q}=1$, we obtain $zx-yw=|M|$. Finally if $\overline{q}>1$
one has relations $\hat{q}r=z-w\text{ and }d_{\hat{q}}=x-y$.\end{rem}
\begin{lem}
\label{lem:definition of epsilon delta}We have that:
\begin{enumerate}
\item $f$ is a decreasing function and $\sum_{c=0}^{z-1}f(c)=|M|$;
\item \label{enu:writing of the form Am+Bn}any element $t\in M$ can be
uniquely written as 
\[
t=Am+Bn\text{ with }0\leq A<z,0\leq B<f(A)
\]

\end{enumerate}
\end{lem}
\begin{proof}
$1)$ If $\overline{q}=1$ it is enough to note that $\hat{q}=0$,
$d_{0}=N$ and $Nr=|M|$. So assume $\overline{q}>1$. We have $x=N-d_{q'}\geq d_{\hat{q}}$
since $d_{\overline{q}}=d_{q'}+d_{\hat{q}}$ and 
\[
\sum_{c=0}^{z-1}f(c)=\hat{q}rx+(\overline{q}r-\hat{q}r)d_{\hat{q}}=(z-w)x+w(x-y)=zx-wy=|M|
\]

$2)$ First of all note that the writings of the form $Am+Bn$ with
$0\leq A<z$, $0\leq B<f(A)$ are $\sum_{c=0}^{z-1}f(c)=|M|$. So
it's enough to prove that they are all distinct. Assume to have writings
$Am+Bn=A'm+B'n\text{ with }0\leq A'\leq A<z,0\leq B<f(A),0\leq B'<f(A')$.

$A'=B'=0$, i.e. $Am+Bn=0$. If $A=0$ then $B=0$ since $f(0)=x\leq N$.
If $A>0$, we can write $A=qr$ for some $0<q<\overline{q}$. In particular
$\overline{q}>1$ and $B=d_{q}<f(A)$. If $q<\hat{q}$ then $f(A)=x=N-d_{q'}>d_{q}$
contradicting \ref{lem:fundamental for the case m,n}, while if $q\geq\hat{q}$
then $f(A)=d_{\hat{q}}\leq d_{q}$.

$A'=B=0$, i.e. $Am=B'n$. If $A=0$ then $B'=0$ as above. If $A>0$
we can write $A=qr$ for some $0<q<\overline{q}$. Again $\overline{q}>1$.
In particular $B'=N-d_{q}<f(0)=x=N-d_{q'}$ and so $d_{q'}<d_{q}$,
while $d_{q'}=\max_{0<q<\overline{q}}d_{q}$.

\emph{General case}. We can write $(A-A')m+Bn=B'n$ and we can reduce
to the previous cases since if $B\geq B'$ then $B-B'\leq B<f(A)\leq f(A-A')$,
while if $B<B'$ then $B'-B\leq B'<f(A')\leq f(0)$.\end{proof}
\begin{defn}
Given $l\in M$ we set $(\E_{l},\delta_{l})$ the unique pair for
$l$ such that $0\leq\E_{t}<z\comma0\leq\delta_{t}<f(\E_{t})$ and
we will consider $\E\comma\delta$ as maps $\Z^{M}/<e_{0}>\arr\Z$.
We will also write $\E^{\overline{q}}\comma\delta^{\overline{q}}$
if necessary.\end{defn}
\begin{prop}
\label{pro:from an algebra to overline q}Let $A$ be an algebra as
in \ref{not:starting from a cover, m.n generate M}. Then
\[
z_{A}=z_{\overline{q}_{A}}\comma y_{A}=y_{\overline{q}_{A}}\comma\hat{q}_{A}=\hat{q}_{\overline{q}_{A}}\comma\E^{A}=\E^{\overline{q}_{A}}\comma\delta^{A}=\delta^{\overline{q}_{A}}\comma f_{A}=f_{\overline{q}_{A}}
\]
and, if $\lambda_{A}=0$, then $x_{A}=x_{\overline{q}_{A}}\comma w_{A}=w_{\overline{q}_{A}}$.\end{prop}
\begin{proof}
Set $\overline{q}=\overline{q}_{A}$. Then $z_{A}=\overline{q}r=z_{\overline{q}}$
and $z_{A}m=y_{A}n=y_{\overline{q}}n$ implies $y_{A}=y_{\overline{q}}$.
Also $\hat{q}_{A}=\hat{q}_{\overline{q}}$ by definition. Taking into
account \ref{lem:From lambda not zero to zero A'} we can now assume
$\lambda_{A}=0$. We claim that all the remaining equalities follow
from $x_{A}=x_{\overline{q}}$. Indeed clearly $w_{A}=w_{\overline{q}}$.
Also by definition of $f_{\overline{q}}$ and thanks to \ref{lem:From a cover we get overline q}
we will have $f_{A}=f_{\overline{q}}$ and therefore $\E^{A}=\E^{\overline{q}}\comma\delta^{A}=\delta^{\overline{q}}$,
that conclude the proof.

We now show that $x_{A}=x_{\overline{q}}$. If $\overline{q}=1$ then
$\hat{q}=0$ and so, from \ref{lem:From a cover we get overline q},
we have $d_{\hat{q}}=N=x_{A}=x_{1}$. If $\overline{q}>1$, by definition
of $f_{\overline{q}}$ and thanks to \ref{lem:definition of epsilon delta}
and \ref{lem:From a cover we get overline q}, we can write
\[
|M|=\sum_{c=0}^{z_{\overline{q}}-1}f_{\overline{q}}(c)=r\hat{q}_{\overline{q}}x_{\overline{q}}+(z_{\overline{q}}-\hat{q}_{\overline{q}}r)d_{\hat{q}_{\overline{q}}}=\sum_{c=0}^{z_{A}-1}f_{A}(c)=r\hat{q}_{A}x_{A}+(z_{A}-\hat{q}_{A}r)d_{\hat{q}_{A}}
\]
and so $x_{A}=x_{\overline{q}}$.\end{proof}
\begin{defn}
\label{def:universal algebra}Define the $M$-graded $\Z[a,b]$-algebra
\[
A^{\overline{q}}=\Z[a,b][s,t]/(s^{z}-at^{y},t^{x}-bs^{w},s^{\hat{q}r}t^{d_{\hat{q}}}-a^{\gamma}b)\text{ where }\gamma=\left\{ \begin{array}{cc}
0 & \text{if }\overline{q}=1\\
1 & \text{if }\overline{q}>1
\end{array}\right.
\]
with $M$-graduation $\deg s=m$, $\deg t=n$. If are given elements
$a_{0},b_{0}$ of a ring $C$ we will also write $A_{a_{0},b_{0}}^{\overline{q}}=A^{\overline{q}}\otimes_{\Z[a,b]}C$,
where $\Z[a,b]\arr C$ sends $a,b$ to $a_{0},b_{0}$.\end{defn}
\begin{prop}
\label{pro:general algebra for m,n 2}$A^{\overline{q}}\in\MCov(\Z[a,b])$,
it is generated in degrees $m,n$ and $\{v_{l}=s^{\E_{l}}t^{\delta_{l}}\}_{l\in M}$
is an $M$-graded basis for it.\end{prop}
\begin{proof}
We have to prove that, for any $l\in M$, $(A^{\overline{q}})_{l}=\Z[a,b]v_{l}$
and we can check this over a field $k$, i.e. considering $A=A_{a,b}^{\overline{q}}$
with $a,b\in k$. We first consider the case $a,b\in k^{*}$, so that
$s,t\in A^{*}$. Let $\pi\colon\Z^{2}\arr M$ the map such that $\pi(e_{1})=m\comma\pi(e_{2})=n$.
The set $T=\{(a,b)\in\Ker\pi\st s^{a}t^{b}\in k^{*}\}$ is a subgroup
of $\Ker\pi$ such that $(z,-y),(-w,x)\in T$. Since $\det\left(\begin{array}{cc}
z & -w\\
-y & x
\end{array}\right)=zx-wy=|M|$ we can conclude that $T=\Ker\pi$. Therefore $v_{l}$ generate $(A^{\overline{q}})_{l}$
since for any $c,d\in\N$ we have $s^{c}t^{d}/v_{cm+dn}\in k^{*}$
and $0\neq v_{l}\in A^{*}$.

Now assume $a=0$. If $\overline{q}=1$ then $\hat{q}=w=0$, $d_{\hat{q}}=x=N$
and so $A=k[s,t]/(s^{z},t^{N}-b)$ satisfies the requests. If $\overline{q}>1$
it is easy to see that $v_{l}$ generates $A_{l}$. On the other hand
$\dim_{k}A=|\{(A,B)\st0\leq A<z,0\leq B<x,A\leq\hat{q}r\text{ or }B\leq d_{\hat{q}}\}|=zx-(z-\hat{q}r)(x-d_{\hat{q}})=zx-yw=|M|$.
The case $b=0$ is similar.\end{proof}
\begin{thm}
\label{pro:Universal algebras generated by m,n with overline q fixed}Let
$k$ be a field. If $\overline{q}\in\Omega_{N-\alpha,N}$ and $\lambda\in k$,
with $\lambda=0$ if $\overline{q}=N/(\alpha,N)$, then
\[
A_{\overline{q},\lambda}=k[s,t]/(s^{z_{\overline{q}}}-\lambda t^{y_{\overline{q}}},t^{x_{\overline{q}}},s^{\hat{q}_{\overline{q}}r}t^{d_{\hat{q}_{\overline{q}}}})
\]
is an algebra as in \ref{not:starting from a cover, m.n generate M}
with $\overline{q}_{A_{\overline{q},\lambda}}=\overline{q}$ and $\lambda_{A_{\overline{q},\lambda}}=\lambda$.
Conversely, if $A$ is an algebra as in \ref{not:starting from a cover, m.n generate M}
then $\overline{q}_{A}\in\Omega_{N-\alpha,N}$, $\lambda_{A}\in k$,
$\lambda_{A}=0$ if $\overline{q}_{A}=N/(\alpha,N)$ and $A\simeq A_{\overline{q}_{A},\lambda_{A}}$.\end{thm}
\begin{proof}
Consider $A=A_{\overline{q},\lambda}$, which is just $A_{\lambda,0}^{\overline{q}}$.
Clearly $t\notin A^{*}$. On the other hand $s\notin A^{*}$ since
$y=0\iff z=o(m)\iff\overline{q}=N/(\alpha,N)$. Therefore $H_{A/k}=0$
and $A$ is an algebra as in \ref{not:starting from a cover, m.n generate M}.
Moreover clearly $\overline{q}_{A}\leq\overline{q}$. If by contradiction
this inequality is strict, we will have a relation $s^{qr}=\omega t^{y'}$
with $0\leq q<\overline{q}$. Since $s^{qr}=v_{qrm}\neq0$ we will
have that $t^{y'}\neq0$ and $y'<x$, a contradiction thanks to \ref{lem:definition of epsilon delta}.
In particular $\lambda=\lambda_{A}$.

Now let $A$ be as in \ref{not:starting from a cover, m.n generate M}
and set $\overline{q}=\overline{q}_{A}$, $\lambda=\lambda_{A}$.
We already know that $\overline{q}\in\Omega_{N-\alpha,N}$ (see \ref{pro:overlineqA is in omeganminusalpha,n}).
We claim that the map $A_{\overline{q},\lambda}\arr A$ sending $s,t$
to $v_{m},v_{n}$ is well defined and so an isomorphism. Indeed we
have $v_{m}^{z}=\lambda v_{n}^{y}$ by definition and, thanks to \ref{pro:from an algebra to overline q},
we have $v_{m}^{\hat{q}r}v_{n}^{d_{\hat{q}}}=0$ since $d_{\hat{q}}=f_{A}(\hat{q}r)$
and $v_{n}^{x}=0$ since $f_{A}(0)=x$. Finally if $\overline{q}=N/(\alpha,N)$
then $y=y_{A}=0$ and $z=o(m)$, so that $\lambda_{A}=v_{m}^{o(m)}=0$.\end{proof}
\begin{cor}
If $k$ is an algebraically closed field then, up to graded isomorphism,
the algebras as in \ref{not:starting from a cover, m.n generate M}
are exactly $A_{\overline{q},1}$ if $\overline{q}\in\Omega_{N-\alpha,N}-\{N/(\alpha,N)\}$
and $A_{\overline{q},0}$ if $\overline{q}\in\Omega_{N-\alpha,N}$.\end{cor}
\begin{proof}
Clearly the algebras above cannot be isomorphic. Conversely if $\lambda\in k^{*}$
(and $\overline{q}<N/(\alpha,N)$) the transformation $t\arr\sqrt[y]{\lambda}t$
with $y=y_{\overline{q}}$ yields an isomorphism $A_{\overline{q},\lambda}\simeq A_{\overline{q},1}$. 
\end{proof}

\subsection{\label{sub:smooth integral rays for hleqtwo}Smooth integral rays
for $h\leq2$.}

In this subsection we continue to keep notation from \ref{not:for alpha,N,r e M},
i.e. $M=M_{r,\alpha,N}$ and we will considered given an element $\overline{q}\in\Omega_{N-\alpha,N}$.
\begin{rem}
We have $z=1\iff\overline{q}=r=1$ and $x=1\iff\overline{q}=N$. Indeed
the first relation is clear, while for the second one note that, by
definition of $x$ and since $N>1$, we have $x=1\iff d_{q'}=N-1\iff\overline{q}=N/(\alpha,N),(\alpha,N)=1$.\end{rem}
\begin{lem}
\label{lem:lambda,delta for overlineq}The vectors of $K_{+}$
\begin{equation}
\begin{array}{lc}
v_{cm,dn} & 0<c<z,0<d<f(c)\\
v_{m,im} & 0<i<z-1\\
v_{n,jn} & 0<j<x-1\\
v_{m,(z-1)m} & \text{if }z>1\\
v_{n,(x-1)n} & \text{if }x>1
\end{array}\label{eq:base for the case m,n}
\end{equation}
form a basis of $K$. Assume $\overline{q}r\neq1$ and $\overline{q}\neq N$,
i.e. $z,x>1$, and denote by $\Lambda,\Delta$ the last two terms
of the dual basis of \ref{eq:base for the case m,n}. Then $\Lambda,\Delta\in\duale K_{+}$
and they form a smooth sequence. Moreover $\Lambda=1/|M|(x\E+w\delta)$,
$\Delta=1/|M|(y\E+z\delta)$ and
\[
\Lambda_{m,-m}=\Delta_{n,-n}=1\comma\Lambda_{n,-n}=\left\{ \begin{array}{cc}
0 & \text{if }\overline{q}=1\\
1 & \text{otherwise}
\end{array}\right.\comma\Delta_{m,-m}=\left\{ \begin{array}{cc}
0 & \text{if }\overline{q}=N/(\alpha,N)\\
1 & \text{otherwise}
\end{array}\right.
\]
\end{lem}
\begin{proof}
Note that we cannot have $z=x=1$ since otherwise $|M|=f(0)=x=1$,
i.e. $M=0$. The vectors of (\ref{eq:base for the case m,n}) are
at most $\rk K$ since
\[
\sum_{c=1}^{z-1}(f(c)-1)+z-2+x-2+2=\sum_{c=0}^{z-1}(f(c)-1)+z-1=|M|-z+z-1=|M|-1=\rk K
\]
If $z=1$ then (\ref{eq:base for the case m,n}) is $v_{n,n},\dots,v_{n,(x-1)n}$.
So $x=|M|=N$, i.e. $n$ generates $M$, and (\ref{eq:base for the case m,n})
is a base of $K$. In the same way if $x=1$, then $m$ generates
$M$ and (\ref{eq:base for the case m,n}) is a base of $K$.

So we can assume that $z,x>1$. $\E$ and $\delta$ define a map $\Z^{M}/<e_{0}>\arrdi{(\E,\delta)}\Z^{2}$.
Denote by $K'$ the subgroup of $K$ generated by the vectors in (\ref{eq:base for the case m,n}),
except the last two lines. We claim that $(\E,\delta)_{|K'}=0$. This
follows by a direct computation just observing that if we have a writing
$Am+Bn$ as in \ref{lem:definition of epsilon delta}, \ref{enu:writing of the form Am+Bn})
then $(\E,\delta)(e_{Am+Bn})=(A,B)$. Consider the diagram   \[   \begin{tikzpicture}[xscale=2.2,yscale=-1.2]     
\node (A0_0) at (0.1, 0.6) {$\scriptstyle{\sigma(e_1)=v_{m,(z-1)m}}$};     
\node (A0_1) at (0.9, 0.6) {$\scriptstyle{\sigma(e_2)=v_{n,(x-1)n}}$};     
\node (A1_0) at (0, 1) {$\Z^2$};     
\node (A1_1) at (1, 1) {$K/K'$};     
\node (A1_2) at (2, 1) {$\Z^M/<e_0,K'>$};    
\node (A1_3) at (3, 1) {$\Z^2$};     
\node (A1_4) at (4, 1) {$\Z^M/<e_0,K'>$};     
\node (A1_5) at (5, 1) {$M$};     
\node (A2_3) at (3.2, 1.4) {$\scriptstyle{\tau(e_1)=e_m}$};     
\node (A2_4) at (3.7, 1.4) {$\scriptstyle{\tau(e_2)=e_n}$};     
\node (A2_5) at (4.7, 1.4) {$\scriptstyle{p(e_l)=l}$};     

\path (A1_4) edge [->]node [auto,swap] {$\scriptstyle{p}$} (A1_5);     \path (A1_0) edge [->]node [auto] {$\scriptstyle{\sigma}$} (A1_1);     \path (A1_0) edge [->,bend left=35]node [auto] {$\scriptstyle{U}$} (A1_3);     \path (A1_1) edge [right hook->]node [auto] {$\scriptstyle{}$} (A1_2);     \path (A1_2) edge [->]node [auto] {$\scriptstyle{(\E,\delta)}$} (A1_3);     \path (A1_3) edge [->,bend right=45]node [auto] {$\scriptstyle{\pi}$} (A1_5);     \path (A1_3) edge [->]node [auto,swap] {$\scriptstyle{\tau}$} (A1_4);   \end{tikzpicture}   \]  We have $(\E,\delta)(v_{m,(z-1)m})=(z,-y)$ since $y<x=f(0)$ and
$(\E,\delta)(v_{n,(x-1)n})=(-w,x)$ since $w<z$. So $|\det U|=zx-yw=|M|$
and, since $\pi\circ U=0$, $U$ is an isomorphism onto $\Ker\pi$.
Moreover $\tau^{-1}=(\E,\delta)$ since $e_{l}\equiv\E_{l}e_{m}+\delta_{l}e_{n}\text{ mod }K'$.
It follows that $\sigma$ is an isomorphism and so (\ref{eq:base for the case m,n})
is a basis of $K$.

Consider now the second part of the statement. Clearly $\Lambda,\Delta\in<\E,\delta>_{\Q}$.
Therefore we have
\[
\Lambda=a\E+b\delta,\left\{ \begin{array}{c}
\Lambda(v_{m,(z-1)m})=1=az-yb\\
\Lambda(v_{n,(x-1)n})=0=xb-aw
\end{array}\then\left\{ \begin{array}{c}
a=x/|M|\\
b=w/|M|
\end{array}\right.\right.
\]
and the analogous relation for $\Delta$ follows in the same way.
Now note that, thanks to \ref{pro:Universal algebras generated by m,n with overline q fixed}
and \ref{pro:from an algebra to overline q}, we have that $\E=\E^{A}\comma\delta=\delta^{A}$
for an algebra $A$ as in \ref{not:starting from a cover, m.n generate M}
with $\overline{q}_{A}=\overline{q}$, $\lambda_{A}=0$ and sharing
the same invariants of $\overline{q}$. So we can apply \ref{lem:computation of E and delta}.
We want to prove that $\Lambda,\Delta\in\duale K_{+}$ so that they
form a smooth sequence by construction. Assume first that $\E_{a,b}>0$.
Clearly $\Lambda_{a,b},\Delta_{a,b}\geq0$ if $\delta_{a,b}\geq0$.
On the other hand if $\delta_{a,b}<0$ we know that $\E_{a,b}\geq z$
and $\delta_{a,b}\geq-y$ and so
\[
|M|\Lambda_{a,b}=x\E_{a,b}+w\delta_{a,b}\geq xz-yw=|M|\text{ and }|M|\Delta_{a,b}=y\E_{a,b}+z\delta_{a,b}\geq yz-zy=0
\]
The other cases follows in the same way. It remains to prove the last
relations. Since $-n=\hat{q}rm+(d_{\hat{q}}-1)n$, we have $\E_{n,-n}=\hat{q}r$
and $\delta_{n,-n}=d_{\hat{q}}$. Using the relation $zx-wy=|M|$
the values of $\Lambda_{n,-n}\comma\Delta_{n,-n}$ can be checked
by a direct computation. Similarly, considering the relations $-m=(\hat{q}r-1)m+d_{\hat{q}}n$
if $1<\overline{q}$, $-m=(r-1)m+(N-\alpha)n$ if $\overline{q}=1$
and $\alpha\neq0$, $-m=(r-1)m$ if $\alpha=0$, we can compute the
values of $\Lambda_{m,-m}$ and $\Delta_{m,-m}$.\end{proof}
\begin{prop}
\label{pro:general algebra for m,n}The multiplication of $A^{\overline{q}}$
(see \ref{def:universal algebra}) with respect to the basis $v_{l}=v_{m}^{\E_{l}}v_{n}^{\delta_{l}}$
is: $a^{\E^{\phi}}$ if $\overline{q}=N$, where $\phi\colon M\arrdi{\simeq}\Z/|M|\Z,$
$\phi(m)=1$; $b^{\E^{\eta}}$ if $\overline{q}r=1$, where $\eta\colon M\arrdi{\simeq}\Z/|M|\Z,$
$\phi(n)=1$; $a^{\Lambda}b^{\Delta}$ if $\overline{q}r\neq1\comma\overline{q}\neq N$,
where $\Lambda,\Delta$ are the rays defined in \ref{lem:lambda,delta for overlineq}.\end{prop}
\begin{proof}
In the proof of \ref{pro:general algebra for m,n 2} we have seen
that if $x=1$ $(\overline{q}=N)$, then $M=<m>$ and $A^{\overline{q}}=\Z[a,b][s]/(s^{|M|}-a)$,
while if $z=1$ $(\overline{q}r=1)$ then $M=<n>$ and $A^{\overline{q}}=\Z[a,b][t]/(t^{|M|}-b)$.
So we can assume $x,z>1$. Let $B$ the $\Di M$-cover over $\Z[a,b]$
given by multiplication $\psi=a^{\Lambda}b^{\Delta}$ and denote by
$\{\omega_{l}\}_{l\in M}$ a graded basis (inducing $\psi)$. By definition
of $\Lambda,\Delta$ we have $\omega_{l}=\omega_{m}^{\E_{l}}\omega_{n}^{\delta_{l}}$
for any $l\in M$ and $\psi_{m,(z-1)m}=a,\;\psi_{n,(x-1)n}=b$. Therefore
\[
\omega_{m}^{z}=\omega_{m}\omega_{(z-1)m}=a\omega_{zm}=a\omega_{yn}=a\omega_{n}^{y}\comma\omega_{n}^{x}=\omega_{n}\omega_{(x-1)n}=b\omega_{xn}=b\omega_{wm}=b\omega_{m}^{w}
\]
and, checking both cases $\overline{q}=1$ and $\overline{q}>1$,
$\omega_{m}^{\hat{q}r}\omega_{n}^{d_{\hat{q}}}=\omega_{-n}\omega_{n}=a^{\Lambda_{n,.n}}b^{\Delta_{n,.n}}=a^{\gamma}b$.
In particular we have an isomorphism $A^{\overline{q}}\arr B$ sending
$v_{m},v_{n}$ to $\omega_{m},\omega_{n}$.\end{proof}
\begin{notation}
From now on M will be any finite abelian group. If $\phi\colon M\arr M_{r,\alpha,N}$
is a surjective map, $r,\alpha,N$ satisfy the conditions of \ref{not:for alpha,N,r e M},
$\overline{q}\in\Omega_{N-\alpha,N}$ with $\overline{q}r\neq1\comma\overline{q}\neq N$
then we set $\Lambda^{r,\alpha,N,\overline{q},\phi}=\Lambda\circ\phi_{*}\comma\Delta^{r,\alpha,N,\overline{q},\phi}=\Delta\circ\phi_{*}$,
where $\Lambda,\Delta$ are the rays defined in \ref{lem:lambda,delta for overlineq}
with respect to $r,\alpha,N,\overline{q}$. If $\phi=\id$ we will
omit it.\end{notation}
\begin{defn}
Set
\[
\Sigma_{M}=\left\{ (r,\alpha,N,\overline{q},\phi)\left|\begin{array}{c}
0\leq\alpha<N\comma r>0\comma N>1\comma(r>1\text{ or }\alpha>1)\\
\overline{q}\in\Omega_{N-\alpha,N}\comma\overline{q}r\neq1\comma\overline{q}\alpha\not\equiv1\text{ mod }N\\
\overline{q}\neq N/(\alpha,N)\comma\phi\colon M\arr M_{r,\alpha,N}\text{ surjective}
\end{array}\right.\right\} 
\]
and $\Delta^{*}\colon\Sigma_{M}\arr\{\text{smooth integral extremal rays of }M\}$.\end{defn}
\begin{rem}
Since $e_{2},e_{1}$ generate $M_{r,\alpha,N}$, there exist unique
$\duale r,\duale{\alpha},\duale N$ with an isomorphism $\duale{(-)}\colon M_{r,\alpha,N}\arr M_{\duale r,\duale{\alpha},\duale N}$
sending $e_{2},e_{1}$ to $e_{1},e_{2}$. One can check that $\duale r=(\alpha,N)$,
$\duale N=rN/(\alpha,N)$ and $\duale{\alpha}=\tilde{q}r$, where
$\tilde{q}$ is the only integer $0\leq\tilde{q}<N/(\alpha,N)$ such
that $\tilde{q}\alpha\equiv(\alpha,N)\text{ mod }N$. 

If $A$ is an algebra as in \ref{not:starting from a cover, m.n generate M}
for $M_{r,\alpha,N}$, then, through $\duale{(-)}$, $A$ can be thought
as a $M_{\duale r,\duale{\alpha},\duale N}$-cover that we will denote
by $\duale A$. $\duale A$ is an algebra as in \ref{not:starting from a cover, m.n generate M}
with respect to $M_{\duale r,\duale{\alpha},\duale N}$ with $\overline{q}_{\duale A}=x_{A}/(\alpha,N)$,
$\lambda_{\duale A}=\mu_{A}$. We can define a bijection $\duale{(-)}\colon\Omega_{N-\alpha,N}-\{N/(N,\alpha)\}\arr\Omega_{\duale N-\duale{\alpha},\duale N}-\{\duale N/(\duale{\alpha},\duale N)\}$
in the following way. Given $\overline{q}$ take an algebra $A$ as
in \ref{not:starting from a cover, m.n generate M} for $M_{r,\alpha,N}$
with $\overline{q}_{A}=\overline{q}$ and $\lambda_{A}\neq0$, which
exists thanks to \ref{pro:Universal algebras generated by m,n with overline q fixed},
and set $\duale{\overline{q}}=\overline{q}_{\duale A}$. Taking into
account \ref{rem:lambda is 0 iff mu is 0} and \ref{pro:from an algebra to overline q},
$\duale{\overline{q}}=y_{\overline{q}}/(\alpha,N)$ since $x_{A}=y_{A}=y_{\overline{q}}$
and $\duale{(-)}$ is well defined and bijective since $\lambda_{\duale A}=\mu_{A}=\lambda_{A}^{-1}$.
Note that the condition $\overline{q}\alpha\equiv1\text{ mod }N$
is equivalent to $\duale r=1$ and $\duale{\overline{q}}=1$

Finally if $\phi\colon M\arr M_{r,\alpha,N}$ is a surjective morphism
then we set $\duale{\phi}=\duale{(-)}\circ\phi\colon M\arr M_{\duale r,\duale{\alpha},\duale N}$.
Note that in any case we have the relation $\duale{\duale{(-)}}=\id$.
In particular, since $\duale 1=\alpha/\duale r$, $\overline{q}=\duale{\alpha}/r$
is the dual of $1\in\Omega_{\duale N-\duale{\alpha},\duale N}$.\end{rem}
\begin{prop}
\label{pro:trivial ray for lambda,delta}Let $r,\alpha,N$ be as in
\ref{not:for alpha,N,r e M}, $\overline{q}\in\Omega_{N-\alpha,N}$
with $\overline{q}r\neq1\comma\overline{q}\neq N$ and $\phi\colon M\arr M_{r,\alpha,N}$
be a surjective map. Set $\chi=(r,\alpha,N,\overline{q},\phi)$. Then
\begin{enumerate}
\item $\overline{q}=N/(\alpha,N)$: $\Delta^{\chi}=\E^{\xi}$, $\xi\colon M\arrdi{\phi}M_{r,\alpha,N}\arr M_{r,\alpha,N}/<m>\simeq<n>\simeq\Z/(\alpha,N)\Z$;
$\overline{q}\alpha\equiv1\text{ mod }N$: $\Delta^{\chi}=\E^{\zeta}$,
$\zeta\colon M\arrdi{\phi}M_{r,\alpha,N}=<e_{1}>$;
\item $\overline{q}=1$: $\Lambda^{\chi}=\E^{\omega}$, $\omega\colon M\arrdi{\phi}M_{r,\alpha,N}\arr M_{r,\alpha,N}/<n>=<m>\simeq\Z/r\Z$;\newline$w_{\overline{q}}=1$:
$\Lambda^{\chi}=\E^{\theta},\theta\colon M\arrdi{\phi}M_{r,\alpha,N}=<e_{2}>$;
\item $\overline{q}>1$ and $w_{\overline{q}}\neq1$: $\Lambda^{\chi}=\Delta^{r,\alpha,N,\overline{q}-\hat{q},\phi}$.
\end{enumerate}
In particular in the first two cases we have $h_{\Lambda^{\chi}}=h_{\Delta^{\chi}}=1$.\end{prop}
\begin{proof}
We can assume $M=M_{r,\alpha,N}$ and $\phi=\id$. The algebra associated
to $0^{\Lambda^{\chi}}$, $0^{\Delta^{\chi}}$ are respectively $C_{\overline{q}}=k[s,t]/(s^{z},t^{x}-s^{w},s^{\hat{q}r}t^{d_{\hat{q}}}-0^{\gamma})$,
$B_{\overline{q}}=k[s,t]/(s^{z}-t^{y},t^{x},s^{\hat{q}r}t^{d_{\hat{q}}})$
by \ref{pro:general algebra for m,n}. 

$1)$ If $\overline{q}=N/(\alpha,N)$, then $z=o(m)$, $y=0$, $d_{\hat{q}}=(\alpha,N)$
and so $B_{\overline{q}}=k[s,t]/(s^{o(m)}-1,t^{(\alpha,N)})$, the
algebra associated to $0^{\E^{\xi}}$. If $\overline{q}\alpha\equiv1\text{ mod }N$
then $\duale r=(\alpha,N)=1$ and $\overline{q}=\duale{\alpha}/r$,
i.e. $\duale{\overline{q}}=1$. So $y=1$ and $B_{\overline{q}}\simeq k[s]/(s^{|M|})$,
the algebra associated to $0^{\E^{\gamma}}$. 

$2)$ If $\overline{q}=1$, then $z=r$, $\hat{q}=w=0$, $x=d_{\hat{q}}=N$
and so $C_{1}=k[s,t](t^{n}-1,s^{r})$, the algebra associated to $0^{\E^{\omega}}$.
If $w=1$ then $\overline{q}>1$ and so $C_{\overline{q}}=k[t]/(t^{|M|}),$
the algebra associated to $0^{\E^{\theta}}$.

$3)$ If $\overline{q}>1$ then $H_{C_{\overline{q}}}=0$ and so $C_{\overline{q}}$
is an algebra as in \ref{not:starting from a cover, m.n generate M}.
An easy computation shows that $z_{C_{\overline{q}}}=w>1$, so that
$\overline{q}_{C_{\overline{q}}}=\overline{q}-\hat{q}$ and $\lambda_{\overline{q}}=1$.
Therefore $\Lambda^{\chi}=\Delta^{r,\alpha,N,\overline{q}-\hat{q}}$
by \ref{pro:Universal algebras generated by m,n with overline q fixed}. \end{proof}
\begin{prop}
\label{pro:classification sm int ray htwo}$\duale{\Sigma}_{M}=\Sigma_{M}$
and we have a bijection 
\[
\Delta^{*}\colon\Sigma_{M}/\duale{(-)}\arr\{\text{smooth integral extremal rays }\E\text{ with }h_{\E}=2\}
\]
\end{prop}
\begin{proof}
$\duale{\Sigma}_{M}\subseteq\Sigma_{M}$ since $\overline{q}\alpha\not\equiv1\text{ mod }N$
is equivalent to $\duale{\overline{q}}\duale r\neq1$. Now let $\E$
be a smooth integral ray such that $h_{\E}=2$ and $A$ the associated
algebra over some field $k$. We can assume $H_{A/k}=H_{\E}=0$. $h_{\E}=2$
means that there exist $0\neq m,n\in M$, $m\neq n$ such that $ $$A$
is generated in degrees $m,n$. So $M=M_{r,\alpha,N}$ as in \ref{not:for alpha,N,r e M}
and $A$ is an algebra as in \ref{not:starting from a cover, m.n generate M}.
By \ref{pro:Universal algebras generated by m,n with overline q fixed}
and \ref{pro:trivial ray for lambda,delta} we can conclude that there
exist $\chi\in\Sigma_{M}$ such that $\E=\Delta^{\chi}$.

Now let $\chi=(r,\alpha,N,\overline{q},\phi)\in\Sigma_{M}$. We have
to prove that $h_{\Delta^{\chi}}=2$ and, since $M_{r,\alpha,N}\neq0$,
assume by contradiction that $h_{\Delta^{\chi}}=1$. We can assume
$M=M_{r,\alpha,N}$ and $\phi=\id$. Note that $h_{\Delta^{\chi}}=1$
means that the associated algebra $B$ is generated in degree $m$
or $n$. If $A$ is an algebra as in \ref{not:starting from a cover, m.n generate M},
then $A$ is generated in degree $n$ if and only if $z=1$, that
means $\overline{q}r=1$. So $B$ is generated in degree $m$, i.e.
$\duale B$ is generated in degree $e_{2}\in M_{\duale r,\duale{\alpha},\duale N}$,
which is equivalent to $1=z_{\duale B}=\duale{\overline{q}}\duale r=1$,
and, as we have seen, to $\overline{q}\alpha\equiv1\text{ mod N}$. 

Now let $\chi'=(r',\alpha',N',\overline{q}',\phi')\in\Sigma_{M}$
such that $\E=\Delta^{\chi}=\Delta^{\chi'}$. Again we can assume
$H_{\E}=0$ and take $B,B'$ the algebras associated respectively
to $\chi,\chi'$. By definition of $\Delta_{*}$, $\phi,\phi'$ are
isomorphisms. If $g=\phi'\circ\phi^{-1}\colon M_{r,\alpha,N}\arr M_{r',\alpha',N'}$
then we have a graded isomorphism $p\colon B\arr B'$ such that $p(B_{l})=B'_{g(l)}$.
Therefore $g(\{e_{1},e_{2}\})=\{e_{1},e_{2}\}$, i.e. $g=\id$ or
$g=\duale{(-)}$. It is now easy to show that $\chi'=\chi$ or $\chi'=\duale{\chi}$.\end{proof}
\begin{notation}
We set $\Phi_{M}=\{\phi\colon M\arr\Z/l\Z\st l>1\comma\phi\text{ surjective}\}$,
$\Theta_{M}^{2}=\{\E^{\phi}\}_{\phi\in\Phi_{M}}\cup\{(\Lambda^{\chi},\Delta^{\chi})\}_{\chi\in\overline{\Sigma}_{M}}$,
where $\overline{\Sigma}_{M}$ is the set of sequences $(r,\alpha,N,\overline{q},\phi)$
where $r,\alpha,N\in\N$ satisfy $0\leq\alpha<N,r>0,r>1\text{ or }\alpha>1$,
$\overline{q}\in\Omega_{N-\alpha,N}$ satisfy $\overline{q}r\neq1$,
$\overline{q}\neq N$ and $\phi\colon M\arr M_{r,\alpha,N}$ is a
surjective map. Finally set $\underline{\E}=(\E^{\phi},\Delta^{\chi})_{\phi\in\Phi_{M},\chi\in\Sigma_{M}/\duale{(-)}}$.\end{notation}
\begin{thm}
\label{thm:fundamental thm for hleqtwo}Let $M$ be a finite abelian
group. Then 
\[
\{h\leq2\}=(\bigcup_{\phi\in\Phi_{M}}\stZ_{M}^{\E^{\phi}})\bigcup(\bigcup_{(\Lambda,\Delta)\in\Theta_{M}^{2}}\stZ_{M}^{\Lambda,\Delta})
\]
In particular $\{h\leq2\}\subseteq\stZ_{M}^{\textup{sm}}$. Moreover
$\pi_{\underline{\E}}\colon\stF_{\underline{\E}}\arr\MCov$ induces
an equivalence of categories
\[
\left\{ (\underline{\shL},\underline{\shM},\underline{z},\lambda)\in\stF_{\underline{\E}}\left|\begin{array}{c}
V(z_{\E^{1}})\cap\cdots\cap V(z_{\E^{r}})\neq\emptyset\text{ iff}\\
r=1\text{ or }(r=2\text{ and }(\E^{1},\E^{2})\in\Theta_{M}^{2})
\end{array}\right.\right\} =\pi_{\underline{\E}}^{-1}(h\leq2)\arrdi{\simeq}\{h\leq2\}
\]
\end{thm}
\begin{proof}
The writing of $\{h\leq2\}$ follows from \ref{pro:Universal algebras generated by m,n with overline q fixed}
and \ref{pro:general algebra for m,n}. Taking into account \ref{pro:classification sm int ray htwo},
the last part instead follows from \ref{pro:piE for theta isomorphism}
taking $\Theta=\Theta_{M}^{2}$.
\end{proof}
In \cite{Maclagan2002} the authors prove that the toric Hilbert schemes
associated to a polynomial algebra in two variables are smooth and
irreducible. The same result is true more generally for multigraded
Hilbert schemes, as proved later in \cite{Maclagan2010}. Here we
obtain an alternative proof in the particular case of equivariant
Hilbert schemes:
\begin{cor}
\label{cor:MHilbmn smooth and irreducible}If $M$ is a finite abelian
group and $m,n\in M$ then $\MHilb^{m,n}$ is smooth and irreducible.\end{cor}
\begin{proof}
Taking into account the diagram in \ref{rem:relation Mhilbm MCovm}
it is enough to note that $\MCov^{m,n}\subseteq\{h\leq2\}\subseteq\stZ_{M}^{\textup{sm}}$.\end{proof}
\begin{prop}
\label{pro:when sigmaM is empty}$\Sigma_{M}=\emptyset$ if and only
if $M\simeq(\Z/2\Z)^{l}$ or $M\simeq(\Z/3\Z)^{l}$.\end{prop}
\begin{proof}
For the only if, note that if $\phi\colon M\arr\Z/l\Z$ with $l>3$
is surjective, then, taking $m=l-1\comma n=1\in\Z/l\Z$, we have $\Z/l\Z\simeq M_{1,l-1,l}$
and $(1,l-1,l,2,\phi)\in\Sigma_{M}$.

For the converse set $M=(\Z/p\Z)^{l}$, where $p=2,3$ and, by contradiction,
assume to have $(r,\alpha,N,\overline{q},\phi)\in\Sigma_{M}$. In
particular $\phi$ is a surjective map $M\arr M_{r,\alpha,N}$. If
$e_{1},e_{2}\in M_{r,\alpha,N}$ are $\F_{p}$-independent then $M_{r,\alpha,N}=<e_{1}>\times<e_{2}>$,
$\alpha=0$, $\Omega_{N-\alpha,N}=\{1\}$ and therefore $\overline{q}=1=N/(\alpha,N)$,
which implies that $\chi\notin\Sigma_{M}$. On the other hand, if
$M_{1,\alpha,p}\simeq\Z/p\Z$, the only extremal rays for $\Z/p\Z$
are $\E^{\id}$ and, if $p=3$, $\E^{-\id}$ since $K_{+\Z/p\Z}\simeq\N^{p-1}$
by \ref{pro:smooth DMCov}.\end{proof}
\begin{thm}
\label{thm:fundamental thm locally factoria hleqtwo}Let $M$ be a
finite abelian group and $X$ be a locally noetherian and locally
factorial scheme. Set 
\[
\catC_{X}^{2}=\left\{ (\underline{\shL},\underline{\shM},\underline{z},\delta)\in\stF_{\underline{\E}}(X)\left|\begin{array}{c}
\codim_{X}V(z_{i_{1}})\cap\cdots\cap V(z_{i_{s}})\geq2\\
\text{if }\nexists\underline{\delta}\in\Theta_{M}^{2}\text{ s.t. }\E^{i_{1}},\dots\E^{i_{s}}\subseteq\underline{\delta}
\end{array}\right.\right\} 
\]
and
\[
\catD_{X}^{2}=\{Y\arrdi fX\in\MCov(X)\st h_{f}(p)\leq2\ \forall p\in X\text{ with }\codim_{p}X\leq1\}
\]
Then $\pi_{\underline{\E}}$ induces an equivalence of categories
\[
\catD_{X}^{2}=\pi_{\underline{\E}}^{-1}(\catC_{X}^{2})\arrdi{\simeq}\catC_{X}^{2}
\]
\end{thm}
\begin{proof}
Apply \ref{thm:fundamental theorem for locally factorial schemes}
with $\Theta=\Theta_{M}^{2}$.\end{proof}
\begin{rem}
In general $\{h\leq3\}$ doesn't belong to the smooth locus on $\stZ_{M}$.
For example, if $M=\Z/4\Z$, $\MCov=\{h\leq3\}$ is integral but not
smooth by \ref{pro:smooth DMCov} and \ref{rem:MCov for M=00003DZfour}.
\end{rem}

\subsection{Normal crossing in codimension 1}

In this subsection we want describe, in the spirit of classification
\ref{thm:regular in codimension 1 covers}, covers of a locally noetherian
and locally factorial scheme with no isolated points and with $(\car X,|M|)=1$
whose total space is normal crossing in codimension $1$.
\begin{defn}
\label{def:normal crossing codimesion one}A scheme $X$ is normal
crossing in codimension $1$ if for any codimension $1$ point $p\in X$
there exists a local and etale map $\widehat{\odi{}}_{X,p}\arr R$,
where $R$ is $k[[x]]$ or $k[[s,t]]/(st)$ for some field $k$ and
$ $$\widehat{\odi{}}_{X,p}$ denote the completion of $\odi{X,p}$.\end{defn}
\begin{rem}
If $X$ is locally of finite type over a perfect field $k$, one can
show that the above condition is equivalent to having an open subset
$U\subseteq X$ such that $\codim_{X}X-U\geq2$ and there exists an
etale coverings $\{U_{i}\arr U\}$ with etale maps $U_{i}\arr\Spec k[x_{1},\dots,x_{n_{i}}]/(x_{1}\cdots x_{r_{i}})$
for any $i$. Anyway we will not use this property.\end{rem}
\begin{notation}
In this subsection we will consider a field $k$ and we will set $A=k[[s,t]]/(st)$.
Given an element $\xi\in\Aut_{k}k[[x]]$ we will write $\xi_{x}=\xi(x)$
so that, if $p\in k[[x]]$ then $\xi(p)(x)=p(\xi_{x})$. We will call
$I\in\Aut_{k}k[[s,t]]$ the unique map such that $I(s)=t\comma I(t)=s$.
Given $B\in k^{*}$ we will denote by $\underline{B}$ the automorphism
of $k[[x]]$ such that $\underline{B}_{x}=Bx$.

Finally, given $f\in k[[x_{1},\dots,x_{n}]]$ and $g\in k[x_{1},\dots,x_{n}]$
the notation $f=g+\cdots$ will mean $f\equiv g\text{ mod }(x_{1},\dots,x_{r})^{\deg g+1}$.
\end{notation}
The first problem to deal with is to describe the action on $A$ of
a finite group $M$ and check when $A$ is a $\Di M$-cover over $A^{M}$,
assuming to have the $|M|$-roots of unity in $k$. We start collecting
some general facts about $A$.
\begin{prop}
We have:
\begin{enumerate}
\item $A=k\oplus sk[[s]]\oplus tk[[t]]$
\item Given $f,g\in A-\{0\}$ then $fg=0$ if and only if $f\in sk[[s]],g\in tk[[t]]$
or vice versa.
\item Any automorphism in $\Aut_{k}A$ is of the form $(\xi,\eta)$ or $I(\xi,\eta)$
where $\xi,\eta\in\Aut_{k}k[[x]]$ and $(\xi,\eta)(f(s,t))=f(\xi_{s},\eta_{t})$.
\item If $\xi\in\Aut_{k}k[[x]]$ has finite order then $\xi=\underline{B}$
where $B$ is a root of unity in $k$. In particular if $(\xi,\eta)\in\Aut_{k}A$
has finite order then $\xi=\underline{B}\comma\eta=\underline{C}$
where $B\comma C$ are roots of unity in $k$.
\item Let $f\in k[[x]]-\{0\}$, $B,C$ roots of unity in $k$. Then $f(Bx)=Cf(x)$
if and only if $C=B^{r}$ for some $r>0$ and, if we choose the minimum
$ $$r$, $f\in x^{r}k[[x^{o(B)}]]$.
\end{enumerate}
\end{prop}
\begin{proof}
$1)$ is straightforward and $2)$ follows easily writing $f$ and
$g$ as in $1)$. For $3)$ note that if $\theta\in\Aut_{k}A$ then
$\theta(s)\theta(t)=0$ and apply $2)$. Finally $4)$ and $5)$ can
be shown looking at the coefficients of $\xi_{x}$ and of $f$.\end{proof}
\begin{lem}
\label{lem:NC excluding stupid subgroups}If $M<\Aut_{k}A$ is a finite
subgroup containing only automorphisms of the form $(\xi,\eta)$ then
$A^{M}\simeq A$.\end{lem}
\begin{proof}
It's easy to show that $A^{M}\simeq k[[s^{a},t^{b}]]/(s^{a}t^{b})\simeq A$
where $a=\text{lcm}\{i\st\exists(\underline{A},\underline{B})\in M\text{ s.t. }\ord A=i\}$
and $b=\text{lcm}\{i\st\exists(\underline{A},\underline{B})\in M\text{ s.t. }\ord B=i\}$.
\end{proof}
Since we are interested in covers of regular in codimension $1$ schemes
(and $A$ is clearly not regular) we can focus on subgroups $M<\Aut_{k}A$
containing some $I(\xi,\eta)$.
\begin{lem}
\label{lem:classification of actions for NC}Let $M<\Aut_{k}A$ be
a finite abelian group and assume that $(\car k,|M|)=1$ and that
there exists $I(\xi,\eta)\in M$. Then, up to equivariant automorphisms,
we have $M=<I(\id,\underline{B})>$ or, if $M$ is not cyclic, $M=<(\underline{C},\underline{C})>\times<I>$
where $\underline{B}\comma\underline{C}$ are roots of unity and $o(C)$
is even.\end{lem}
\begin{proof}
The existence of an element of the form $I(\xi,\eta)$ in $M$ implies
that $s$ and $t$ cannot be homogeneous in $m_{A}/m_{A}^{2}$, that
$2\mid|M|$ and therefore that $\car k\neq2$. 

Applying the exact functor $\Hom_{k}^{M}(m_{A}/m_{A}^{2},-)$, we
get that the surjection $m_{A}\arr m_{A}/m_{A}^{2}$ has a $k$-linear
and $M$-equivariant section. This means that there exists $x,y\in m_{A}$
such that $m_{A}=(x,y)$ and $M$ acts on $x\comma y$ with characters
$\chi,\zeta$. In this way we get an action of $M$ on $k[[X,Y]]$
and an equivariant surjective map $\phi\colon k[[X,Y]]\arr A$. Moreover
$\Ker\phi=(h)$, where $h=fg$ and $f\comma g\in k[[X,Y]]$ are such
that $\phi(f)=s$, $\phi(g)=t$. We can write $f=aX+bY+\cdots\comma g=cX+dY+\cdots$
with $ad-bc\neq0$. Since $ax+by=s$ in $m_{A}/m_{A}^{2}$ and $s$
is not homogeneous there, we have $a,b\neq0$. Similarly we get $c,d\neq0$.
In particular, up to normalize $f\comma g\comma x$ we can assume
$b=c=d=1$. Now $h=aX^{2}+(a+1)XY+Y^{2}+\cdots$ and applying Weierstrass
preparation theorem \cite[Theorem 9.2]{Lang2002}, there exists a
unique $\tilde{h}\in(h)$ such that $(\tilde{h})=(h)$ and $\tilde{h}=\psi_{0}(X)+\psi_{1}(X)Y+Y^{2}$.
The uniqueness of $\tilde{h}$ and the $M$-invariance of $(h)$ yield
the relations $m(\tilde{h})=\eta(m)^{2}\tilde{h}$,
\begin{equation}
m(\psi_{0})=\psi_{0}(\chi(m)X)=\eta(m)^{2}\psi_{0}\comma m(\psi_{1})=\psi_{1}(\chi(m)X)=\eta(m)\psi_{1}\label{eq:Relation for describe A NC}
\end{equation}
for any $m\in M$. Moreover $\tilde{h}=\mu h$ where $\mu\in k[[X,Y]]^{*}$
and, since the coefficient of $Y^{2}$ in both $h$ and $\tilde{h}$
is $1$, we also have $\mu(0)=1$. In particular $\psi_{0}=aX^{2}+\cdots\text{ and }\psi_{1}=(a+1)X+\cdots$
and so $(a+1)(\chi-\zeta)=0$ by \ref{eq:Relation for describe A NC}.
Since $s$ is not homogeneous in $m_{A}/m_{A}^{2}$, $\chi\neq\eta$
and $a=-1$. Since $\car k\neq2$ we can write $\tilde{h}=(Y+\psi_{1}/2)^{2}-(\psi_{1}^{2}/4-\psi_{0})=y'^{2}-z'$.
Note that $y',z'$ are homogeneous thanks to \ref{eq:Relation for describe A NC}.
Moreover, by Hensel's lemma, we can write $z'=X^{2}+\cdots=X^{2}q^{2}$
for an homogeneous $q\in k[[x]]$ with $q(0)=1$. So $x'=xq$ is homogeneous
and $\tilde{h}=y'^{2}-x'^{2}$. This means that we can assume $s=x-y$,
$t=x+y$. In particular $\chi^{2}=\eta^{2}$ and $M$ acts on $s,t$
as
\[
m(s)=\frac{\chi+\zeta}{2}(m)s+\frac{\chi-\zeta}{2}(m)t\qquad m(t)=\frac{\chi-\zeta}{2}(m)s+\frac{\chi+\zeta}{2}(m)t
\]
Consider the exact sequence
\begin{equation}
0\arr H\arr M\arrdi{\chi/\eta}\{-1,1\}\arr0\label{eq:exact sequence for NC}
\end{equation}
If $M$ is cyclic, say $M=<m>$, we have $\chi(m)=-\eta(m)$ and so
$m=I(\underline{B},\underline{B})$, where $B=(\chi(m)-\eta(m))/2$
is a root of unity. Up to normalize $s$ we can write $m=I(\id,\underline{B})$.

Now assume that $M$ is not cyclic. $H$ acts on $s$ and $t$ with
the character $\chi_{|H}=\zeta_{|H}$ and this yields an injective
homomorphism $\chi_{|H}\colon H\arr\{\text{roots of unity of }k\}$.
So $H=<(\underline{C},\underline{C})>$ for some root of unity $C$.
The extension \ref{eq:exact sequence for NC} corresponds to an element
of $\Ext^{1}(\Z/2\Z,H)\simeq H/2H$ that differs to the sequence $0\arr H\arr\Z/2o(C)\Z\arr\{-1,1\}\arr0$.
So $H/2H\simeq\Z/2\Z$, $o(C)$ is even and the sequence \ref{eq:exact sequence for NC}
splits. We can conclude that $M=<(\underline{C},\underline{C})>\times<m>$,
where $m=I(\underline{D},\underline{D})$ for some root of unity $D$
and $o(m)=2$. Normalizing $s$ we can write $m=I(\id,\underline{D})=I$.\end{proof}
\begin{prop}
\label{pro:NC complete description of invariants, algebras multiplication}Let
$M<\Aut_{k}A$ be a finite abelian group such that $(\car k,|M|)=1$
and that there exists $I(\xi,\eta)\in M$. Also assume that $k$ contains
the $|M|$-roots of unity. Then $A^{M}\simeq k[[z]]$, $A\in\MCov(A^{M})$
and only the following possibilities happen: there exists a row of
table \ref{tab:table for NC} such that $M\simeq H$ is generated
by $m,n$, $H\simeq M_{r,\alpha,N}$, $A\simeq B$ as $M$-covers,
where $\deg U=m\comma\deg V=n$ and $A$ over $A^{M}$ is given by
multiplication $z^{\E}$. Moreover all the rays of the form $\Delta^{*}$
in the table satisfy $h_{\Delta^{*}}=2$.

\begin{table}
\caption{\label{tab:table for NC}}
\begin{tabular}{|c|c|c|c|}
\hline 
$H$ & $m,n,r,\alpha,N,\overline{q}$ & $B$ & $\E$\tabularnewline
\hline 
\hline 
$\Z/2\Z$ & $1,1,1,1,2,1$ & $\frac{k[[z]][U]}{(U^{2}-z^{2})}$ & $2\E^{id}$\tabularnewline
\hline 
$(\Z/2\Z)^{2}$ & $(1,0),(0,1),2,0,2,1$ & $\frac{k[[z]][U,V]}{(U^{2}-z,V^{2}-z)}$ & $\E^{\pr_{1}}+\E^{\pr_{2}}$\tabularnewline
\hline 
$\begin{array}{c}
\Z/2l\Z\times\Z/2\Z\\
l>1
\end{array}$ & $(1,0),(1,1),2,2,2l,1$ & $\frac{k[[z]][U,V]}{(U^{2}-V^{2},V^{2l}-z)}$ & $\Delta^{2,2,2l,1}$\tabularnewline
\hline 
$\Z/4l\Z$ & $1,2l+1,1,2l+1,4l,2$ & $\frac{k[[z]][U,V]}{(U^{2}-V^{2},V^{2l+1}-zU,UV^{2l-1}-z)}$ & $\Delta^{1,2l+1,4l,2}$\tabularnewline
\hline 
$\begin{array}{c}
\Z/2l\Z\\
l>1\text{ odd}
\end{array}$ & $1,l+1,2,2,l,1$ & $\frac{k[[z]][U,V]}{(U^{2}-V^{2},V^{l}-z)}$ & $\Delta^{2,2,l,1}$\tabularnewline
\hline 
\end{tabular}

\end{table}
\end{prop}
\begin{proof}
We can reduce to the actions obtained in \ref{lem:classification of actions for NC}.
We first consider the cyclic case, i.e. $M=<I(\id,\underline{B})>\simeq\Z/2l\Z$
where $l=o(B)$. There exists $E$ such that $E^{2}=B$. Given $0\leq r<|M|=2l$,
we want to compute $A_{r}=\{a\in A\st I(\id,\underline{B})a=E^{r}a\}$.
$a=c+f(s)+g(t)\in A_{r}$ if and only if $a=0$ if $r>0$, $f(t)=E^{r}g(t)$
and $g(Bs)=E^{r}f(s)$. Moreover $f(t)=E^{-r}g(Bt)=E^{-2r}f(Bt)\then f(Bt)=B^{r}f(t)$.
If we denote by $\delta_{r}$ the only integer such that $0\leq\delta_{r}<l$
and $\delta_{r}\equiv r\text{ mod }l$, we have that, up to constants,
$A^{r}$ is given by elements of the form $E^{r}f(s)+f(t)$ for $f\in X^{\delta_{r}}k[[X^{l}]]$.
Call $\beta=s^{l}+t^{l}\in A_{0}=A^{M}$ and $v_{r}=E^{r}s^{\delta_{r}}+t^{\delta_{r}}$,
$v_{0}=1$. We claim that $A^{M}=A_{0}=k[[\beta]]$ and $v_{r}$ freely
generates $A_{r}$ as an $A_{0}$ module. The first equality holds
since $A_{0}$ is a domain and we have relations
\[
\sum_{n\geq1}a_{n}s^{nl}+\sum_{n\geq1}a_{n}t^{nl}=\sum_{n\geq1}a_{n}(s^{l}+t^{l})^{n}=\sum_{n\geq1}a_{n}\beta^{n}
\]
while the second claim come from the relation
\[
E^{r}s^{\delta_{r}}(c+h(s))+t^{\delta_{r}}(c+h(t))=(E^{r}s^{\delta_{r}}+t^{\delta_{r}})(c+h(s)+h(t))\text{ for }h\in X^{l}k[[X^{l}]]
\]
 and the fact that $v_{r}$ is not a zero divisor in $A$.

So $A\in\MCov(k[[\beta]])$ and it is generated by $v_{1}=Es+t$ and
$v_{l+1}=-Es+t$ and so in degrees $1$ and $l+1$. If $l=1$, so
that $M\simeq\Z/2\Z$, $B=1$, $E=-1$ and $v_{1}^{2}=\beta^{2}$.
This means that $A\simeq k[[\beta]][U]/(U^{2}-\beta^{2})$ and its
multiplication over $k[[\beta]]$ is given by $\beta^{2\E^{\id}}$.
This is the first row. Assume $l>1$ and set $m=1\comma n=l+1$. Note
that $0\neq m\neq n$ and that $M\simeq M_{r,\alpha,N}$ for some
$r\comma\alpha\comma N$ that we are going to compute.

$l$ odd. We have $r=\alpha=2$ and $N=l$ since $<l+1>=<2>\subseteq\Z/2l\Z$.
Consider $\overline{q}=1\in\Omega_{N,N-\alpha}$ and the associated
numbers are $z=r=2\comma y=\alpha=2\comma\hat{q}=0\comma d_{\hat{q}}=x=N=l\comma w=0$.
Since $v_{1}^{z}=v_{l+1}^{y}$ and $v_{l+1}^{l}=\beta$, we will have
$A\simeq_{k[[\beta]]}A_{\lambda,\mu}^{1}$ where $\lambda,\mu=1,\beta\in k[[\beta]]$
(see \ref{def:universal algebra}) and therefore the multiplication
is $\beta^{\Delta^{2,2,l,1}}$ by \ref{pro:general algebra for m,n}.
This is the fifth row.

$l$ even. We have $r=1\comma\alpha=l+1\comma N=2l$ since $<l+1>=\Z/2l\Z$.
Since $d_{1}=l-1\equiv-\alpha$ and $d_{2}=2l-2\equiv2(-\alpha)$
modulo $2l$ we can consider $\overline{q}=2\in\Omega_{N-\alpha,N}$.
The associated numbers are $z=y=2\comma\hat{q}=1\comma d_{\hat{q}}=l-1\comma x=N-(d_{\overline{q}}-d_{\hat{q}})=l+1\comma w=1\equiv xn=(l+1)^{2}\text{ mod }2l$.
Since $v_{1}^{z}=v_{l+1}^{y}$, $v_{l+1}^{x}=\beta v_{1}$ and $v_{1}^{\hat{q}r}v_{l+1}^{d_{\hat{q}}}=\beta$,
we will have $A\simeq_{k[[\beta]]}A_{\lambda,\mu}^{2}$ where $\lambda,\mu=1,\beta\in k[[\beta]]$
whose multiplication is $\beta^{\Delta_{1,l+1,2l,2}}$. This is the
fourth row.

Now consider the case $M=<(\underline{C},\underline{C})>\times<I>$
with $o(C)=l$ even. Set $\beta=s^{l}+t^{l}$, $v_{1,0}=s+t$ and
$v_{1,1}=-s+t$. Note that $v_{r,i}$ is homogeneous of degree $(r,i)$.
Set $m=(1,0)\comma n=(1,1)$. They are generators of $M$ and so $M\simeq M_{r,\alpha,N}$
for some $r,\alpha,N$. We have $N=o(n)=l$, $r>1$ since $<n>\neq M$
and so $r=2$ since $2m=2n$. If $l=2$ we get $\alpha=0$ and if
$l>2$ we get $\alpha=2$. Choose $\overline{q}=1$ so that the associated
numbers are $z=2\comma y=\alpha\comma\hat{q}=0\comma d_{\hat{q}}=x=N=l\comma w=0$.
As done above, it is easy to see that $A^{M}=k[[\beta]]$. We first
consider the case $l=2$. Since $v_{1,0}^{2}=\beta\comma v_{1,1}^{2}=\beta$
we get a surjection $A_{\beta,\beta}^{1}\arr A$ which is an isomorphism
by dimesion. From the writing of $A_{\beta,\beta}^{1}$ we can deduce
directly that the multiplication is $\beta^{\E^{\pr_{1}}+\E^{\pr_{2}}}$,
where $\pr_{i}\colon(\Z/2\Z)^{2}\arr\Z/2\Z$ are the two projections.
This is the second row.

Now assume $l>2$. Since $v_{1,0}^{2}=v_{1,1}^{2}$ and $v_{1,1}^{l}=\beta$
and arguing as above we get $A\simeq_{k[[\beta]]}A_{\lambda,\mu}^{1}$
where $\lambda,\mu=1,\beta\in k[[\beta]]$ and the multiplication
$\beta^{\Delta^{2,2,l,1}}$. This is the third row.

Finally the last sentence is clear by definition of $\Sigma_{M}$
and \ref{pro:classification sm int ray htwo}.\end{proof}
\begin{rem}
\label{rem:Y/X NC then X has a k}If $X$ is a locally noetherian
integral scheme and there exists a $\Di M$-cover $Y/X$ such that
$Y$ is normal crossing in codimension $1$, then $X$ is defined
over a field. Indeed if $\car\odi X(X)=p$ then $\F_{p}\subseteq\odi X(X)$.
Otherwise $\Z\subseteq\odi X(X)$ and we have to prove that any prime
number $q\in\Z$ is invertible. We can assume $X=\Spec R$, where
$R$ is a local noetherian domain. If $\dim R=0$ then $R$ is a field,
otherwise, since $\alt(q)\leq1$, we can assume $\dim R=1$ and $R$
complete. By definition of normal crossing in codimension $1$, if
$Y=\Spec S$ and $p\in Y$ is over $m_{R}$ we have a flat and local
map $R\arr S\arr S_{p}\arr B$, such that $B$ contains a field $k$.
$q$ is a non zero divisor in $R$ and therefore in $B$. In particular
$0\neq q\in k^{*}\subseteq B^{*}$ and $q\in R^{*}$.\end{rem}
\begin{thm}
\label{thm:NC in codimension one}Let $M$ be a finite abelian group,
$X$ be a locally noetherian and locally factorial scheme with no
isolated points and $(\car X,|M|)=1$. Define
\[
NC_{X}^{1}=\{Y/X\in\MCov(X)\st Y\text{ is normal crossing in codimension }1\}
\]
Then $NC_{X}^{1}\neq\emptyset$ if and only if each connected component
of $X$ is defined over a field. In this case define
\[
\underline{\E}=\left(\begin{array}{c}
\E^{\phi}\text{ for }\phi\colon M\arr\Z/l\Z\text{ surjective with }l\geq1,\\
\Delta^{2,2,l,1,\phi}\text{ for }\phi\colon M\arr M_{2,2,l}\text{ surjective with }l\geq3,\\
\Delta^{1,2l+1,4l,2,\phi}\text{ for }\phi\colon M\arr M_{1,2l+1,4l}\text{ surjective with }l\geq1
\end{array}\right)
\]
and $\catC_{NC,X}^{1}$ as the subcategory of $\stF_{\underline{\E}}(X)$
of objects $(\underline{\shL},\underline{\shM},\underline{z},\lambda)$
such that:
\begin{enumerate}
\item for all $\E\neq\delta\in\underline{\E}$, $\codim V(z_{\E})\cap V(z_{\delta})\geq2$
except the case where $\E=\E^{\phi}\comma\delta=\E^{\psi}$   $$ \begin{tikzpicture}[xscale=2.0,yscale=-0.5]     \node (A0_2) at (2, 0) {$\Z/2\Z$};     \node (A1_0) at (0, 1) {$M$};     \node (A1_1) at (1, 1) {$(\Z/2\Z)^2$};     \node (A2_2) at (2, 2) {$\Z/2\Z$};     \path (A1_0) edge [->,bend left=65]node [auto,swap] {$\scriptstyle{\psi}$} (A2_2);     \path (A1_0) edge [->>]node [auto] {$\scriptstyle{}$} (A1_1);     \path (A1_1) edge [->]node [auto,swap] {$\scriptstyle{\pr_2}$} (A2_2);     \path (A1_1) edge [->]node [auto] {$\scriptstyle{\pr_1}$} (A0_2);     \path (A1_0) edge [->,bend right=65]node [auto] {$\scriptstyle{\phi}$} (A0_2);   \end{tikzpicture}  $$
in which $v_{p}(z_{\E^{\phi}})=v_{p}(z_{\E^{\psi}})=1$ if $p\in Y^{(1)}\cap V(z_{\E^{\phi}})\cap V(z_{\E^{\psi}})$;
\item for all $\E\in\underline{\E}$ and $p\in Y^{(1)}$ $v_{p}(z_{\E})\leq2$
and $v_{p}(z_{\E})=2$ if and only if $\E=\E^{\phi}$ where $\phi\colon M\arr\Z/2\Z$
is surjective.
\end{enumerate}
Then we have an equivalence of categories
\[
\catC_{NC,X}^{1}=\pi_{\underline{\E}}^{-1}(NC_{X}^{1})\arrdi{\simeq}NC_{X}^{1}
\]
\end{thm}
\begin{proof}
The first claim comes from \ref{rem:Y/X NC then X has a k}. We will
make use of \ref{thm:fundamental thm locally factoria hleqtwo}. If
$Y/X\in NC_{Y}^{1}$ and $p\in Y^{(1)}$ we have $h_{Y/X}(p)\leq\dim_{k(p)}m_{p}/m_{p}^{2}\leq2$
since etale maps preserve tangent spaces and $\dim m_{A}/m_{A}^{2}\leq2$.
So $NC_{X}^{1}\subseteq\catD_{X}^{2}$.

Let $\underline{\delta}$ be the sequence of smooth integral rays
given in \ref{thm:fundamental thm locally factoria hleqtwo}. We know
that $\pi_{\underline{\delta}}^{-1}(NC_{X}^{1})\subseteq\catC_{X}^{2}$.
So we have only to prove that $\pi_{\underline{\delta}}^{-1}(NC_{X}^{1})\subseteq\stF_{\underline{\E}}(X)\subseteq\stF_{\underline{\delta}}(X)$
and that any element $Y\in NC_{X}^{1}$ locally satisfies the requests
of the theorem. So we can reduce to the case where $X=\Spec R$, where
$R$ is a complete DVR. Since $R$ contains a field, then $R\simeq k[[x]]$
. Let $\chi\in\pi_{\underline{\E}}^{-1}(\catD_{X}^{2})$ and $D$
the associated $M$-cover over $R$. Let $C$ be the maximal torsor
of $D/R$ and $H=H_{D/R}$. Note that, for any maximal ideal $q$
of $C$ we have $C_{q}\simeq k(q)[[x]]$ since $C/R$ is etale. Moreover
$\Spec D\in NC_{X}^{1}$ for $M$ if and only if for any maximal prime
$p$ of $D$ $\Spec D_{p}\in NC_{\Spec C_{q}}^{1}$ for $M/H$, where
$q=C\cap p$. In the same way $\chi\in\catC_{NC,X}^{1}$ for $M$
if and only if, for any maximal prime $q$ of C, $\chi_{|\Spec C_{q}}\in\catC_{NC,\Spec C_{q}}^{1}$
for $M/H$. We can therefore reduce to the case $H_{D/R}=0$. We can
also assume that $k$ contains the $|M|$-roots of unity.

First assume that $\Spec D\in NC_{Y}^{1}$. If $D$ is regular, the
conclusion comes from \ref{thm:regular in codimension 1 covers}.
So assume $D$ not regular and denote by $\mu\colon R=k[[x]]\arr D$
the associated map. We know that $D/m_{A}=k$. By Cohen's structure
theorem we can write $D=k[[y]]/I$ in such a way that $\mu_{|k}=\id_{k}$.
By definition, since $D$ is local and complete, there exists an etale
extension $D\arr B=L[[s,t]]/(st)$. Using the properties of complete
rings, $B/D$ is finite and so $B\simeq D\otimes_{k}L$. Up to change
the base $R$ with $R\otimes_{k}L$ we can assume that $D\simeq k[[s,t]]/(st)$.
$\mu_{|k}\colon k\arr D$ extends to a map $\nu\colon D\arr D$ sending
$s,t$ to itselves. This map is clearly surjective. Since $\Spec D$
contains $3$ points, $\nu$ induces a closed immersion $\Spec D\arr\Spec D$
which is a bijection. Since $D$ is reduced $\nu$ is an isomorphism.
This shows that we can write $D=A=k[[s,t]]/(st)$ in such a way that
$\mu_{|k}=\id_{k}$. So $\Di M\simeq\underline{M}$ acts as a subgroup
of $\Aut_{k}A$ such that $A^{M}\simeq k[[z]]$. In particular, by
\ref{lem:NC excluding stupid subgroups}, there exists $I(\xi,\eta)\in M$.
Up to equivariant isomorphisms the possibilities allowed are described
in \ref{pro:NC complete description of invariants, algebras multiplication}
and coincides with the ones of the statement. So $\chi\in\catC_{NC,X}^{1}$.

Now assume that $\chi\in\catC_{NC,X}^{1}$. By definition of $\pi_{\underline{\E}}$
the multiplication that defines $D$ over $R$ is something of the
form $\psi=\lambda z^{\E}$, where $\lambda$ is an $M$-torsor and
$\E$ is one of the ray of table \ref{tab:table for NC}. The case
$\E=\E^{\phi}$ comes from \ref{thm:regular in codimension 1 covers}.
Since, in our hypothesis, an $M$-torsor (in the fppf meaning) is
also an etale torsor, up to change the base $R$ by an etale neighborhood
(that maintains the form $k[[x]]$), we can assume $\lambda=1$. In
this case, thanks to \ref{lem:classification of actions for NC} and
\ref{pro:NC complete description of invariants, algebras multiplication},
we can conclude that $A\simeq k[[s,t]]/(st)$ as required.\end{proof}
\begin{cor}
Let $X$ be a locally noetherian and regular in codimension $1$ (normal)
scheme with no isolated points, $M$ be a finite abelian group with
$(\car X,|M|)=1$ and $|M|$ odd. If $Y/X$ is a $\Di M$-cover and
$Y$ is normal crossing in codimension $1$ then $Y$ is regular in
codimension $1$ (normal).\end{cor}
\begin{proof}
Since $Y/X$ has Cohen-Macaulay fibers it is enough to prove that
$Y$ is regular in codimension $1$ by Serre's criterion. So we can
assume $X=\Spec R$, where $R$ is a DVR, and apply \ref{thm:regular in codimension 1 covers}
just observing that $\widetilde{Reg}_{X}^{1}=\catC_{NC,X}^{1}$.\end{proof}
\begin{rem}
We keep notation from \ref{thm:NC in codimension one} and set $\underline{\delta}=(\E^{\eta},\eta\colon M\arr\Z/d\Z\text{ surjective },d>1)$.
We have that $\pi_{\underline{\delta}}^{-1}(NC_{X}^{1})=\catC_{NC,X}^{1}\cap\stF_{\underline{\delta}}$,
i.e. the covers $Y/X\in NC_{X}^{1}$ writable only with the rays in
$\underline{\delta}$, has the same writing of $\catC_{NC,X}^{1}$
but with object in $\stF_{\underline{\delta}}$. Therefore the multiplications
that yield a not smooth but with normal crossing in codimension $1$
covers are only $\E^{\phi}+\E^{\psi}$, where $\phi\comma\psi$ are
morphism as in $1)$, and $\E^{2\phi}$, where $\phi\colon M\arr\Z/2\Z$
is surjective. This result can also be found in \cite[Theorem 1.9]{Alexeev2011}.
In particular, if $M=(\Z/2\Z)^{r}$, where $\underline{\delta}=\underline{\E}$
thanks to \ref{pro:when sigmaM is empty}, these are the only possibilities.
\end{rem}

\nocite{Stillman2002,Maclagan2002,Laumon1999}

\bibliographystyle{beta}
\bibliography{biblio}

\end{document}